\documentclass[a4paper,reqno,11pt]{amsart}

\usepackage{amsmath, amsfonts, amssymb, amsthm, amscd}
\usepackage{graphicx}
\usepackage{psfrag}
\usepackage{perpage}
\usepackage{url}
\usepackage{color}
\usepackage{mathrsfs}
\usepackage{tikz}
\usepackage{mathabx}
\usetikzlibrary{arrows,decorations.pathmorphing,backgrounds,positioning,fit,petri} % for TIKZ
\usepackage{tikz,tikz-cd}\usepackage{caption,subcaption}

\usepackage{dsfont} % nice indicator function symbol (blackboard 1)

\usepackage[utf8]{inputenc}
\usepackage[T1]{fontenc}
\usepackage{microtype}

\usepackage[a4paper,scale={0.72,0.74},marginratio={1:1},footskip=7mm,headsep=10mm]{geometry}

\usepackage{hyperref}

\makeatletter
\def\@secnumfont{\bfseries\scshape}

\def\section{\@startsection{section}{1}%
  \z@{.7\linespacing\@plus\linespacing}{.5\linespacing}%
  {\normalfont\large\bfseries\scshape\centering}}

\def\subsection{\@startsection{subsection}{2}%
  \z@{.5\linespacing\@plus.7\linespacing}{-.5em}%
  {\normalfont\bfseries\scshape}}

\def\subsubsection{\@startsection{subsubsection}{3}%
  \z@{.5\linespacing\@plus.7\linespacing}{-.5em}%
  {\normalfont\scshape}}

\def\specialsection{\@startsection{section}{1}%
  \z@{\linespacing\@plus\linespacing}{.5\linespacing}%
  {\normalfont\centering\large\bfseries\scshape}}
\makeatother

\makeatletter

\renewenvironment{proof}[1][\proofname]{\par
\pushQED{\qed}%
\normalfont \topsep4\p@\@plus4\p@\relax
\trivlist
\item[\hskip\labelsep
\bfseries
#1\@addpunct{.}]\ignorespaces
}{%
\popQED\endtrivlist\@endpefalse
}
\makeatother

\makeatletter
\newcommand \Dotfill {\leavevmode \leaders \hb@xt@ 6pt{\hss .\hss }\hfill \kern \z@}
\makeatother

\makeatletter
\def\@tocline#1#2#3#4#5#6#7{\relax
  \ifnum #1>\c@tocdepth % then omit
  \else
    \par \addpenalty\@secpenalty\addvspace{#2}%
    \begingroup \hyphenpenalty\@M
    \@ifempty{#4}{%
      \@tempdima\csname r@tocindent\number#1\endcsname\relax
    }{%
      \@tempdima#4\relax
    }%
    \parindent\z@ \leftskip#3\relax \advance\leftskip\@tempdima\relax
    \rightskip\@pnumwidth plus4em \parfillskip-\@pnumwidth
    #5\leavevmode\hskip-\@tempdima
      \ifcase #1
       \or\or \hskip 1.65em \or \hskip 3.3em \else \hskip 4.95em \fi%
      #6\nobreak\relax
    \Dotfill
    \hbox to\@pnumwidth{\@tocpagenum{#7}}\par
    \nobreak
    \endgroup
  \fi}
\makeatother

\makeatletter
\def\l@section{\@tocline{1}{0pt}{1pc}{}{\scshape}}
\renewcommand{\tocsection}[3]{%
\indentlabel{\@ifnotempty{#2}{\ignorespaces#1 #2.\hskip 0.7em}}#3}
\def\l@subsection{\@tocline{2}{0pt}{1pc}{5pc}{}}

\def\l@subsubsection{\@tocline{3}{0pt}{1pc}{7pc}{}}

\makeatother

\numberwithin{equation}{section}

\newtheoremstyle{mytheorem}{.7\linespacing\@plus.3\linespacing}{.7\linespacing\@plus.3\linespacing}%
     {\itshape}%         Body font
     {}%         Indent amount (empty = no indent, \parindent = para indent)
     {\bfseries}% Thm head font (e.g. \bfseries, \scshape, \sffamily)
     {. }%        Punctuation after thm head
     {0.3ex}%     Space after thm head (\newline = linebreak)
     {\thmname{{\bfseries #1}}\thmnumber{ {\bfseries #2}}\thmnote{ (#3)}}  % Thm head spec

\theoremstyle{mytheorem}

\newtheorem{theorem}{Theorem}[section]
\newtheorem{lemma}[theorem]{Lemma}
\newtheorem{proposition}[theorem]{Proposition}

\newtheorem{remark}[theorem]{Remark}

\newcommand{\bbE}{{\ensuremath{\mathbb E}} }

\newcommand{\bbP}{{\ensuremath{\mathbb P}} }

\newcommand{\bbT}{{\ensuremath{\mathbb T}} }

\newcommand{\bbV}{{\ensuremath{\mathbb V}} }

\newcommand{\cD}{{\ensuremath{\mathcal D}} }
\newcommand{\cE}{{\ensuremath{\mathcal E}} }

\newcommand{\cH}{{\ensuremath{\mathcal H}} }

\newcommand{\cK}{{\ensuremath{\mathcal K}} }
\newcommand{\cL}{{\ensuremath{\mathcal L}} }

\newcommand{\gb}{\beta}
\newcommand{\gd}{\delta}

\newcommand{\gl}{\lambda}

\newcommand{\go}{\omega}

\renewcommand{\tilde}{\widetilde}          % wider `tilde'
\DeclareMathSymbol{\leqslant}{\mathalpha}{AMSa}{"36} % nicer `smaller or equal'
\DeclareMathSymbol{\geqslant}{\mathalpha}{AMSa}{"3E} % nicer `larger or equal'
\DeclareMathSymbol{\eset}{\mathalpha}{AMSb}{"3F}     % nicer `emptyset'
\newcommand{\sumtwo}[2]{\sum_{\substack{#1 \\ #2}}} % sum with 2 lines
\newcommand{\be}{\begin{equation}}
\newcommand{\ee}{\end{equation}}

\newcommand{\R}{\mathbb{R}}

\newcommand{\Z}{\mathbb{Z}}
\newcommand{\N}{\mathbb{N}}

\def\bs{\boldsymbol}

\newcommand{\PEfont}{\mathrm}

\newcommand{\p}{\ensuremath{\PEfont P}}

\newcommand{\e}{\ensuremath{\PEfont E}}

\newcommand{\E}{\e}
\renewcommand{\P}{\p}

\newcommand\bP{\ensuremath{\bs{\mathrm{P}}}}
\newcommand\bE{\ensuremath{\bs{\mathrm{E}}}}

\DeclareMathOperator{\bbvar}{\ensuremath{\mathbb{V}ar}}
\DeclareMathOperator{\bbcov}{\ensuremath{\mathbb{C}ov}}

\newcommand{\ind}{\mathds{1}}

\newcommand{\eps}{\varepsilon}
\renewcommand{\epsilon}{\varepsilon}
\renewcommand{\theta}{\vartheta}
\renewcommand{\rho}{\varrho}

\newenvironment{myenumerate}{%
\renewcommand{\theenumi}{\arabic{enumi}}%
\renewcommand{\labelenumi}{{\rm(\theenumi)}}%
\begin{list}{\labelenumi}
	{%
	\setlength{\itemsep}{0.4em}%
	\setlength{\topsep}{0.5em}%
	\setlength\leftmargin{2.45em}%
	\setlength\labelwidth{2.05em}%
	\setlength{\labelsep}{0.4em}%
	\usecounter{enumi}%
	}%
	}%
{\end{list}
}

{\end{list}
}

{\end{list}
}

\renewenvironment{enumerate}{
\begin{myenumerate}}%
{\end{myenumerate}}

\newenvironment{myitemize}{%
\begin{list}{$\bullet$}%
 	{%
	\setlength{\itemsep}{0.4em}%
	\setlength{\topsep}{0.5em}%
	\setlength\leftmargin{2.65em}%
	\setlength\labelwidth{2.65em}%
	\setlength{\labelsep}{0.4em}%
	}%
	}%
{\end{list}}

\renewenvironment{itemize}{
\begin{myitemize}}%
{\end{myitemize}}

\MakePerPage[2]{footnote} % restarts footnote counter at every new page

\date{\today}

\newcommand\dd{\mathrm{d}}

\newcommand\hbeta{{\hat{\beta}}}

\newcommand\bq{\bs{q}}
\newcommand\bxi{\bs{\xi}}

\title[The 2d KPZ equation in the subcritical regime]{The
two-dimensional KPZ equation\\ in the entire subcritical regime}

\author[F. Caravenna]{Francesco Caravenna}
\address{Dipartimento di Matematica e Applicazioni\\
 Universit\`a degli Studi di Milano-Bicocca\\
 via Cozzi 55, 20125 Milano, Italy}
\email{francesco.caravenna@unimib.it}

\author[R. Sun]{Rongfeng Sun}
\address{Department of Mathematics\\
National University of Singapore\\
10 Lower Kent Ridge Road, 119076 Singapore
}
\email{matsr@nus.edu.sg}

\author[N. Zygouras]{Nikos Zygouras}
\address{Department of Statistics\\
University of Warwick\\
Coventry CV4 7AL, UK}
\email{N.Zygouras@warwick.ac.uk}

\begin{document}

\begin{abstract}
We consider the KPZ equation in space dimension $2$
driven by space-time white noise.
We showed in previous work that 
if the noise is mollified in space on scale $\epsilon$ and
its strength is scaled as $\hat\beta / \sqrt{|\log \epsilon|}$,
then a transition occurs
with explicit critical point $\hat\beta_c = \sqrt{2\pi}$.
Recently Chatterjee and Dunlap
showed that the solution
admits subsequential scaling limits as $\eps\downarrow 0$, 
for sufficiently small $\hat\beta$.
We prove here that the limit exists in the entire subcritical regime
$\hat\beta \in (0, \hat\beta_c)$ and we identify it as
the solution of an additive Stochastic Heat Equation, establishing
so-called Edwards-Wilkinson fluctuations.
The same result holds 
for the directed polymer model in random environment
in space dimension~$2$.
\end{abstract}


\maketitle

%\tableofcontents

\section{Introduction and main results}

We present first our results for the two-dimensional KPZ equation, and then similar results 
for its discrete analogue, the directed polymer model in random environment in dimension $2+1$. 
We close the introduction with an outline of the rest of the paper.

\subsection{KPZ in two dimensions}
The KPZ equation is a stochastic PDE, formally written as
\begin{equation}\label{eq:KPZ}
	\partial_t h(t, x) = \frac{1}{2} \Delta h(t,x) + \frac{1}{2}|\nabla h(t,x)|^2 
	+ \beta \, \xi(t,x), \qquad t\geq 0, x\in \R^d,
\end{equation}
where $\xi(t,x)$ is the space-time white noise, and $\beta>0$ governs the strength of the noise. It 
was introduced by Kardar, Parisi and Zhang~\cite{KPZ86} as a model for random interface growth, and 
has since been an extremely active area of research for both physicists and mathematicians.
The equation is ill-posed due to
the singular term $|\nabla h|^2$ which is undefined, because $\nabla h$ is
expected to be a distribution (generalized function).

In spatial dimension $d=1$,
these difficulties can be bypassed by considering the so-called
\emph{Cole-Hopf solution} $h:=\log u$, where $u$ is
defined as the solution of the multiplicative Stochastic Heat Equation
$\partial_t u = \frac{1}{2} \Delta u + \beta \xi u$,
which is linear and well-posed in dimension $d=1$, by classical Ito theory.
On large space-time scales,
the Cole-Hopf solution exhibits the same fluctuations as 
many exactly solvable one-dimensional interface growth models, 
all belonging to the so-called KPZ universality class. See the surveys \cite{C12, QS15} 
for reviews on the extensive literature. Few results are known in higher dimensions
(see below).

Along a different line, intense research has been carried out in recent years to make sense of the solutions 
of the KPZ equation and other singular stochastic PDEs.
A robust theory was lacking until the seminal work by Hairer \cite{H13} and his subsequent theory 
of regularity structures \cite{H14}. Since then, a few alternative approaches have been developed, 
including the theory of paracontrolled distributions by Gubinelli, Imkeller, and Perkowski~\cite{GIP15}, 
the theory of energy solutions by Gon\c calves and Jara~\cite{GJ14}, and the renormalization approach 
by Kupiainen~\cite{K16}. 
All these approaches are only applicable to KPZ in
space dimension $d=1$, 
where the equation is so-called {\em subcritical}, in the sense that the non-linearity 
vanishes in the small scale 
limit with a scaling that preserves the linear and the noise terms in the equation. In the language 
of renormalization groups, the KPZ equation
in $d=1$ is {\em super-renormalizable} (see e.g.~\cite{K16}), while 
regarded as a disordered system, it would be called {\em disorder relevant} 
(see e.g.~\cite{H74}, \cite{G10}, and \cite{CSZ17a, CSZ17b}). 

\smallskip

In this paper we focus on $d=2$,
which for KPZ is the {\em critical} dimension
(the {\em renormalizable} or {\em disorder marginal} case).
To define a solution to \eqref{eq:KPZ}, we follow the standard approach and 
consider a spatially mollified version $\xi^\eps:=j_\eps * \xi$ of the noise, 
where $\eps>0$, $j\in C_c(\R^2)$ is a probability density on $\R^2$ with $j(x)=j(-x)$, 
and $j_\eps(x) := \eps^{-2} j(x/\eps)$. 
The key question is whether it is possible to replace 
$\beta\, \xi$ in \eqref{eq:KPZ} by $\beta_\eps\, \xi^\eps-C_\eps$,
for suitable constants $\beta_\eps, C_\eps$, such that the corresponding solution $h^\eps$ 
converges to a non-trivial limit as $\eps\downarrow 0$.

It turns out that in space dimension $d=2$
the right way to tune the noise strength is
\be \label{eq:beps}
	\beta_\eps := \hbeta \sqrt{\frac{2\pi}{\log \eps^{-1}}} \,, \qquad  \mbox{for some } 
	\hbeta \in (0,\infty) \,,
\ee
and to consider the following \emph{mollified KPZ equation}
(with $\Vert j\Vert_2^2 :=\int_{\R^2} j(x)^2 {\rm d}x$):
\begin{equation}\label{eq:mollifiedKPZ}
	\partial_t h^{\eps} = \frac{1}{2}\Delta h^{\eps} + \frac{1}{2} |\nabla h^{\eps}|^2
	+ \beta_\eps \,  \xi^\eps - C_\eps, \qquad 
	\text{where } \ C_\eps:=\beta_\eps^2 \, \eps^{-2} \, \Vert j\Vert_2^2 \,.
\end{equation}
For simplicity, we take $h^{\eps}(0, \cdot)\equiv 0$ as initial datum.
If we define
\begin{equation}\label{eq:CH}
	u^\epsilon(t,x) := e^{h^\epsilon(t,x)} \,,
\end{equation}
then, by Ito's formula, $u^\epsilon$ solves the \emph{mollified multiplicative 
Stochastic Heat Equation} (SHE):
\begin{equation} \label{eq:ueps0}
	\partial_t u^\eps
	= \frac{1}{2} \Delta u^\eps + \beta_\eps \, u^\eps \, \xi^\eps \,,
	\qquad u^\eps(0, \cdot)\equiv 1 \,.
\end{equation}

\smallskip

In \cite{CSZ17b} we investigated the \emph{finite-dimensional distributions}
of the mollified KPZ solution $h^\epsilon$ as $\epsilon \downarrow 0$.
In particular, we discovered in \cite[Section 2.3]{CSZ17b} that there 
is a transition in the {\em one-point distribution} 
as $\hat\beta$ varies, with critical value $\hat\beta_c := 1$: For any $t>0$,
\begin{equation}\label{eq:hlim}
	h^\eps(t, x) \, \xrightarrow[\eps\downarrow 0]{d} \,
	\begin{cases}
	\sigma_{\hat\beta} Z -\tfrac{1}{2}\sigma_{\hat\beta}^2 & \text{ if } \hat\beta < 1  \\
	-\infty & \text{ if } \hat \beta\geq 1
	\end{cases}
	\qquad \text{with} \quad  \sigma^2_{\hat\beta}:= \log \tfrac{1}{1-\hat\beta^2} \,,
	\quad Z  \sim N(0,1) \,.
\end{equation}
(Note that the limiting distribution does not depend on $t>0$.)
This can be viewed as a {\em weak disorder to stronger disorder transition},
where we borrow terminology from the directed polymer model 
(see Section~\ref{S:DP}).
It was also shown in \cite{CSZ17b} that in the subcritical regime $\hat\beta < \hat\beta_c := 1$
the $k$-point distribution of $h^\epsilon$ asymptotically factorizes:
for any finite set of distinct points $(x_i)_{1\leq i\leq k}$,
the random variables $(h^\eps(t, x_i))_{1\leq i\leq k}$ converge
as $\epsilon \downarrow 0$ to \emph{independent}
Gaussians.

\smallskip

It is natural to investigate the fluctuations of $h^\eps$, regarded as a random field,
as $\epsilon \downarrow 0$.
This is what Chatterjee and Dunlap recently addressed
in \cite{CD18}. They actually considered a variant of the mollified KPZ
equation \eqref{eq:mollifiedKPZ},
where $\beta_\eps$ is placed in front of the non-linearity
instead of  the noise, namely,
\begin{equation}\label{eq:CDKPZ}
	\partial_t \tilde h^{\eps} = \frac{1}{2}\Delta \tilde h^{\eps} + \frac{1}{2}  \beta_\eps
	|\nabla \tilde h^{\eps}|^2
	+ \xi^\eps \,.
\end{equation}
However, there is a simple relation between $\tilde h^\eps$ in \eqref{eq:CDKPZ}
and $h^\eps$ in \eqref{eq:mollifiedKPZ} (see Appendix~\ref{S:scaling}):
\begin{equation}\label{eq:KPZeq}
	\tilde h^\eps(t, x) - \bbE[\tilde h^\eps(t, x)] = \frac{1}{\beta_\eps}
	\big( h^\eps(t, x) - \bbE[h^\eps(t,x)] \big) \,,
\end{equation}
therefore working with $\tilde h^\eps$ or $h^\eps$ is equivalent.

The main result in \cite{CD18} is that for any fixed $t>0$, 
when $\hat\beta$ is sufficiently small, the centered solution 
$\tilde h^\eps(t, \cdot) - \bbE[\tilde h^\eps(t, \cdot)]$,
viewed as a random distributions on $\R^2$,
admits non-trivial weak subsequential limits as $\epsilon \downarrow 0$
(in a negative H\"older space). 
As a matter of fact, \cite{CD18} considered the
KPZ equation \eqref{eq:CDKPZ} on the 
two-dimensional torus ${\bf T}^2$, for technical reasons,
but it is reasonable to believe that their results should also hold on $\R^2$.

\smallskip

The perturbative approach followed by Chatterjee and Dunlap~\cite{CD18} is limited to $\hat\beta$ 
sufficiently small, and it does not prove the existence of a unique limiting random field.
Our main result shows that such a limit indeed exists,
in the \emph{entire subcritical regime} $\hat\beta \in (0,1)$,
and identifies it as the solution of an \emph{additive} SHE
with a non-trivial noise strength (that depends explicitly on $\hat\beta$).
This is commonly called Edwards-Wilkinson fluctuations~\cite{EW82}.

\begin{theorem}[Edwards-Wilkinson fluctuations for 2-dimensional KPZ]\label{T:KPZ}
Let $h^\eps$ be the solution of the mollified KPZ equation \eqref{eq:mollifiedKPZ},
with $\beta_\eps$ as in \eqref{eq:beps}
and $\hat\beta\in (0,1)$. Denote
\be
	\label{eq:mainresult}
	\mathfrak{h}^\eps(t,x) := \frac{h^\eps(t, x) - \bbE[h^\eps(t,x)]}{\beta_\epsilon}
	= \frac{\sqrt{\log \eps^{-1}}}{\sqrt{2\pi} \, \hat\beta}
	\big( h^\eps(t, x) - \bbE[h^\eps(t,x)] \big) \,,
\ee
where the centering satisfies $\bbE[h^\eps(t,x)] = -\frac{1}{2}
\sigma_{\hat\beta}^2 + o(1)$ as $\epsilon \downarrow 0$, see \eqref{eq:hlim}.

For any $t > 0$ and $\phi \in C_c(\R^2)$, the following convergence in law holds:
\begin{equation}\label{eq:mainconv}
	\langle \mathfrak{h}^\eps(t,\cdot), \phi(\cdot)\rangle
	=  \int_{\R^2} \mathfrak{h}^\eps(t,x) \phi(x) {\rm d}x
	\  \xrightarrow[\eps \downarrow 0]{d} \ \langle
	v^{(c_{\hat\beta})}(t, 	\cdot), \phi(\cdot) \rangle,
\end{equation}
where $v^{(c)}(t,x)$ is the solution of the two-dimensional additive Stochastic Heat Equation
\begin{equation}\label{eq:ASHE}
\left\{
\begin{aligned}
	\partial_t v^{(c)}(t,x) & = \frac{1}{2} \Delta v^{(c)} (t,x) + c \, \xi(t,x) \\
	v^{(c)} (0,x) & \equiv 0
\end{aligned}
\right., \qquad \text{where } \
	c := c_{\hat\beta} := \sqrt{ \tfrac{1}{1-\hat\beta^2}} \,.
%	\sqrt{ \tfrac{\hat\beta^2}{1-\hat\beta^2}} \,.
\end{equation}
\end{theorem}

\smallskip

\begin{remark}
For the version \eqref{eq:CDKPZ} of KPZ, Chatterjee and Dunlap showed in \cite{CD18}
that any subsequential limit of $\tilde h^\epsilon - \bbE[\tilde h^\epsilon]$ 
as $\epsilon \downarrow 0$
does not coincide with the solution of the additive SHE obtained by simply
dropping the non-linearity $\beta_\epsilon \, |\nabla \tilde h^{\eps}|^2$ in \eqref{eq:CDKPZ}.
Here we show that the limit of 
$\tilde h^\epsilon - \bbE[\tilde h^\epsilon]$ actually coincides with the solution of
the additive SHE with a \emph{strictly larger noise strength $c = c_{\hat\beta} > 1$}.
In other words, the non-linearity 
in \eqref{eq:CDKPZ} produces an independent non-zero noise term in the limit,
even though its strength $\beta_\epsilon \to 0$.

Our proof of Theorem~\ref{T:KPZ} is based
on an analogous fluctuation result we proved in \cite{CSZ17b} for the solution of the SHE 
\eqref{eq:ueps0}. The independent noise can be seen
to arise from the second and higher order chaos expansions
of the solution, supported on microscopic scales.
\end{remark}

\begin{remark}
We can view $\mathfrak{h}^\eps(t,\cdot)$
as a \emph{random distribution on $\R^2$}, i.e.\ a random
element of the space of distributions  $\cD'$,
the dual space of $\cD = C^\infty_c(\R^2)$.
%\footnote{The space $\cD = C^\infty_c(\R^2)$ with its standard topology
%is non metrizable.
%Its topological dual $\cD'$, equipped with the weak-* topology, is also non metrizable,
%so the usual results about weak convergence on Polish spaces cannot be applied directly.
%The Borel $\sigma$-algebra on $\cD'$ coincides with the ``cylindrical''
%$\sigma$-algebra generated by the maps $\langle \,\cdot\,, \phi \rangle$
%for $\phi \in \cD$, see \cite[Prop.~III.2.2]{Fernique}.}
Our results show that \emph{$\mathfrak{h}^\eps(t,\cdot)$ converges
in law to $v^{(c_{\hat\beta})}(t, \cdot)$} as random distributions.
This is because convergence in law on $\cD'$
is equivalent to the pointwise convergence
of the characteristic functional \cite[Th.~III.6.5]{Fernique}
(see also \cite[Cor.~2.4]{BDW} for an analogue for tempered distributions):
\begin{equation*}
	\forall \phi \in \cD = C^\infty_c(\R^2): \qquad
	\bbE\big[e^{i \langle \mathfrak{h}^\eps(t,\cdot), \phi(\cdot)\rangle} \big]
	\xrightarrow[\epsilon\downarrow 0]{}
	\bbE\big[e^{i \langle v^{(c_{\hat\beta})}(t, \cdot), \phi(\cdot)\rangle} \big]
\end{equation*}
and this clearly follows by \eqref{eq:mainconv}.
\end{remark}

\begin{remark}
For simplicity, we only formulated the convergence of $\mathfrak{h}^\eps(t,\cdot)$ 
to $v^{(c_{\hat\beta})}(t, \cdot)$ as a random distribution in space for each fixed $t$. 
However, our proof can be easily adapted to prove the convergence 
of $\mathfrak{h}^\eps(\cdot,\cdot)$ to $v^{(c_{\hat\beta})}(\cdot, \cdot)$ 
as a random distribution in space and time.
\end{remark}

\begin{remark} \label{rem:EW}
The solution $v^{(c)}(t,\cdot)$
of the additive SHE \eqref{eq:ASHE}, also known as the Edwards-Wilkinson equation \cite{EW82}, 
is the random distribution on $\R^2$ formally given by
\begin{equation}\label{eq:g}
	v^{(c)}(t,x) = c \, \int_0^t \int_{\R^2} g_{t-s}(x-z) \, \xi(s, z) \, \dd s \dd z \,, \qquad 
	\text{with } \quad g_t(x) = \frac{1}{2\pi \, t} \, e^{-\frac{|x|^2}{2t}} \,.
\end{equation}
For any $\phi \in C_c(\R^2)$, we have that
$\langle v^{(c)}(t,\cdot), \phi \rangle :=\int_{\R^2} v^{(c)} (t, x) \, \phi(x) \, \dd x$
is a Gaussian random variable with zero mean and variance
\begin{equation} \label{eq:sigma2phi}
	\bbvar\big[ \langle v^{(c)}(t,\cdot), \phi \rangle \big] = c^2 \, \sigma_{\phi}^2 \,, \qquad
	\sigma_{\phi}^2 := \langle \phi, K_t \, \phi \rangle
	= \int_{(\R^2)^2} \phi(x) \, K_t(x,y) \, \phi(y) \, \dd x \, \dd y \,,
\end{equation}
where the covariance kernel is given by
\begin{equation} \label{eq:K}
	K_t(x,y)
	:= \int_0^t  \frac{1}{4\pi u} e^{-\frac{|x-y|^2}{4u}} \, \dd u
	= \frac{1}{4\pi} \, \int_{\frac{|x-y|^2}{4t}}^\infty
	\frac{e^{-z} }{z} \, \dd z  \,.
\end{equation}
%Note that by \cite[Entry 8.367 (12), page 906]{GR},
%\begin{equation*}
%	K_t(x,y)
%	= \tfrac{1}{4\pi} \, \Big( \log \tfrac{4t}{|x-y|^2}  -
%	\gamma + o(1) \Big) \qquad \text{as} \quad
%	\tfrac{|x-y|^2}{t} \to 0 \,,
%\end{equation*}
%where $\gamma := -\Gamma'(1) = -\int_0^\infty (\log z) \, e^{-z} \, \dd z \simeq 0.577$
%is the Euler-Mascheroni constant.
\end{remark}

In \cite{CSZ17b} we also proved Edwards-Wilkinson fluctuations 
for the solution $u^\eps$ of the 2-dimensional multiplicative SHE \eqref{eq:ueps0}.
More precisely, if similarly to \eqref{eq:mainresult} we set
\be
	\label{eq:mainresult2}
	\mathfrak{u}^\eps(t,x) := 
	\tfrac{ 1 }{\beta_\epsilon} \big( u^\eps(t, x) - \bbE[u^\eps(t,x)] \big)
	= \tfrac{\sqrt{\log \eps^{-1}}}{\sqrt{2\pi} \, \hat\beta}
	\big( u^\eps(t, x) - 1 \big) \,,
\ee
then as $\epsilon \downarrow 0$ we have the convergence in law
$\langle \mathfrak{u}^\eps(t,\cdot), \phi(\cdot)\rangle \to \langle
v^{(c_{\hat\beta})}(t, \cdot), \phi(\cdot) \rangle$ as in \eqref{eq:mainconv}
in the entire subcritical regime $\hat\beta \in (0,1)$,
see \cite[Theorem 2.17]{CSZ17b}
(which is formulated for space-time fluctuations,
but its proof is easily adapted to space fluctuations).

Since $u^{\epsilon}(t,x) = \exp(h^\epsilon(t,x))$, it is tempting to 
relate \eqref{eq:mainresult2} and \eqref{eq:mainresult} via Taylor expansion.
This is non obvious, because the one-point distributions of
$h^\epsilon(t,x)$ do not vanish as $\epsilon \downarrow 0$,
see \eqref{eq:hlim}, so we cannot approximate $h^\epsilon(t,x) \approx u^{\epsilon}(t,x) - 1$.
We will show in Section~\ref{sec:methods} that the approximation of 
$h(t,x)$ is highly non trivial, and the main contribution
actually comes from specific parts of the expansion of $u(t,x)$ 
\emph{which are negligible relative to $u(t,x)$}. 

\smallskip

For future work, the goal will be to understand the scaling limit of the KPZ
solution $h^\epsilon(t,x)$ at or 
above the critical point $\hat\beta_c=1$. To our best knowledge, this remains a mystery also for 
physicists (even the weak to strong disorder transition \eqref{eq:hlim} discovered in 
\cite{CSZ17b} seems not to have been noted previously in the physics literature). 
Also the scaling limit of the SHE solution $u^\epsilon(t,x)$ at or above the critical point
is not completely known, even though we recently made some progress at 
the critical point~\cite{CSZ18}, improving the study initiated in \cite{BC98}
(where the regime \eqref{eq:beps},
with $\hat\beta$ close to $1$, was first studied).

\smallskip

We conclude this subsection with an overview of related results.
In space dimension $d=1$, 
the Cole-Hopf solution $h(t,x) := \log u(t,x)$ of the KPZ equation \eqref{eq:KPZ} is well-defined
as a random function, for any $\beta \in (0,\infty)$, and there is no phase transition
in the one-point distribution as $\beta$ varies.
Edwards-Wilkinson fluctuations for $h(t,x)$ and $u(t,x)$ are
easily established as $\beta \downarrow 0$, combining Wiener chaos and Taylor expansion
(because $u(t,x) \to 1$).

In space dimensions $d \ge 3$,
the right way to scale the disorder
strength is $\beta_\eps = \hat\beta \, \eps^{\frac{d-2}{2}}$.
It was shown in \cite{MSZ16, CCM18} that the mollified SHE solution
$u^\epsilon(t,x)$ of \eqref{eq:ueps0} undergoes a weak to stronger
disorder transition, similar to the directed polymer model~\cite{CSY04}:
there is a critical value
$\hat\beta_c \in (0,\infty)$ such that $u^\epsilon(t,x)$ converges in law 
as $\epsilon \downarrow 0$ to a strictly positive limit
when $\hat\beta < \hat\beta_c$, while it converges to zero if 
$\hat\beta > \hat\beta_c$. The KPZ solution $h^\epsilon(t,x) = \log u^\epsilon(t,x)$ is thus
qualitatively similar to the 2-dimensional case \eqref{eq:hlim}:
$h^\epsilon(t,x)$ converges in law to a finite limit for $\hat\beta < \hat\beta_c$,
while it converges to $-\infty$ for $\hat\beta > \hat\beta_c$.
The value of $\hat\beta_c$ is unknown.

Edwards-Wilkinson fluctuations for the KPZ solution $h^\epsilon(t,x)$ in dimension $d \ge 3$
have been established recently by Magnen and Unterberger \cite{MU18}, 
assuming that the noise strength $\hat\beta$ is sufficiently small.
The corresponding result for the SHE solution $u^\epsilon(t,x)$
was proved in \cite{GRZ18, CCM18}.
The approaches in these papers do not allow to cover the entire subcritical 
regime, as we do in dimension $2$.

We should also mention that in space dimension $d=2$,
Edwards-Wilkinson fluctuations are believed to hold 
 (and verified in some cases, see e.g.~\cite{T17})  
also for models in the \emph{anisotropic} KPZ class, where anysotropy means that
the term $|\nabla h|^2$ in the KPZ equation \eqref{eq:KPZ} is replaced
by $\langle \nabla h, A \nabla h\rangle$ for some matrix $A$ with ${\rm det}(A)\leq 0$.

\smallskip

Shortly after we posted our paper, Dunlap et al.~\cite{DGRZ18} gave 
an alternative proof (to \cite{MU18}) of Edwards-Wilkinson fluctuations 
for the KPZ equation in  dimension $d\geq 3$ when $\hat\beta$ is sufficiently small. 
Using the same techniques (Clark-Ocone formula and second order Poincar\'e inequality), 
Gu~\cite{G18} proved the same Edwards-Wilkinson fluctuation as in our Theorem \ref{T:KPZ} for 
the KPZ equation in dimension $d=2$, except his result is restricted to $\hat\beta$ small instead of covering 
the entire subcritical regime.

\subsection{The directed polymer model}\label{S:DP}
In this subsection, we state our result for the 
partition function of the directed polymer model in dimension $2+1$.
See \cite{C17} for an overview of the directed 
polymer model. In the language of disordered systems,
space dimension $2$ is critical for this
model, where disorder is {\em marginally relevant}. For further background on the notion of 
disorder relevance/irrelevance (which corresponds to subcriticality/supercriticality in the 
context of singular SPDEs), see e.g.~\cite{H74, G10, CSZ17a}.

The directed polymer model is defined as a change of measure for a random walk,
depending on a random environment (disorder). 
Let $S$ be the simple symmetric random walk on $\Z^2$. If $S$ starts at $x\in \Z^2$, then we denote its law by $\P_x$ with expectation $\E_x$, and we omit $x$ when $x=0$. We set
\begin{equation}\label{eq:q}
	q_n(x) := \P(S_n = x) \,.
\end{equation}
Denoting by $\tilde S$ an independent copy of $S$, we define the expected overlap by
\begin{equation} \label{eq:RN}
	R_N := \sum_{n=1}^N \P(S_n = \tilde S_n)
	= \sum_{n=1}^N \sum_{x\in\Z^2} q_n(x)^2
	= \sum_{n=1}^N q_{2n}(0) = \frac{\log N}{\pi} + O(1) \,.
\end{equation}
We fix $\hat\beta \in (0,\infty)$ and define $(\beta_N)_{N\in\N}$ by
\begin{equation} \label{eq:betaN}
	\beta_N := \frac{\hbeta}{\sqrt{R_N}}
	= \frac{\sqrt{\pi} \hat\beta}{\sqrt{\log N}} \bigg(1 + \frac{O(1)}{\log N}\bigg) \,.
\end{equation}

Disorder is given by i.i.d.\ random variables $(\omega(n,x))_{n\in\N, x \in \Z^2}$ with law $\bbP$, such that
\begin{equation}\label{eq:omega}
	\bbE[\omega] = 0 \,, \qquad \bbE[\omega^2] = 1 \,, \qquad
	\lambda(\beta) := \log \bbE[e^{\beta \omega}] < \infty
	\quad \forall \beta > 0 \text{ small enough} \,.
\end{equation}
For technical reasons, we require that the law of $\omega$ satisfies a concentration inequality.
Recall that a function $f: \R^n \to \R$ is called $1$-Lipschitz if $|f(x)-f(y)| \le |x-y|$ for all $x,y\in\R^n$,
with $|\cdot|$ the Euclidean norm. We assume the following:
\begin{equation}\label{assD2}
\begin{split}
	\exists \gamma>1, C_1, C_2 \in (0,\infty): \text{ for all $n\in\N$ and } f: \R^n \to \R
	\text{ convex and $1$-Lipschitz} \\
	\bbP \Big( \big| f(\omega_1, \ldots, \omega_N) - M_f \big|
	\ge t\Big) \le C_1\exp \bigg(-\frac{t^\gamma}{C_2}\bigg) \,, \qquad \qquad \quad
\end{split}
\end{equation}
where $M_f$ denotes a median of
$f(\omega_1, \ldots, \omega_N)$. (By changing $C_1, C_2$, one can equivalently replace $M_f$ by
$\bbE[f(\omega_1, \ldots, \omega_N)]$, see \cite[Proposition 1.8]{Led}.)
Condition \eqref{assD2} is satisfied if $\omega$ is  bounded, or if it is Gaussian,
or more generally if it has a density $\exp(-V(\cdot) + U(\cdot))$, with $V$ uniformly strictly convex and 
$U$ bounded.
See \cite{Led} for more details.

Given $\omega$, $N\in\N$, and $\beta_N$ as defined in \eqref{eq:betaN}, we define the Hamiltonian by
\begin{equation} \label{eq:cHN}
	H_N :=
	\sum_{n=1}^N \big( \beta_N \, \omega(n,S_n) - \lambda(\beta_N) \big)
	= \sum_{n=1}^N
	\sum_{y \in \Z^2} \big( \beta_N \, \omega(n,y) - \lambda(\beta_N) \big) \,
	\ind_{\{S_n = y\}} \,.
\end{equation}
We will be interested in the family of partition functions
\begin{equation} \label{eq:ZN}
\begin{aligned}
	Z_N(x) & = Z_{N,\beta_N}(x) = \E_x\big[ e^{H_N}\big] \,, \qquad 
	& N\in\N, \ x\in \Z^2 \,, \\
	Z_N(x) & := Z_{\lfloor N \rfloor}(\lfloor x \rfloor),\, \qquad 
	& N\in[0,\infty), \ x\in \R^2 \,.
\end{aligned}
\end{equation}
We will write $Z_N := Z_N(0)$ for simplicity.
Note that the law of $Z_N(x)$ does not depend on $x\in\Z^2$, 
and we have $\bbE[Z_N(x)] = \bbE[Z_N] = 1$.
\smallskip

The partition function $Z_N(x)$ is a discrete analogue
(modulo a time reversal) of the SHE solution $u^\epsilon(t,x)$ in
\eqref{eq:ueps0}, as can be seen from its Feynman-Kac 
formula \eqref{eq:ueps} below
(see also \cite{AKQ14}). 
Then $\log Z_N(x)$ is a discrete analogue of the KPZ solution $h^\epsilon(t,x)$
in \eqref{eq:mollifiedKPZ}.
In fact, we proved in \cite[Theorem~2.8]{CSZ17b} 
that for $\beta_N$ as in \eqref{eq:betaN}, the random variable
$\log Z_N(x)$ converges in distribution to the same limit as in \eqref{eq:hlim},
with critical value $\hat\beta_c=1$.
It is not surprising that here we can also prove the following analogue of Theorem \ref{T:KPZ}.

\begin{theorem}[Edwards-Wilkinson fluctuations for directed polymer]\label{th:main_polymer}
Let $Z_{N,\beta_N}(x)$ be the family of partition functions
defined as in \eqref{eq:ZN}, with $\beta_N$ as in \eqref{eq:betaN} with $\hat\beta \in (0,1)$,
and the disorder $\omega$ 
satisfying assumptions \eqref{eq:omega} and \eqref{assD2}.
Denote
\begin{equation}\label{eq:H}
	\mathfrak{h}_N(t,x) := \frac{\log Z_{tN}(x \sqrt{N}) -
	\bbE[\log Z_{tN}]}{\beta_N} =
	\frac{\sqrt{\log N}}{\sqrt{\pi} \, \hat\beta}
	\big(\log Z_{tN}(x \sqrt{N}) -
	\bbE[\log Z_{tN}] \big) \,.
\end{equation}
For any $t>0$ and $\phi \in C_c^\infty(\R^2)$, the following convergence in law holds,
with $c_{\hat\beta}$ as in \eqref{eq:ASHE}:
\begin{equation} \label{eq:Hconv}
	\langle \mathfrak{h}_N(t,\cdot), \phi(\cdot) \rangle
	= \int_{\R^2} \mathfrak{h}_N(t,x) \, \phi(x) \, \dd x
	\ \xrightarrow[N\to\infty]{d} \ \langle v^{(\sqrt{2}c_{\hat\beta})}(t/2,\cdot), \phi(\cdot) \rangle \,,
\end{equation}
where $v^{(c)}(s,x)$ is the solution of the two-dimensional additive SHE as in \eqref{eq:ASHE}.
\end{theorem}
\begin{remark}
Here the limit $v^{(\sqrt{2}c_{\hat\beta})}(t/2,\cdot)$ differs from $v^{(c_{\hat\beta})}(t, \cdot)$ in 
Theorem~\ref{T:KPZ} because the increment of the simple symmetric random walk on $\Z^2$ 
has covariance matrix $\frac{1}{2}I$.
\end{remark}

We will in fact prove Theorem \ref{th:main_polymer} first, 
since the structure is more transparent in the discrete setting,
and then outline the changes needed to prove Theorem \ref{T:KPZ} for KPZ.

%\blue{Problem: can we replace $\bbE[\log Z_{tN}(z\sqrt{N})]$ by $-\frac{1}{2} \sigma_{\hat\beta}^2$
%in \eqref{eq:H}? This is OK if we can prove that
%$|\bbE[\log Z_N] - (-\frac{1}{2} \sigma_{\hat\beta}^2)| = o(1/\sqrt{\log N})$ as $N\to\infty$.}
%\end{remark}

\subsection{Outline}
The rest of the paper is organized as follows. 
\begin{itemize}
\item In Section~\ref{sec:methods}, we present the proof steps and 
describe the main ideas.
\item In Section~\ref{S:moments}, we give bounds on positive and negative moments 
for the directed polymer partition function, based on
concentration inequalities and hypercontractivity.
\item In Section~\ref{S:polymer}, we prove our main result
Theorem~\ref{th:main_polymer} for directed polymer. 
\item In Section~\ref{S:KPZ}, we explain how the proof for directed polymer 
can be adapted to prove our main result Theorem~\ref{T:KPZ} for KPZ. 
\end{itemize}
We will conclude with a few appendices
which might be of independent interest, where we prove some results needed
in the proofs.
\begin{itemize}
\item Appendix~\ref{S:scaling} 
establishes scaling relations for KPZ with different parameters. 
\item Appendix~\ref{sec:hyper} 
recalls and refines known hypercontractivity results for
suitable functions (polynomial chaos) of i.i.d.\ random variables. 
\item Appendix~\ref{sec:concGauss} formulates a concentration of measure result for
the left tail of convex functions that are not globally Lipschitz,
defined on general Gaussian spaces. 
\item Lastly in Appendix~\ref{sec:stocversion}
we discuss linearity and measurability properties of stochastic integrals,
which are needed in the proof in Section~\ref{S:KPZ}.
\end{itemize}

\section{Outline of proof steps and main ideas}\label{sec:methods}

In this section, we outline the proof steps for Theorems~\ref{T:KPZ}
and~\ref{th:main_polymer} and describe the basic setup. We focus on the directed polymer partition function
(the case of KPZ follows the same steps). The two main ideas are
a \emph{decomposition of the partition function $Z_N$
which allows to ``linearize''
$\log Z_N$} (see \S\ref{sec:linearize}),
and a representation of $Z_N$ as a \emph{polynomial chaos expansion}
in the disorder (see \S\ref{sec:poly}). The {\it ``linearization''} of $\log Z_N$
essentially reduces Theorem~\ref{th:main_polymer}
to an analogous result for $Z_N$ which we proved in \cite[Theorem~2.13]{CSZ17b}.

\subsection{Decomposition and linearization}
\label{sec:linearize}

Given a subset $\Lambda \subseteq \N \times \Z^2$, we denote by
\emph{$Z_{\Lambda,\beta}(x)$ the partition function where disorder is only sampled
from within $\Lambda$}, i.e.
\begin{equation} \label{eq:ZLambda}
	Z_{\Lambda,\beta}(x) := \E_x\big[ e^{H_{\Lambda,\beta}} \big] \,,
	\qquad \text{where} \qquad
	H_{\Lambda,\beta} := \sum_{(n,x) \in \Lambda} (\beta \omega_{n,x} -
	\lambda(\beta)) \ind_{\{S_n = x\}} \,.
\end{equation}
The original partition function
$Z_{N,\beta}(x)$ in \eqref{eq:cHN}-\eqref{eq:ZN}
corresponds to $\Lambda = \{1,\ldots, N\} \times \Z^2$.

\smallskip

In our previous study in \cite{CSZ17b}, a key observation was
that for $\hat\beta\in (0,1)$ the partition function
$Z_{N, \beta_N}(x)$ essentially depends only on disorder in a space-time window around
the starting point $(0,x)$
that is negligible on the diffusive scale $(N, \sqrt{N})$. This motivates 
us to approximate $Z_{N, \beta_N}(x)$ by a partition function $Z_{N, \beta_N}^A(x)$ with 
disorder present only in such a space-time window $A_N^x$. More precisely, we define a 
scale parameter $a_N$ tending to zero as
\begin{align}\label{eq:epsilon}
	a_N:=\frac{1}{(\log N)^{1-\gamma}} \qquad \text{with} \qquad \gamma
	\in (0, \gamma^*) \,,
\end{align}
where $\gamma^*>0$ depends only on $\hat\beta$ in Theorem \ref{th:main_polymer} and
its choice will be clear from the estimate in \eqref{RN2_gamma} later on. 
We now introduce the space-time window
\begin{align}\label{def:setA}
	A_N^x:=\Big\{ (n,z)\in \N\times \Z^2 \,\, \colon \,\, n\leq N^{1-a_N}\, , \,
	|z-x|<N^{\tfrac{1}{2}-\tfrac{a_N}{4}}  \Big\} \,,
\end{align}
and define $Z_{N,\beta}^A(x)$ as the partition function
which only samples disorder in $A_N^x$, i.e.\
\begin{equation}\label{eq:ZA0}
	Z_{N,\beta}^A(x) :=
	Z_{\Lambda,\beta}(x) \ \text{ with } \ \Lambda = A_N^x \,.
\end{equation}

We then decompose the original
partition function $Z_{N,\gb}(x)$ as follows:
\begin{align}\label{decomposition1}
Z_{N,\gb}(x) = Z_{N,\gb}^A(x) + \hat Z_{N,\gb}^A( x),
\end{align}
where $\hat Z_{N,\gb}^A(x)$, defined by the previous relation,
is a ``remainder''.
In a sense that we will make precise later (see \eqref{eq:2mom''})
it holds that for any {\it fixed} $x$, $\hat Z_{N,\gb_N}^A(x) \ll Z_{N,\gb_N}^A(x) $ and thus
\begin{align}\label{first_approx_heur0}
\log Z_{N,\gb_N}(x)
&= \log Z_{N,\gb_N}^A( x) + \log \Big(1+ \frac{\hat Z_{N,\gb_N}^A( x)}{Z_{N,\gb_N}^A( x)} \Big)
\approx \log Z_{N,\gb_N}^A( x) + \frac{\hat Z_{N,\gb_N}^A( x)}{Z_{N,\gb_N}^A( x)}.
\end{align}
More precisely, if we define the error $O_N(x)$ via
\begin{align}\label{first_approx_heur}
\log Z_{N,\gb_N}(x) =
 \log Z_{N,\gb_N}^A( x) + \frac{\hat Z_{N,\gb_N}^A( x)}{Z_{N,\gb_N}^A( x)}+ O_N(x),
\end{align}
then we will show the following:

\begin{proposition}\label{prop:R}
Let $O_N(\cdot)$ be defined as in \eqref{first_approx_heur}, then
for any $\phi \in C_c(\R^2)$
\begin{equation} \label{eq:Rconv}
	\sqrt{\log N} \, \frac{1}{N} \sum_{x\in \Z^2}
	\big( O_N(x) - \bbE[O_N(x)] \big) \, \phi(\tfrac{x}{\sqrt{N}})
	\ \xrightarrow[N\to\infty]{L^2(\bbP)} \ 0 \,.
\end{equation}
\end{proposition}

Remarkably, even though $\log Z_{N,\gb_N}^A(x)$ gives the dominant contribution to $\log Z_{N,\gb_N}(x)$ for any {\it fixed} $x$, it does not
contribute to the fluctuations of $\log Z_{N,\gb_N}(x)$ when {\it averaged} over $x$, that is:

\begin{proposition}\label{prop:ZA}
Let  $Z_{N,\gb_N}^A(\cdot)$ be defined as in \eqref{eq:ZA0}, then
for any $\phi \in C_c(\R^2)$
\begin{equation} \label{eq:Z'conv}
	\sqrt{\log N} 
	\, \frac{1}{N} \sum_{x\in\Z^2}
	\big( \log Z_{N,\gb_N}^A(x) - \bbE[\log Z_{N,\gb_N}^A(x)] \big)
	\, \phi(\tfrac{x}{\sqrt{N}})
	\ \xrightarrow[N\to\infty]{L^2(\bbP)} \ 0 \,.
\end{equation}
\end{proposition}
As a consequence, the fluctuations of $\log Z_{N,\gb_N}(\cdot) $
are determined by the ``normalized remainder''
$\hat Z_{N,\gb_N}^A( \cdot)/Z_{N,\gb_N}^A( \cdot)$.
To determine the fluctuations of this term, we define the set
\begin{equation} \label{eq:B}
	B_N^\geq := \big(\, (N^{1-9a_N/40}, N] \cap \N \,\big) \times \Z^2 ,
\end{equation}
and we let $Z_{N,\gb_N}^{B^\geq}(x)$ be the partition function where disorder is sampled only from $B^\geq_N$, i.e.\
\begin{equation}\label{eq:ZB0}
	Z_{N,\beta_N}^{B^\geq}(x) :=
	Z_{\Lambda,\beta_N}(x) \ \text{ with } \ \Lambda = B_N^\geq \,.
\end{equation}
Note that $\bbE[Z_{N,\beta_N}^{B^\geq}(x)] = 1$, so $( Z_{N,\gb_N}^{B^\geq}(x) - 1)$
is a centered random variable.
The key point, and the more involved step, will be to show that
\begin{align}\label{main-appro}
\hat Z_{N,\gb_N}^{A}(x) \approx Z_{N,\gb_N}^A(x) \,
\big( Z_{N,\gb_N}^{B^\geq}(x) - 1 \big)  \,,
\end{align}
in the following sense.

\begin{proposition}\label{prop:ZB}
Let $Z_{N,\gb_N}^A(\cdot)$, $\hat Z_{N,\gb_N}^A(\cdot)$, $Z_{N,\gb_N}^{B^\geq}(\cdot)$
be defined as in \eqref{eq:ZA0},
\eqref{decomposition1}, \eqref{eq:ZB0}. Then
for any $\phi \in C_c(\R^2)$
\begin{equation} \label{main-approx}
	\sqrt{\log N}  \, \frac{1}{N} \sum_{x\in\Z^2}
	\bigg(
	\frac{\hat Z_{N,\gb_N}^A(x)}{Z_{N,\gb_N}^A(x)} \, - \,
	\big( Z_{N,\gb_N}^{B^\geq}(x) - 1 \big)
	\bigg)
	\, \phi(\tfrac{x}{\sqrt{N}})
	\ \xrightarrow[N\to\infty]{L^1(\bbP)} \ 0 \,.
\end{equation}
\end{proposition}

It remains to identify the fluctuations of $Z_{N,\gb_N}^{B^\geq}(\cdot)$.
This falls within the scope of Theorem~2.13 in \cite{CSZ17b},
which we will recall in Section~\ref{proof:prop:hatZA} and which will show that
the fluctuations of  $Z_{N,\gb_N}^{B^\geq}(\cdot)$ converges to the solution 
$v^{(\sqrt{2}c_{\hat\beta})}(1/2,\cdot)$ of the two-dimensional additive SHE, as in 
Theorem~\ref{th:main_polymer}.
The proof is based on polynomial chaos expansions of the partition function, 
which we will recall in the next subsection.

\begin{proposition}\label{prop:hatZA}
Let $Z^{B^\geq}_{N,\gb_N}(\cdot)$ be defined as in \eqref{eq:ZB0}. Then
\begin{equation} \label{eq:hatZAconv}
	\frac{\sqrt{\log N}}{\sqrt{\pi} \, \hat\beta}  \, \frac{1}{N} \sum_{x\in\Z^2}
	 \big( Z_{N,\gb_N}^{B^\geq}(x) - 1 \big)
	\, \phi(\tfrac{x}{\sqrt{N}})
	\ \xrightarrow[N\to\infty]{d} \
	\langle v^{(\sqrt{2}c_{\hat\beta})}(1/2,\cdot), \phi \rangle \,,
\end{equation}
where $v^{(c)}(s,x)$ is the solution of the two-dimensional additive SHE as in \eqref{eq:ASHE}.
\end{proposition}

Theorem~\ref{th:main_polymer} is a direct corollary
of the decomposition \eqref{first_approx_heur}
and Propositions~\ref{prop:R}-\ref{prop:hatZA}.
Regarding the centering, it suffices to note that
$\bbE[\log Z_{N,\gb_N}(x)] =
\bbE[\log Z_{N,\gb_N}^A( x)] +
\bbE[O_N(x)]$, because the random variable
$\hat Z_{N,\gb_N}^A(x)/Z_{N,\gb_N}^A(x)$ has zero mean, which follows from
the polynomial chaos expansions of partition functions that we now present.

\subsection{Polynomial chaos expansions} \label{sec:poly}

Our analysis of the partition functions $Z_{N,\gb_N}^A(\cdot)$, 
$\hat Z_{N,\gb_N}^A(\cdot)$, $Z_{N,\gb_N}^B(\cdot)$ is based on multi-linear expansions, 
known as polynomial chaos expansions, which have also been used extensively in \cite{CSZ17a, CSZ17b}.

\smallskip

Recall the definition \eqref{eq:betaN} of $\beta_N$ and
our assumptions \eqref{eq:omega} on the disorder,
and note that by \eqref{eq:omega} and Taylor expansion, 
$\lambda(2\beta) - 2\lambda(\beta) \sim \beta^2$ as $\beta \to 0$.
We introduce the sequence
\begin{equation}\label{eq:sigma}
	\sigma_N := \sqrt{ e^{\lambda(2\beta_N) - 2\lambda(\beta_N)}-1 }
	\underset{N\to\infty}{\sim} \beta_N \,,
\end{equation}
where we agree that $a_N \sim b_N$ means $\lim_{N\to\infty} a_N/b_N = 1$,
and we define the random variables
\begin{equation} \label{eq:xi}
	\xi_{n,x}^{(N)} := \sigma_N^{-1}
	\Big( e^{\beta_N \omega(n,x) -\gl(\gb_N)} - 1 \Big) \,.
\end{equation}
We will suppress the dependence of $\xi_{n,x}^{(N)}$ on $N$, for notational simplicity.
Note that $(\xi_{n,x})_{n\in\N, x\in \Z^2}$ are i.i.d.\ with $\bbE[\xi_{n,x}]=0$ and 
$\bbE[\xi^2_{n,x}] = 1$.

Recall the definition \eqref{eq:cHN}-\eqref{eq:ZN} of
the partition function $Z_{N,\gb_N}(x)$ of the polymer that starts at time zero from location $x$.
This can be written as
\begin{equation}\label{eq:poly}
\begin{split}
	Z_{N,\gb_N}(x) & = \E_x\Bigg[ \prod_{1 \le n \le N, \, y \in \Z^2}
\big( 1+\sigma_N \, \xi_{n,y}\, \ind_{\{S_n = y\}}  \big) \Bigg] \\
	& = 1 + \sum_{k=1}^N \sigma_N^k
	\sumtwo{0 =n_0 < n_1 < \ldots < n_k \le N}
	{x_0=x, \ x_1, \ldots, x_k \in \Z^2}
	\prod_{i=1}^k q_{n_i - n_{i-1}}(x_i - x_{i-1}) \, \xi_{n_i, x_i}, 
\end{split}
\end{equation}
where $q_n(x):=\P(S_n=x)$. Note that the terms in the sum are orthogonal to each other in $L^2$, 
and when $\hat\beta\in (0,1)$ the dominant contribution to $Z_{N, \gb_N}(x)$ comes from
disorder $\xi_{\cdot, \cdot}$ in a space-time window that is negligible on the diffusive scale. 
More precisely, a second moment calculation (see \eqref{eq:2mom''} below) shows that
$Z_{N,\gb_N}(x)$ is close in $L^2$ to the partition function $Z_{N,\gb_N}^A(x)$
which only samples disorder from within $A_N^x$ (recall \eqref{def:setA}).

It will be convenient to introduce a concise representation 
for the expansion \eqref{eq:poly} as follows:
given a point $(n_0, x_0)$
and a finite subset $\tau:=\{(n_1,x_1),...,(n_{|\tau|},x_{|\tau|})\}$
of $\N_0\times \Z^2$ with $n_0<n_1<\cdots<n_{|\tau|}$,
we introduce the notation
\begin{align*}
	\bq^{(n_0,x_0)}(\tau) := \prod_{i=1}^{|\tau|} q_{n_i-n_{i-1}}(x_i-x_{i-1}) 
	\qquad \text{and}\qquad  \bxi(\tau) := \prod_{i=1}^{|\tau|} \xi_{n_i, x_i}.
\end{align*}
For $\tau = \emptyset$ we define $\bq^{(n_0,x)}(\tau)= \bxi(\tau):=1$.
In this way we can write concisely the chaos expansion of $Z_{N,\beta_N}(x) $ as
\begin{align}\label{Z-chaos}
Z_{N,\beta_N}(x) =
\sum_{\tau \subset \{1,\ldots, N\} \times \Z^2} \sigma_N^{|\tau|}
\,\bq^{(0,x)}(\tau) \, \bxi(\tau) \,.
\end{align}
Similarly, for the partition functions $Z_{N,\beta}^A(x)$, $Z_{N,\beta}^{B^\geq}(x)$
in \eqref{eq:ZA0}, \eqref{eq:ZB0} we can write
\begin{gather}\label{eq:ZAB}
	Z_{N,\beta_N}^A(x) = \sum_{\tau\,\subset\, A_N^x}
	\sigma_N^{|\tau|}\,\bq^{(0,x)}(\tau) \,\bxi(\tau) \,, \qquad
	Z_{N,\gb_N}^{B^\geq}(x) =\sum_{\tau \,\subset\,
	B_N^\geq}
	\sigma_N^{|\tau|} \, \bq^{(0,x)}(\tau)\, \bxi(\tau) \,.
\end{gather}
A graphical illustration of $Z_{N,\beta}^A(x)$ and $Z_{N,\beta}^{B^\geq}(x)$ appears in Figure~\ref{fig:multi_repAB}.

%%%%%%%%%%%%%%%%%%%%%%%%%%%%%%%%%%
% Figure
\begin{figure}[t]
\hskip -0.2cm
\begin{minipage}[b]{.33\linewidth}
\centering
\begin{tikzpicture}[scale=0.5]
\draw (0,-5)--(0,5); \draw (0,-2)--(7,-2)--(7,2)--(0,2);
\draw[dashed] (7,-5)--(7,5);
\node at (1,-1.5) {\scalebox{0.6}{$A_N^x$}};
\draw[<->] (0.1,2.5)--(6.9, 2.5);  \node at (3.5,3) {\scalebox{0.6}{$N^{1-a_N}$}};
\draw  [fill] (0, 0)  circle [radius=0.1]; \draw  [fill] (1, 1)  circle [radius=0.1]; \draw  [fill] (2, -0.3)  circle [radius=0.1]; \draw  [fill] (3.5, -0.7)  circle [radius=0.1];
\draw  [fill] (4.5, 0.8)  circle [radius=0.1]; \node at (6,0) {\scalebox{0.6}{{$\cdots$}}};
\draw (0,0) to [out=60,in=-160] (1,1) to [out=-80,in=160] (2,-0.3) to [out=30,in=130] (3.5,-0.7) to [out=80,in=-150] (4.5, 0.8) to [out=30,in=100] (5.5, 0);
\node at (-1,0) {\scalebox{0.6}{$(0,x)$}};
\end{tikzpicture}
\subcaption{Partition fuction $Z_{N,\gb_N}^A(x)$.
\label{figure:ZA}}
\end{minipage}
\,\,
\begin{minipage}[b]{.33\linewidth}
\centering
\begin{tikzpicture}[scale=0.5]
\draw (0,-5)--(0,5); \draw (0,-2)--(4,-2)--(4,2)--(0,2);
\draw[<->] (8.5,-4.5)--(12.5, -4.5);
\draw[<->] (0.1,2.5)--(3.9, 2.5);  \node at (2,3) {\scalebox{0.6}{$N^{1-a_N}$}};
\draw[dashed] (4,-5)--(4,5); \draw[dashed] (8.5,-5)--(8.5,5);
\node at (1,-1.5) {\scalebox{0.6}{$A_N^x$}};
%\node at (1,-5) {\scalebox{0.6}{$C_N^x$}};
\node at (10.5,-5) {\scalebox{0.6}{$B_N^{\geq}$}};
\node at (-1,0) {\scalebox{0.6}{$(0,x)$}}; \draw  [fill] (0, 0)  circle [radius=0.1];
 \draw  [fill] (9,1.5)  circle [radius=0.1];
\draw  [fill] (11, -1)  circle [radius=0.1];
\draw (0,0) to [out=60,in=-160]   (9,1.5) to [out=30, in=100] (11,-1) 
to [out=90, in=180] (12,0); \node at (12.5,0) {\scalebox{0.6}{$\cdots$}};
\end{tikzpicture}
\subcaption{Partition function $Z_{N,\gb_N}^{B^\ge}(x).$\label{figure:ZB}}
\end{minipage}
\hskip -0.3cm
\caption{
The above figures depict the chaos expansions of $Z_{N,\gb_N}^A(x)$ and $Z_{N,\gb_N}^{B^\geq}(x)$. The disorder sampled by
 $Z_{N,\gb_N}^A(x)$ is restricted to the set $A_N^x$, while that of $Z_{N,\gb_N}^{B^\geq}(x)$ is restricted to $B_N^\geq$\label{fig:multi_repAB}}
\end{figure}
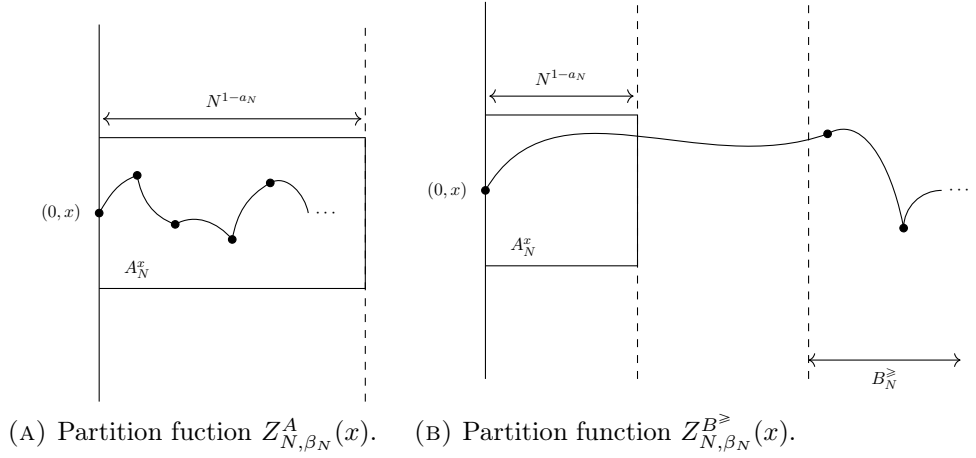

%%%%%%%%%%%%%%%%%%%%%%%%%%%%%%%%%%

\smallskip

These polynomial chaos expansions are discrete analogues of Wiener-It\^o chaos expansions.
They are especially suited for variance calculations and provide important insight. For instance, 
the partition function
$\hat Z_{N,\beta}^A(x) := Z_{N,\beta}(x) - Z_{N,\beta}^A(x)$,
see \eqref{decomposition1}, is obtained by restricting
the sum in \eqref{Z-chaos} to $\tau$
which include space-time points outside the set $A_N^x$, 
and hence $\hat Z_{N,\gb_N}^A(x)/Z_{N,\gb_N}^A(x)$ has zero mean
due to the independence between the disorder inside and outside $A_N^x$. Similarly, 
the centered partition function $(Z_{N,\beta}^{B^\geq}(x) - 1)$,
which appears in \eqref{main-approx}-\eqref{eq:hatZAconv}, 
is the contribution to \eqref{eq:ZAB} given 
by configurations $\tau$ that contain only points (and at least one point) in $B^\geq_N$.

\section{Moment bounds}\label{S:moments}

In this section we collect some moment bounds that will be used in the proof.

\subsection{Second moment}\label{S:2ndmom}
We bound the second moment of
$Z_{N,\beta_N}(x)$, $Z_{N,\beta_N}^A(x)$, $\hat Z_{N,\beta_N}^A(x)$.

\smallskip

We start from $Z_{N,\beta_N}(x)$.
It follows by \eqref{eq:poly} and  \eqref{eq:RN} that
\begin{equation} \label{eq:2momZ}
\begin{split}
	\bbE[Z_{N,\beta_N}(x)^2] &=
	\sum_{\tau\,\subset\, \{1, \ldots, N\}\times \Z^2}
	(\sigma_N^2)^{|\tau|}\,\bq^{(0,x)}(\tau)^2 \\
	&= 1 +  \sum_{k=1}^N (\sigma_N^2)^k 
	\sumtwo{0 =:n_0 < n_1 < \ldots < n_k \le N}
	{x_0 =: x, \ x_1, \ldots, x_k \in \Z^2}
	\prod_{i=1}^k q_{n_i - n_{i-1}}(x_i - x_{i-1})^2 \\
	&= 1 +  \sum_{k=1}^N (\sigma_N^2)^k
	\sum_{0 =:n_0 < n_1 < \ldots < n_k \le N}
	q_{2(n_i - n_{i-1})}(0) \,.
\end{split}
\end{equation}
If we let each increment
$n_i - n_{i-1}$ vary freely in $\{1,2,\ldots, N\}$,
by \eqref{eq:RN} we get the bound $\bbE[Z_{N,\beta_N}(x)^2] \le
\sum_{k\ge 0} (\sigma_N^2 \, R_N)^k = (1-\sigma_N^2 \, R_N)^{-1}$.
Recalling \eqref{eq:betaN} and \eqref{eq:sigma}, we obtain
\begin{equation} \label{eq:2mom}
	\forall \hat\beta \in (0,1)
	\ \exists\, C_{\hat\beta} < \infty \ \text{ such that }
	\ \forall N \in \N: \qquad
	\bbE[Z_{N,\beta_N}(x)^2] \le C_{\hat\beta} \,,
\end{equation}
where $C_{\hat\beta}$ will denote a generic constant depending on $\hat \beta$.

\smallskip

Next we look at $Z_{N,\beta_N}^A(x)$.
The polynomial chaos expansion for $Z_{N,\beta_N}^A(x)$ is a subset
of the one for $Z_{N,\beta_N}(x)$, hence the same bound \eqref{eq:2mom} applies:
\begin{equation} \label{eq:2mom'}
	\forall \hat\beta \in (0,1)
	\ \exists\, C_{\hat\beta < \infty} \ \text{ such that }
	\ \forall N \in \N: \qquad
	\bbE[Z_{N,\beta_N}^A(x)^2] \le C_{\hat\beta}\,.
\end{equation}

\smallskip

We turn to $\hat Z_{N,\beta_N}^A(x)$.
The bound \eqref{eq:2mom} can again be applied,
but it is quite poor.
In fact, the following much better bound holds
(recall that $a_N$ is defined in \eqref{eq:epsilon}):
\begin{equation} \label{eq:2mom''}
	\forall \hat\beta \in (0,1) \ \exists\, C_{\hat\beta} < \infty
	\ \text{ such that } \ \forall N \in \N: \qquad
	\bbE[\hat Z_{N,\beta_N}^A(x)^2] \le C_{\hat\beta} \, a_N \,.
\end{equation}
The proof, given below, is elementary but slightly technical (see
Subsection~\ref{sec:2mom''}).

\smallskip

We conclude with an alternative viewpoint on the bound \eqref{eq:2mom}.
If we denote by $S$ and $\tilde S$ two independent copies of the random walk,
by \eqref{eq:cHN}-\eqref{eq:ZN} we can compute
\begin{equation} \label{eq:2momov}
	\bbE[Z_{N,\beta_N}^2] = \bbE \E[e^{H_{N,\beta_N}(S) + H_{N,\beta_N}(\tilde S)}]
	= \E[e^{(\lambda(2\beta_N) - 2\lambda(\beta_N)) \, \cL_N(S,\tilde S)}] \,,
\end{equation}
where $\cL_N(S,\tilde S)$ is the overlap of the two copies $S, \tilde S$ up to time $N$, defined by
\begin{equation} \label{eq:LN}
	\cL_N(S,\tilde S) := \sum_{n=1}^N \ind_{\{S_n = \tilde S_n\}}
	= \big| S \cap \tilde S \cap (\{1,\ldots,N\} \times \Z^2) \big| \,.
\end{equation}
Since $\lambda(\beta) \sim \frac{1}{2} \beta^2$ as $\beta \downarrow 0$,
see \eqref{eq:omega}, we get
\begin{equation}\label{eq:2mom-overlap}
	\bbE[Z_{N,\beta_N}^2] = \E[e^{(1+\epsilon_N) \beta_N^2 \,
	\cL_N(S,\tilde S)}] \,,
	\qquad \text{where} \qquad \lim_{N\to\infty} \epsilon_N = 0 \,.
\end{equation}
Note that $\frac{\pi}{\log N}\cL_N(S,\tilde S)$ converges in law to a mean $1$ 
exponential random variable, see e.g.~\cite{ET60}.
This matches with $\lim_{N\to\infty} \bbE[Z_{N,\beta_N}^2] = (1-\hat\beta^2)^{-1}$ 
for $\beta_N$ as in \eqref{eq:betaN}.

\subsection{Positive moments via hypercontractivity}

We will bound higher positive moments of our partition functions using the \emph{hypercontractivity}
of polynomial chaos \cite{MOO}, which we recall (with some strengthening) in Appendix~\ref{sec:hyper}.
\smallskip

By \eqref{eq:poly}, each
partition function $Z_{N,\beta_N}(x)$, $Z_{N,\beta_N}^A(x)$, $\hat Z_{N,\beta_N}^A(x)$
can be expressed as a series
\begin{equation} \label{eq:series0}
	\sum_{k=0}^\infty X_k^{(N)}
\end{equation}
(actually a finite sum)
where $X_k^{(N)}$ is a multi-linear
polynomial of degree $k$ in the i.i.d.\ random variables
$(\xi^{(N)}_{n,x})_{(n,x) \in \N \times \Z^2}$, which have zero mean and 
unit variance, see \eqref{eq:xi}. These random variables have uniformly bounded higher moments:
\begin{equation}\label{eq:asseta+0}
	\forall p \in (2,\infty): \qquad
	\sup_{N\in\N} \bbE[|\xi^{(N)}_{n,x}|^{p}] < \infty \,,
\end{equation}
as one can check directly from \eqref{eq:xi} and \eqref{eq:omega}
(see \cite[eq. (6.7)]{CSZ17a}).

Under these conditions, hypercontractivity
ensures that, for every $p\in (2,\infty)$,
the $p$-th moment of the series \eqref{eq:series0} can be bounded
in terms of second moments:
\begin{equation} \label{eq:hyper}
	\bbE \bigg[ \bigg| \sum_{k=0}^\infty X_k^{(N)} \bigg|^p \bigg]
	\le
	\bigg( \sum_{k=0}^\infty (c_p^k)^2 \, \bbE\big[ \big(
	X_k^{(N)} \big)^2 \big] \bigg)^{p/2} \,,
\end{equation}
where $c_p \in (1,\infty)$ is a constant, uniform in $N$,
which only depends on the laws of the $\xi^{(N)}_{n,x}$.
This is proved in \cite[\S~3.2]{MOO} (extending \cite{J97}), where a
non-optimal value of $c_p$ is given. We will recall these results
in Appendix~\ref{sec:hyper}, where we will prove that the optimal $c_p$ satisfies
\begin{equation}\label{eq:cp1}
	\lim_{p\downarrow 2} c_p = 1 \,.
\end{equation}

This result, which is of independent interest, is crucial
in order to apply \eqref{eq:hyper} to our partition functions
$Z_{N,\beta_N}(x)$, $Z_{N,\beta_N}^A(x)$,
$\hat Z_{N,\beta_N}^A(x)$,
because \emph{for any subcritical
$\hat\beta < 1$ we can fix $p > 2$
such that $c_p \hat\beta < 1$ is still subcritical}. 
More precisely, note that multiplying $X_k^{(N)}$ by $c_p^k$ amounts
to replacing $\sigma_N$ by $c_p \, \sigma_N$, see \eqref{eq:poly},
and this corresponds asymptotically
to replacing $\hat\beta$ by $c_p \hat\beta$, see \eqref{eq:sigma} and \eqref{eq:betaN}.
Then, by \eqref{eq:2mom}-\eqref{eq:2mom''},
we obtain:
\begin{align}
	\nonumber
	& \forall \hat\beta \in (0,1) \ \ \exists
	p = p_{\hat\beta} \in (2,\infty) \ \ \exists C'_{\hat\beta} < \infty
	\ \text{ such that } \ \forall N \in \N \\
	\label{eq:pmom}
	& \bbE\big[ Z_{N,\beta_N}(x)^p \big] \le
	C'_{\hat\beta} \,,
	\qquad \bbE\big[ Z_{N,\beta_N}^A(x)^p \big] \le
	C'_{\hat\beta} \,,
	\qquad \bbE\big[ |\hat Z_{N,\beta_N}^A(x)|^p \big] \le
	C'_{\hat\beta} \, (a_N)^{p/2} \,.
\end{align}

\subsection{Negative moments via concentration}

We give bounds on the \emph{negative} moments of
partition functions $Z_{N,\beta_N}(x)$
and
$Z_{N,\beta_N}^A(x)$
(see \eqref{eq:ZN}, \eqref{eq:poly} and \eqref{def:setA},\eqref{eq:ZAB}).
We work with the general partition function $Z_{\Lambda,\beta}(x)$
defined in \eqref{eq:ZLambda}, which coincides with $Z_{N,\beta_N}(x)$,
resp.\ $Z_{N,\beta_N}^A(x)$,
for $\Lambda = \{1,\ldots, N\} \times \Z^2$,
resp.\ $\Lambda = A_N^x$.

For fixed (say bounded) $\Lambda\subseteq \N \times \Z^2$,
it is not difficult to show that the log partition function
$\log Z_{\Lambda,\gb}$ is a convex and Lipschitz function of the random
variables $(\omega(n,y): \ (n,y) \in \Lambda)$.
However, if $\beta = \beta_N$ and
the subset $\Lambda$ grows with $N$, its Lipschitz
constant can diverge as $N\to\infty$, hence we cannot directly apply
the concentration inequality \eqref{assD2}.
However, it turns out that, for any $\Lambda = \Lambda_N \subseteq \{1,\ldots, N\} \times \Z^2$,
the Lipschitz constant is \emph{tight} as $N\to\infty$.
This yields the following estimate for the \emph{left tail} of
$\log Z_{\Lambda_N,\gb_N}$, proved below.

\begin{proposition}[Left tail]\label{neg-mom-poly}
For any $\hat\beta \in (0,1)$, there exists
$c_{\hat\beta} \in (0,\infty)$ with the following property:
for every $N\in\N$ and for every choice
of $\Lambda \subseteq \{1,\ldots, N\} \times \Z^2$, one has
\begin{align} \label{eq:boundf}
	\forall t \ge 0: \qquad
	\ \bbP(\log Z_{\Lambda,\gb_N}\leq -t) \leq
	c_{\hat\gb} \, e^{- t^\gamma / c_{\hat\gb}} \,,
\end{align}
where $\gamma > 1$ is the same exponent appearing in assumption \eqref{assD2}.
\end{proposition}

\smallskip

As a corollary,
for every $p\in (0,\infty)$ we can estimate,
\emph{uniformly in $\Lambda \subseteq \{1,\ldots, N\} \times \Z^2$},
\begin{equation*}
\begin{split}
	\bbE \big[ (Z_{\Lambda,\gb_N})^{-p} \big]
	& = \bbE\big[ e^{-p\log Z_{\Lambda,\gb_N}}\big] = p \int_{-\infty}^{\infty}
	e^{pt} \, \bbE\big[ \ind_{\{t < -\log Z_{\Lambda,\gb_N}\}} \big]
	\, \dd t \\
	& \le 1 + p \int_{0}^{\infty} e^{pt} \,
	c_{\hat\beta} \, \exp\big( - t^\gamma / c_{\hat\beta} \big)  \, \dd t
	=: C_{p,\hat\beta} < \infty \,.
\end{split}
\end{equation*}
Choosing $\Lambda = \{1,\ldots, N\} \times \Z^2$ or $\Lambda = A_N^x$,
we finally obtain the bounds
\begin{align} \label{eq:-pmom}
	\forall \hat\beta \in (0,1) \ \forall p \in (0,\infty) \quad
	\exists C_{p,\hat\beta} < \infty: \qquad
	& \sup_{N\in\N} \bbE \big[ Z_{N,\gb_N}(x)^{-p} \big] \le C_{p,\hat\beta} < \infty \,, \\
	\label{eq:-pmom'}
	& \sup_{N\in\N}
	\bbE \big[ Z_{N,\gb_N}^A(x)^{-p} \big] \le C_{p,\hat\beta} < \infty \,.
\end{align}
For later use, let us also state the following consequence:
\begin{equation} \label{eq:momlog}
	\forall \hat\beta \in (0,1) \ \forall p \in (0,\infty) \quad \exists
	C_{p,\hat\beta} < \infty: \qquad
	\sup_{N\in\N} \bbE[|\log Z_{N,\gb_N}^A(x)|^p] \le C_{p,\hat\beta} < \infty \,.
\end{equation}
The proof of this fact is simple: we can bound $|\log y| \le C_p \, (y^{1/p} + y^{-1/p})$
for all $y > 0$ and
for suitable $C_p < \infty$ (just distinguish $y \ge 1$ and $y < 1$). This leads to
$\bbE[|\log Z_{N,\gb_N}^A(x)|^p] \le C_p(\bbE[Z_{N,\gb_N}^A(x)]
+ \bbE[Z_{N,\gb_N}^A(x)^{-1}])
= C_p(1 + \bbE[Z_{N,\gb_N}^A(x)^{-1}])$, so \eqref{eq:momlog} follows
by \eqref{eq:-pmom'}.

\medskip

It remains to prove Proposition~\ref{neg-mom-poly}.
To this goal,
we follow the strategy developed in \cite{CTT} for the pinning model,
which generalizes \cite{cf:Moreno}.
We need the following result, which is \cite[Proposition~3.4]{CTT},
inspired by \cite[Proposition 1.6]{Led}.

\begin{proposition}\label{prop:concentration}
Assume that disorder $\omega$ has the concentration property \eqref{assD2}. There
exist constants $c_1, c_2 \in (0,\infty)$ such that,
for every $n\in\N$ and for every differentiable convex function $f\colon \R^n\to \R$,
the following bound holds
for all $a\in \R$ and $t,c\in (0,\infty)$,
\begin{align} \label{eq:thebound}
	\bbP\big(f(\omega)\leq a-t \big) \,
	\,\bbP\big(f(\omega)\geq a, |\nabla f(\go) | \leq c \big)
	\leq c_1 \exp\Big( -\frac{(t/c)^\gamma}{c_2} \,\Big),
\end{align}
where $\omega = (\omega_1, \ldots, \omega_n)$ and
$|\nabla f(\go)|:=\sqrt{\sum_{i=1}^n(\partial_i f(\go))^2}$
is the norm of the gradient.
\end{proposition}

We can deduce the bound \eqref{eq:boundf} from \eqref{eq:thebound}
applied to the function $f = f_N$ given by
\begin{equation} \label{eq:fN}
	f_N(\omega) = \log Z_{\Lambda, \beta_N} \,.
\end{equation}
We only need to bound from below the second probability in the left
hand side of \eqref{eq:thebound}. This is provided by the next lemma,
which completes the proof of Proposition~\ref{neg-mom-poly}.

\begin{lemma}\label{Conce_aux}
For any $\hat\beta \in (0,1)$, there exist $c_{\hat\gb} \in (0,\infty)$
and $\theta_{\hat\gb} \in (0,1)$ such that
\begin{align}\label{concentrate1}
	\inf_{N \in \N} \ \inf_{\Lambda \subseteq \{1,\ldots, N\} \times \Z^2} \
	\bbP \big( f_N(\go) \geq -\log 2 \,,
	|\nabla f_N(\go)|\leq c_{\hat\gb} \big) \ge
	\theta_{\hat\beta} >0 \,.
\end{align}
\end{lemma}

\begin{proof}
We set $a=-\log 2$. For any $c > 0$, we have
\begin{align} \label{eq:deba}
	\bbP\big( f_N(\go)\geq a \,,|\nabla f_N(\go)|\leq c\big) = \bbP\big(f_N(\go)\geq a\big) -
	\bbP\big(f_N(\go)\geq a \,,|\nabla f_N(\go)|> c\big)
\end{align}
The first probability can be estimated using the Paley-Zygmund inequality:
\begin{align}\label{PL}
	\bbP\big(f_N(\go)\geq a\big) = \bbP\big( Z_{\Lambda,\gb_N} \geq \tfrac{1}{2}\big) =  	
	\bbP\big( Z_{\Lambda,\gb_N} \geq \tfrac{1}{2} \bbE [Z_{\Lambda,\gb_N}] \big)
	\geq \frac{\bbE [Z_{\Lambda,\gb_N}]^2}{4 \, \bbE [ (Z_{\Lambda,\gb_N})^2 ]}.
\end{align}
Note that $\bbE [Z_{\Lambda,\gb_N}]=1$. For $\Lambda \subseteq \{1,\ldots, N\} \times \Z^2$
we have $\bbE [ (Z_{\Lambda,\gb_N})^2 ] \le
\bbE [ (Z_{N,\gb_N})^2 ] \le C_{\hat\beta}$, see \eqref{eq:2mom}, hence
\begin{equation}\label{PL2}
	\bbP\big(f_N(\go)\geq a\big)
	\geq \tfrac{1}{4 C_{\hat\beta}} =: 2 \, \theta_{\hat\beta} \,.
\end{equation}

We now proceed to estimate the second term in \eqref{eq:deba}.
First, we compute for $n\in \N, x\in \Z^2$
\begin{align*}
\frac{\partial f_N(\go)}{\partial \go_{n,x}} = \frac{1}{Z_{\Lambda,\gb_N}}
\E \Big[ \gb_N \ind_{(n,x)\in S\cap \Lambda} \,
e^{ H_{\Lambda,\gb_N}(S)} \Big] \qquad \text{and}
\end{align*}
\begin{align*}
|\nabla f_N(\go)|^2 = \sum_{(n,x)\in \N\times \Z^2}
\Big(\frac{\partial f_N}{\partial \go_{n,x}}\Big)^2 =
 \frac{1}{(Z_{\Lambda,\gb_N})^2} \E \Big[ \gb_N^2
 \,|S\cap \tilde S \cap \Lambda| \,\,
 e^{ H_{\Lambda,\gb_N}(S)+H_{\Lambda,\gb_N}(\tilde S)} \Big] \,,
\end{align*}
where $S$ and $\tilde S$ are two independent copies of the random walk,
and with some abuse of notation, we also denote by $S$ the random subset
$\{(n,S_n)\}_{n\in\N} \subseteq \N \times \Z^2$.

For $\Lambda \subseteq \{1,\ldots, N\} \times \Z^2$ we have
$|S\cap \tilde S \cap \Lambda| \le \cL_N(S,\tilde S)$, see \eqref{eq:LN}, where
$\cL_N(S,\tilde S)$ denotes the overlap
up to time $N$ of the two trajectories $S$ and $\tilde S$.
On the event that $f_N(\go)\geq a = -\log 2$,
that is $Z_{\Lambda,\gb_N}\geq 1/2$,
we can thus bound
\begin{align*}
	|\nabla f_N(\go)|^2
	\leq 4 \,
	\E \Big[ \gb_N^2  \, \cL_N(S,\tilde S) \,\,
	e^{ H_{\Lambda,\gb_N}(S)+H_{\Lambda,\gb_N}(\tilde S)} \Big] \,,
\end{align*}
and note that, arguing as in \eqref{eq:2momov}-\eqref{eq:2mom-overlap},
for every $\delta > 0$ we have, for all $N$ large enough,
\begin{align*}
	\bbE \, \E \Big[ \gb_N^2  \, \cL_N(S,\tilde S) \,\,
	e^{ H_{\Lambda,\gb_N}(S)+H_{\Lambda,\gb_N}(\tilde S)} \Big] &\leq
	\E \Big[ \gb_N^2  \, \cL_N(S, \tilde S) \,\,
	e^{(1+\gd)\gb_N^2 \, \cL_N(S,\tilde S)} \,\Big]\\
	&\leq \frac{1}{\delta} \,
	\E \Big[ e^{(1+2\gd) \,\gb_N^2 \, \cL_N(S,\tilde S)} \,\Big] \,,
\end{align*}
where we used the bound $x \le \frac{1}{\delta} e^{\delta x}$.
Thus, for all $N$ large enough we have
\begin{align*}
	\bbP\big(f_N(\go)\geq a\,,|\nabla f_N(\go)|> c\big)
	& \leq \frac{1}{c^2} \,\,\bbE \Big[ |\nabla f_N(\go)|^2 \,
	\ind_{\{ f_N(\go)\geq a \}} \Big]
	\leq \frac{4}{c^2} \,
	\frac{1}{\delta} \,\,
	\E \Big[ e^{(1+2\gd) \,\gb_N^2 \, \cL_N(S,\tilde S)} \,\Big] \,.
\end{align*}

Let us now define $\hat\beta' := \frac{1+\hat\beta}{2}$, so that $\hat\beta < \hat\beta' < 1$, and
define $\beta'_N := \hat\beta' / \sqrt{R_N}$, see \eqref{eq:betaN}.
Then we can fix $\delta = \delta_{\hat\beta} > 0$ small enough so that
$(1+2\delta)\beta_N^2 < \lambda(2\beta'_N) - 2 \lambda(\beta'_N)$
(note that $\lambda(2\beta) - 2 \lambda(\beta) \sim \beta^2$ as $\beta \to 0$),
hence by \eqref{eq:2momov} and \eqref{eq:2mom}
\begin{equation*}
	\bbP\big(f_N(\go)\geq a\,,|\nabla f_N(\go)|> c\big)
	\le \frac{4}{c^2} \,
	\frac{1}{\delta_{\hat\beta}} \, C_{\hat\beta'} \,.
\end{equation*}
Choosing $c=c_{\hat\gb}$ large enough, we can make the right hand side smaller than
$\theta_{\hat\beta}$, see \eqref{PL2}. 
Looking back at \eqref{eq:deba}, we see that \eqref{concentrate1} is proved.
\end{proof}

\subsection{Proof of equation \eqref{eq:2mom''}.} 
\label{sec:2mom''}
The quantity
$\bbE[\hat Z_{N,\gb_N}^A(x)^2]$ admits a representation
similar to the first line of \eqref{eq:2momZ}, without
the constant term $1$
and with the inner sum restricted to space-time points 
such that $(n_i, x_i) \not\in A_N^x$ for some
$i=1,\ldots, k$, i.e.\ either $n_i > N^{1-a_N}$
or $|x_i - x| \ge N^{\frac{1}{2} - \frac{a_N}{4}}$.
Since there are $k$ space-time points, for some $j=1,\ldots, k$ we must have
either $n_j - n_{j-1} > \frac{1}{k} N^{1-a_N}$
or $|x_j - x_{j-1}| \ge \frac{1}{k} N^{\frac{1}{2} - \frac{a_N}{4}}$
(we recall that $n_0 = 0$ and $x_0 = x$). Defining the new variables
$\ell_i := n_i - n_{i-1}$ and $z_i := x_i - x_{i-1}$,
and enlarging the range $0 < n_1 < \ldots < n_k \le N$
to $\ell_1, \ldots, \ell_k \in \{1,\ldots, N\}$, we can then bound
\begin{equation} \label{eq:2momZhat}
\begin{split}
	\bbE[\hat Z_{N,\gb_N}^A(x)^2]
	\le \sum_{k=1}^N (\sigma_N^2)^k 
	& \sumtwo{\ell_1, \ldots, \ell_k \in \{1,\ldots, N\}}
	{z_1, \ldots, z_k \in \Z^2} \
	\sum_{j=1}^k \,\Big(
	\ind_{\{\ell_j > \frac{1}{k} N^{1-a_N}\}}
	\\
	& \qquad
	+ \ind_{\{\ell_j \le \frac{1}{k} N^{1-a_N},
	\, |z_j| \ge \frac{1}{k} N^{\frac{1}{2} - \frac{a_N}{4}}\}} \Big)
	\, \prod_{i=1}^k q_{\ell_i}(z_i)^2 \,.
\end{split}
\end{equation}

We now switch the sum over $j$ with the double sum over $\ell_i, z_i$'s.
We can sum over all variables $z_j$'s with $i \ne j$, replacing
each kernel $q_{\ell_i}(z_i)^2$ by $q_{2 \ell_i}(0)$
(see \eqref{eq:RN}), and then sum $q_{2 \ell_i}(0)$ for all
$\ell_i$'s with $i\ne j$, which gives $(R_N)^{k-1}$ (see again \eqref{eq:RN}).
This yields
\begin{equation*}
\begin{split}
	\bbE[\hat Z_{N,\gb_N}^A(x)^2]
	\le \sum_{k=1}^N  (\sigma_N^2)^k \,R_N^{k-1}\,
	k  \sumtwo{\ell \in \{1,\ldots, N\}}
	{z \in \Z^2} \,\Big(
	\ind_{\{\ell > \frac{1}{k} N^{1-a_N}\}}
	+ \ind_{\{\ell \le \frac{1}{k} N^{1-a_N},
	\, |z| \ge \frac{1}{k} N^{\frac{1}{2} - \frac{a_N}{4}}\}} \Big)
	\, q_{\ell}(z)^2 \,.
\end{split}
\end{equation*}
We now consider separately the contributions of the two indicator functions.
\begin{itemize}
\item Recalling \eqref{eq:RN}, \eqref{eq:betaN}, \eqref{eq:sigma},
the contribution of $\{\ell > \frac{1}{k} N^{1-a_N}\}$ is controlled by
\begin{equation*}
\begin{split}
	\sum_{k=1}^N (\sigma_N^2)^k 
	\, R_N^{k-1}  \, k 
	& \sum_{\frac{1}{k} N^{1-a_N} < \ell \le N} q_{2\ell}(0)
	\le C
	\sum_{k=1}^N k \, (\hat\beta^2)^k \, \frac{1}{\log N}
	\sum_{\tfrac{1}{k} N^{1-a_N} <\ell \le N} \frac{1}{\ell} \\
	& \le C'
	\sum_{k=1}^N k \,(\hat\beta^2)^k \,
	\frac{a_N \log N + \log k}{\log N} 
	\le C' \bigg( \tilde C_{\hbeta} \, a_N + \hat C_{\hbeta} \, \frac{1}{\log N} \bigg)\,,
\end{split}
\end{equation*}
where $\tilde C_{\hbeta} := \sum_{k=1}^\infty k \, (\hat\beta^2)^k$
and $\hat C_{\hbeta} := \sum_{k=1}^\infty k \,
(\log k) \, (\hat\beta^2)^k$ are
finite, $\hat\beta$-dependent constants. This
contribution is consistent with \eqref{eq:2mom''}
(recall \eqref{eq:epsilon}).

\item The contribution of  $\{\ell \le \frac{1}{k} N^{1-a_N},
\, |z| \ge \frac{1}{k} N^{\frac{1}{2} - \frac{a_N}{4}}\}$ is given by
\begin{equation} \label{eq:leftto}
	\sum_{k=1}^N (\sigma_N^2)^k \,R_N^{k-1} 
	\, k \,
	\sum_{1 \le \ell \le \frac{1}{k} N^{1-a_N}}
	\sum_{|z| >  \frac{1}{k} N^{\frac{1}{2} - \frac{a_N}{4}}}
	q_\ell(z)^2 \,.
\end{equation}
Note that we can enlarge the range of the last sum to
$|z| > \theta \sqrt{\ell}$,
with $\theta = N^{\frac{a_N}{4}} / \sqrt{k}$.
Note that $\sup_{z\in\Z^2} q_\ell(z) \le c/\ell$, by Gnedenko's local limit theorem.
Then, by Gaussian estimates for the simple random walk
on $\Z^2$, there is $\eta > 0$ such that
\begin{equation*}
	\sum_{|z| > \theta \sqrt{\ell}} q_\ell(z)^2
	\le \frac{c}{\ell} \, \P(|S_\ell| > \theta \sqrt{\ell}) \le
	\frac{c}{\ell} \, e^{- \eta \, \theta^2} \,,
	\qquad \forall \ell \in \N \,, \ \forall \theta > 0 \,.
\end{equation*}
Then we can bound \eqref{eq:leftto} by a constant multiple of
\begin{equation} \label{eq:leftto2}
	\sum_{k=1}^N (\hat\beta^2)^k \, k \, \frac{1}{R_N}
	\, \sum_{1 \le \ell \le N} \frac{c}{\ell} \, e^{-\eta \, \theta^2}
	\le C \, \sum_{k=1}^N (\hat\beta^2)^k \, k \, e^{-\eta
	\frac{N^{\frac{a_N}{2}}}{k}} \,.
\end{equation}
We split the sum according to
$k \le (N^{\frac{a_N}{2}})^{1/2}$ and $k > (N^{\frac{a_N}{2}})^{1/2}$, getting
the bound
\begin{equation*}
	\bigg\{ \sum_{k=1}^\infty (\hat\beta^2)^k \, k \bigg\}
	\, e^{-\eta (N^{\frac{a_N}{2}})^{1/2}} \,+\,
	\bigg\{ \sum_{k=1}^\infty \hat\beta^k \, k \bigg\}
	\, \hat\beta^{(N^{\frac{a_N}{2}})^{1/2}} \,.
\end{equation*}
Both brackets are finite, $\hat\beta$-dependent constants,
while the other factors are both $o(a_N)$ as $N\to\infty$, by \eqref{eq:epsilon},
because $\hat\beta < 1$
and $N^{\frac{a_N}{2}} 
= \exp(\frac{1}{2} (\log N)^\gamma) \gg \log N$.
\end{itemize}
This
completes the proof of \eqref{eq:2mom''}.\qed

\section{Edwards-Wilkinson fluctuations for directed polymer}
\label{S:polymer}

In this section, we prove Theorem \ref{th:main_polymer}, which consists of proving Propositions \ref{prop:R}, \ref{prop:ZA}, \ref{prop:ZB}, and \ref{prop:hatZA} as described in Section~\ref{sec:methods}. The proofs are given in the following subsections.

\subsection{Proof of Proposition~\ref{prop:R}}
\label{sec:prop:R}

Recalling \eqref{eq:Rconv}, we need to show that
\begin{equation*}
	\frac{\log N}{N^2} \sum_{x,y\in \Z^2}
	\bbcov \big[O_N(x), O_N(y) \big]  \,\phi(\tfrac{x}{\sqrt{N}})  \phi(\tfrac{y}{\sqrt{N}})
	\ \xrightarrow[N\to\infty]{} \ 0 \,.
\end{equation*}
By translation invariance and Cauchy-Schwarz, it suffices to show that for any $x\in \Z^2$,
\begin{equation}\label{eq:goalR}
	(\log N) \, \bbE[O_N(x)^2]
	\ \xrightarrow[N\to\infty]{} \ 0 \,.
\end{equation}

We recall that $O_N(x)$ is defined in \eqref{first_approx_heur},
and in view of \eqref{decomposition1} we can write
\begin{equation*}
	O_N(x) = \log\bigg(1 + \frac{\hat Z_{N,\gb_N}^A(x)}{Z_{N,\gb_N}^A(x)}\bigg) -
	\frac{\hat Z_{N,\gb_N}^A(x)}{Z_{N,\gb_N}^A(x)}
\end{equation*}
We can bound, for a suitable constant $C < \infty$,
\begin{equation} \label{eq:boundlog}
	|\log (1+y) - y| \le C \cdot \begin{cases}
	\sqrt{\frac{|y|}{1+y}} & \text{if } -1 < y < 0 \\
	y^2 & \text{if } -\frac{1}{2} \le y \le \frac{1}{2} \\
	|y| & \text{if } 0 < y < \infty
	\end{cases} \,.
\end{equation}
The three domains are chosen to overlap on purpose: in fact, we will apply these inequalities
in the domains $(-1,-a_N^{2/7})$,
$[-a_N^{2/7}, a_N^{2/7}]$ and
$(a_N^{2/7}, \infty)$ (recall $a_N$ from \eqref{eq:epsilon}). We define 
\begin{equation*}
	D_N^\pm := \bigg\{ \pm\frac{\hat Z_{N,\gb_N}^A(x)}{Z_{N,\gb_N}^A(x)} > a_N^{2/7} \bigg\} \,,
	\qquad
	D_N := D_N^+ \cup D_N^-
	= \bigg\{ \bigg|\frac{\hat Z_{N,\gb_N}^A(x)}{Z_{N,\gb_N}^A(x)}\bigg| > a_N^{2/7} \bigg\} \,,
\end{equation*}
and we bound
\begin{equation} \label{eq:PDN}
\begin{split}
	\bbP(D_N) & \le \bbP \big( Z_{N,\gb_N}^A(x) < a_N^{1/7} \big)
	+ \bbP \big( |\hat Z_{N,\gb_N}^A(x)| > a_N^{3/7} \big) \\
	& \le a_N^{1/7} \,
	\bbE\big[Z_{N,\gb_N}^A(x)^{-1}\big] \,+\, a_N^{-6/7} \,\bbE\big[\hat Z_{N,\gb_N}^A(x)^2\big]
	\le \big( C_{2,\hat\beta} + C_{\hat\beta}  \big) \, a_N^{1/7} \,,
\end{split}
\end{equation}
thanks to \eqref{eq:-pmom'} and \eqref{eq:2mom''}.
Then by \eqref{eq:boundlog}
\begin{equation*}
	\frac{1}{C} \bbE[O_N(x)^2]
	 \le \bbE\bigg[
	\bigg( \frac{\hat Z_{N,\gb_N}^A(x)}{Z_{N,\gb_N}^A(x)} \bigg)^4
	\, \ind_{D_N^c} \bigg] +
	\bbE\bigg[
	\bigg( \frac{\hat Z_{N,\gb_N}^A(x)}{Z_{N,\gb_N}^A(x)} \bigg)^2
	\, \ind_{D_N^+} \bigg] +
	\bbE\bigg[
	 \frac{ \big|\hat Z_{N,\gb_N}^A(x)/Z_{N,\gb_N}^A(x) \big|}{1+  \hat Z_{N,\gb_N}^A(x)/Z_{N,\gb_N}^A(x) }
	\, \ind_{D_N^-} \bigg] ,
\end{equation*}
and given that
\begin{align*}
1+  \frac{\hat Z_{N,\gb_N}^A(x)}{Z_{N,\gb_N}^A(x) } = \frac{ Z_{N,\gb_N}(x)}{Z_{N,\gb_N}^A(x) },
\end{align*}
we can choose $p=p_{\hat\beta}>2$ close to $2$ as in \eqref{eq:pmom} such that
\begin{equation*}
\begin{split}
	\frac{1}{C}\bbE[O_N(x)^2]
	& \le \bbE\bigg[
	\bigg( \frac{\hat Z_{N,\gb_N}^A(x)}{Z_{N,\gb_N}^A(x)} \bigg)^4
	\, \ind_{D_N^c} \bigg] +
	\bbE\bigg[
	\bigg( \frac{\hat Z_{N,\gb_N}^A(x)}{Z_{N,\gb_N}^A(x)} \bigg)^2
	\, \ind_{D_N^+} \bigg] +
	\bbE\bigg[
	 \bigg| \frac{ \hat Z_{N,\gb_N}^A(x)}{Z_{N,\gb_N}(x)} \bigg|
	\, \ind_{D_N^-} \bigg] \\
	& \le a_N^{\frac{8}{7}} \,+\,
	\bbE\big[\hat Z_{N,\gb_N}^A(x)^p \,\big]^{\frac{2}{p}} \,\bigg(\,
	\bbE\big[Z_{N,\gb_N}^A(x)^{-\frac{2p}{p-2}} \, \ind_{D_N^+}\big]^{1-\frac{2}{p}}
	+ \bbE\big[Z_{N, \beta_N}(x)^{-2} \, \ind_{D_N^-}\big]^{\frac{1}{2}} \,\bigg) \\
	& \le a_N^{\frac{8}{7}} \,+\,
	C'_{\hat\beta} \, a_N \,\bigg(
	\bbE[Z_{N,\gb_N}^A(x)^{-\frac{4p}{p-2}}]^{\frac{1}{2}-\frac{1}{p}}\bbP(D_N)^{\frac{1}{2}-\frac{1}{p}} + \bbE[Z_{N, \beta_N}(x)^{-4}]^{\frac{1}{4}}
	\bbP(D_N)^{\frac{1}{4}}\bigg) \\
	& \le a_N^{\frac{8}{7}} \,+\,
	C'_{\hat\beta} \, a_N \, \bbP(D_N)^{\frac{1}{4}\wedge (\frac{1}{2}-\frac{1}{p})} 
	\le C'_{\hat\beta} a_N^{1+\frac{1}{7}(\frac{1}{2}-\frac{1}{p})},
\end{split}
\end{equation*}
where the second last inequality holds by
\eqref{eq:pmom}, \eqref{eq:-pmom} and \eqref{eq:-pmom'},
in the last inequality we applied \eqref{eq:PDN}, and
$C'_{\hat \beta}<\infty$ is a generic constant depending only on $\hat\beta$.  Recall from \eqref{eq:epsilon} that $a_N=(\log N)^{\gamma-1}$. We can then choose $\gamma \in (0, \gamma^*)$ with
$\gamma^*>0$ small enough such that
\begin{equation}\label{RN2_gamma}
	\bbE[O_N(x)^2]
	\le C'_{\hat\beta} \, a_N^{1+\frac{1}{7}(\frac{1}{2}-\frac{1}{p})}
	= C'_{\hat\beta} \, (\log N)^{-(1-\gamma)(1+\frac{1}{7}(\frac{1}{2}-\frac{1}{p}))} = o\big((\log N)^{-1}\big)\,.
\end{equation}
Therefore \eqref{eq:goalR} holds.\qed

\subsection{Proof of Proposition~\ref{prop:ZA}}
\label{sec:prop:ZA}

We need to show that
\begin{equation}\label{eq:goalZ'}
	\frac{\log N}{N^2} \sum_{x,y\in\Z^2}
	\bbcov \big[\log Z_{N,\gb_N}^A(x), \log Z_{N,\gb_N}^A(y) \big]  \,\phi(\tfrac{x}{\sqrt{N}}) \phi(\tfrac{y}{\sqrt{N}})
	\ \xrightarrow[N\to\infty]{} \ 0 \,.
\end{equation}
We recall that $Z_{N,\gb_N}^A(x)$ depends only on the disorder within set  $A_N^x$, defined in \eqref{def:setA},
hence $Z_{N,\gb_N}^A(x)$ and $Z_{N,\gb_N}^A(y)$ are independent for $|x-y| > 2 N^{\frac{1}{2} - \frac{a_N}{4}}$.
 By Cauchy-Schwarz and \eqref{eq:momlog}, we can bound the left hand side of \eqref{eq:goalZ'} as follows:
\begin{equation*}
\begin{split}
	&\,C_{2,\hat\beta }\frac{\log N}{N^2}  \!\!\!\!\!\!
	\sum_{x,y\in\Z^2: \ |y-x| \le 2 N^{\frac{1}{2} - \frac{a_N}{4}}}
	 \,\phi(\tfrac{x}{\sqrt{N}}) \phi(\tfrac{y}{\sqrt{N}})
	\le \,c \, C_{2,\hat\beta}  \frac{\log N}{N^2} \, \,  N^{1 - \frac{a_N}{2}} \, 
	|\phi|_\infty \, \sum_{x\in\Z^2} |\phi(\tfrac{x}{\sqrt{N}})|\\
	& \leq  c' \, C_{2,\hat\beta} \, (\log N) \, N^{-\frac{a_N}{2}} \, |\phi|_\infty \, 
	|\phi|_{L^1(\R^2)}
	= c' \, C_{2,\hat\beta} \, e^{\log (\log N) -\frac{1}{2} (\log N)^\gamma}
	\, |\phi|_\infty \,  |\phi|_{L^1(\R^2)}
	\xrightarrow[N\to\infty]{} 0 \,,
\end{split}
\end{equation*}
where $c, c'$ are generic constants,
and the last equality holds by definition of $a_N$ in \eqref{eq:epsilon}.\qed

\subsection{Proof of Proposition \ref{prop:ZB}}
\label{sec:proof:ZB}

We need to show that
\begin{align}\label{main_approx}
	\frac{\sqrt{\log N}}{N} \sum_{x\in\Z^2}
	\frac{\hat Z_{N,\gb_N}^A(x)}{Z_{N,\gb_N}^A(x)} \, \phi(\tfrac{x}{\sqrt{N}})
	\; -\; \frac{\sqrt{\log N}}{N} \sum_{x\in\Z^2} (Z_{N,\gb_N}^{B^\geq}(x)-1) \, 
	\phi(\tfrac{x}{\sqrt{N}}) \ \xrightarrow[\ N\to\infty\ ]{L^1(\bbP)} \ 0 \,.
\end{align}

 \begin{figure}
\hskip -0.2cm
\begin{minipage}[b]{.33\linewidth}
\centering
\begin{tikzpicture}[scale=0.5]
\draw (0,-5)--(0,5); \draw (0,-2)--(6,-2)--(6,2)--(0,2);
\draw[dashed] (6,-5)--(6,5); \draw[dashed] (7,-5)--(7,5); \draw[<->] (7.1,-4.5)--(12,-4.5);
\draw[<->] (6.1,4.5)--(12,4.5);\node at (9,4) {\scalebox{0.6}{$B_N$}};
\node at (1,-1.5) {\scalebox{0.6}{$A_N^x$}};\node at (1,-5) {\scalebox{0.6}{$C_N^x$}};
\node at (9.5,-5) {\scalebox{0.6}{$B_N^{\geq}$}};
\node at (-1,0) {\scalebox{0.6}{$(0,x)$}}; \draw  [fill] (0, 0)  circle [radius=0.1]; \draw  [fill] (1, 1)  circle [radius=0.1];\draw  [fill] (2.5, 3.5)  circle [radius=0.1];
\draw  [fill] (3,-0.5)  circle [radius=0.1]; \draw  [fill] (4, -2.5)  circle [radius=0.1]; \draw  [fill] (5.7, -1)  circle [radius=0.1]; \draw  [fill] (8,1.5)  circle [radius=0.1];
\draw  [fill] (10, -1)  circle [radius=0.1];
\draw (0,0) to [out=60,in=-160] (1,1) to [out=60,in=-160] (2.5, 3.5) to [out=-30,in=90] (3,-0.5) to [out=-80,in=-150] (4, -2.5) to [out=30,in=100] (5.7, -1) to
[out=30, in=-100] (8,1.5) to [out=30, in=100] (10,-1) to [out=90, in=180] (11,0); \node at (12,0) {\scalebox{0.6}{$\cdots$}};
\end{tikzpicture}
\subcaption{Partition function $Z_{N,\gb_N}^{A,C}(x).$\label{figure:ZAC}}
\end{minipage}
\qquad\qquad\qquad
\begin{minipage}[b]{.33\linewidth}
\centering
\begin{tikzpicture}[scale=0.5]
\draw (0,-5)--(0,5); \draw (0,-2)--(6,-2)--(6,2)--(0,2);
\draw[<->] (7.1,-4.5)--(12,-4.5);
\draw[<->] (6.1,4.5)--(12,4.5);\node at (9,4) {\scalebox{0.6}{$B_N$}};
\node at (9.5,-5) {\scalebox{0.6}{$B_N^{\geq}$}};
\draw[dashed] (6,-5)--(6,5); \draw[dashed] (7,-5)--(7,5);
\node at (1,-1.5) {\scalebox{0.6}{$A_N^x$}};\node at (1,-5) {\scalebox{0.6}{$C_N^x$}};
\node at (-1,0) {\scalebox{0.6}{$(0,x)$}}; \draw  [fill] (0, 0)  circle [radius=0.1]; \draw  [fill] (1, 1)  circle [radius=0.1];
\draw  [fill] (3,-0.5)  circle [radius=0.1]; \draw  [fill] (5.5, -1)  circle [radius=0.1]; \draw  [fill] (8,3)  circle [radius=0.1];
\draw  [fill] (10, -1)  circle [radius=0.1];
\draw (0,0) to [out=60,in=-160] (1,1) to  [out=-30,in=90] (3,-0.5) to [out=30,in=100] (5.5, -1) to
[out=80, in=180] (8,3) to [out=-10, in=100] (10,-1) to [out=90, in=180] (11,0); \node at (12,0) {\scalebox{0.6}{$\cdots$}};
\end{tikzpicture}
\subcaption{ Partition function $Z_{N,\gb_N}^{A,B}$ .
\label{figure:ZAB}}
\end{minipage}
\caption{The figures depict the chaos expansions of 
$ Z_{N,\gb_N}^{A,C}(x), Z_{N,\gb_N}^{A,B}(x)$. Each term in the expansion 
for $Z_{N,\gb_N}^{A,C}(x)$ must include disorder from $C_N^x$; while each term 
in the expansion for $Z_{N,\gb_N}^{A,B}(x)$ contain
only disorder from $A_N^x\cup B_N$, with at 
least some disorder from $B_N$.\label{fig:multi_repAB,AC}}
\end{figure}
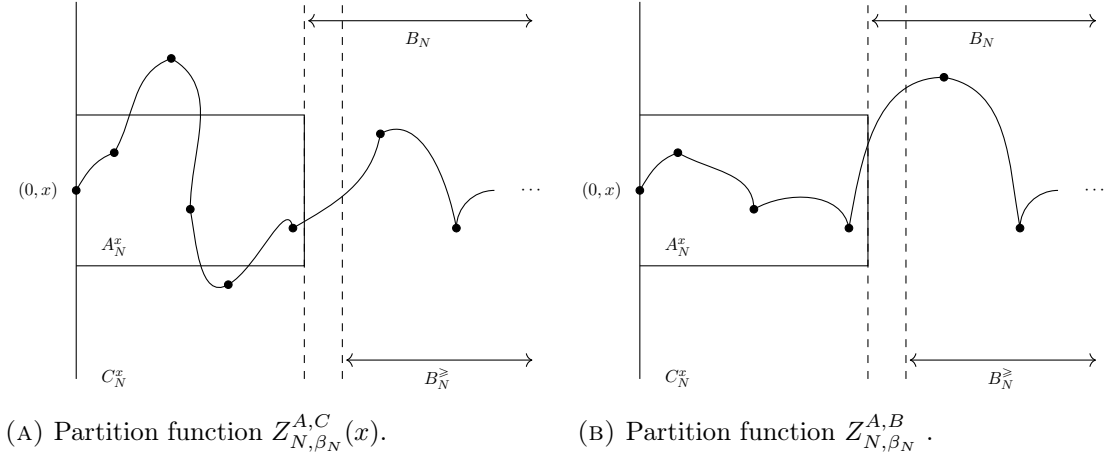

We recall that $B_N^\geq := \big(\, (N^{1-9a_N/40}, N] \cap \N \,\big) \times \Z^2$, 
see \eqref{eq:B}.
We define new subsets
\begin{gather}
	\label{eq:B2}
	B_N := \big(\, (N^{1-a_N}, N] \cap \N \,\big) \times \Z^2 \,, \\
	\label{eq:C}
	C_N^{\,x}:= \big\{ (n,z)\in \N\times \Z^2 \colon n \leq N^{1-a_N}\,,\, 
	|z-x|\geq N^{\tfrac{1}{2}-\tfrac{a_N}{4}} \big\} \,,
\end{gather}
and we introduce new ``partition functions'':
\begin{align}
	Z_{N,\gb_N}^{\,A,C}(x) &:= 
	\sum_{\tau \subset\{1,\ldots, N\} \times \Z^2:
	\ \tau \cap C_N^x\neq \emptyset} \,\,\sigma_N^{|\tau|} \, 
	\bq^{(0,x)}(\tau) \, \bxi(\tau) \,, \label{defAC}\\
	Z_{N,\gb_N}^{\,A,B}(x) &:= \sum_{\tau \subset  \,A_N^x\cup B_N : \
	\tau \cap B_N\neq \emptyset}
	 \,\,\sigma_N^{|\tau|} \, \bq^{(0,x)}(\tau) \, \bxi(\tau) \,,\label{defAB}
\end{align}
similar to the polynomial chaos expansions for $Z_{N, \gb_N}, Z_{N, \gb_N}^A$ 
and $Z_{N, \gb_N}^{B^\geq}$ in \eqref{Z-chaos}--\eqref{eq:ZAB}. See Figure \ref{fig:multi_repAB,AC} for a graphical representation of the chaos expansions.

\smallskip

Recall that $A_N^x$ was defined in \eqref{def:setA},
and note that $(\{1,\ldots,N\} \times \Z^2) \, \setminus \, A_N^x
= C_N^x \cup B_N$.
We can then decompose $\hat Z_{N,\gb_N}^A (x)$, defined in \eqref{decomposition1}, 
as follows:
\begin{align}\label{hatA_deco}
\hat Z_{N,\gb_N}^A (x) = Z_{N,\gb_N}^{A,B}(x) +  Z_{N,\gb_N}^{A,C}(x).
\end{align}

We split the sequel in three steps. The first step is:
\begin{enumerate}
\item\label{it:1} We will show that the contribution of the term $Z_{N, \gb_N}^{A,C}$ 
is negligible for \eqref{main_approx}.
\end{enumerate}
To treat the term $Z_{N,\gb_N}^{A,B}(x)$, we decompose 
its chaos expansion \eqref{defAB} according to the last point $(t,w)$ of $\tau$ 
that lies in $A_N^x$
and the first point $(r,z)$ of $\tau$ that lies in $B_N$:
\begin{align}\label{AB-exp}
	Z_{N,\gb_N}^{A,B}(x) =\sumtwo{(t,w) \,\in\, 
	\{(0,x)\} \cup A_N^x}{(r,z)\,\in\, B_N} Z_{0,t,\gb_N}^A(x,w) 
	\cdot\, q_{r-t}(z-w) \cdot \sigma_N \, \xi_{r,z} \cdot
	\,\,Z_{r,N,\gb_N} (z),
\end{align}
where $Z_{0,t,\gb_N}^A(x,w)$ is the ``point-to-point'' partition function
from $(0,x)$ to $(t,w)$, defined by
$Z_{0,t,\gb_N}^A(x,w) := 1$ if $(t,w) = (0,x)$ and by
\begin{align}\label{AB-exp2}
	Z_{0,t,\gb_N}^A(x,w) := \sum_{\tau \subset A_N^x \cap ( [0,t]\times \Z^2):
	\ \tau \ni (t,w)} \sigma_N^{|\tau|} \, \bq^{(0,x)}(\tau) \, \bxi(\tau)
	\qquad \text{if } t > 0 \,,
\end{align}
while $Z_{r,N,\gb_N} (z)$ is the
``point-to-plane'' partition function starting at $(r,z)$ and 
running until time $N$:
\begin{equation} \label{eq:ZrN}
	Z_{r,N,\gb_N}(z) := \sum_{\tau \subset \{r+1, \ldots, N\} \times \Z^2}
	\sigma_N^{|\tau|} \, \bq^{(r,z)}(\tau) \, \bxi(\tau) \,.
\end{equation}
The next steps are:
\begin{enumerate}\addtocounter{enumi}{1}
\item\label{it:2} We will show that in \eqref{AB-exp} 
the contribution from $r<N^{1-9a_N/40}$ is negligible for \eqref{main_approx}.

\item\label{it:3} 
We will show that in \eqref{AB-exp} we can replace the kernel $q_{r-t}(z-w)$ 
by $q_{r}(z-x)$, i.e.\ the transition kernel from $(0,x)$ to $(r,z)$,
because their difference is negligible for \eqref{main_approx}.
\end{enumerate}

Finally, note that when we restrict the sum in \eqref{AB-exp} to $r\ge N^{1-9a_N/40}$,
i.e.\ to $(r,z) \in B_N^\ge$ (recall \eqref{eq:B}), and we replace 
$q_{r-t}(z-w)$ by $q_{r}(z-x)$, \emph{the right hand side
of \eqref{AB-exp} becomes exactly $Z_{N,\beta_N}^{A}(x) \, (Z_{N,\beta_N}^{B^\geq}(x)-1)$}
(recall \eqref{eq:ZAB}). This completes the proof of \eqref{main_approx}.

\smallskip

It remains to prove the three steps stated above.

\medskip

\noindent
{\bf Step \eqref{it:1}.}
We show that the contribution of $Z_{N, \gb_N}^{A,C}$ in \eqref{hatA_deco} to 
\eqref{main_approx} is negligible, that is,
\begin{equation} \label{eq:stuno}
\frac{\sqrt{\log N}}{N} \sum_{x\in\Z^2}
	\frac{Z_{N,\gb_N}^{\, A,C}(x)}{Z_{N,\gb_N}^A(x)} \, \phi(\tfrac{x}{\sqrt{N}})
	\ \xrightarrow[N\to\infty]{L^2(\bbP)} \ 0.
\end{equation}

Since the chaos expansion of $ Z_{N,\gb_N}^{\,A,C}(x)$ in \eqref{defAC}
contains disorder $\xi$ outside $A_N^x$, not contained in the expansion of
$Z_{N,\gb_N}^A(x)$, we have that
$\bbE \big[ Z_{N,\gb_N}^{\,A,C}(x) / Z_{N,\gb_N}^A(x) \big] = 0$
thus $L^2(\bbP)$ and variance computations are equivalent. We then have
\begin{align*}
\bbvar \Big( \frac{\sqrt{\log N}}{N} \sum_{x\,\in\,\Z^2}
	\frac{ Z_{N,\gb_N}^{\,A,C}(x)}{Z_{N,\gb_N}^A(x)} \, \phi(\tfrac{x}{\sqrt{N}}) \Big)
=\frac{\log N}{N^2} \sum_{x,y \,\in \,\Z^2}
 \bbE \Bigg[ \frac{ Z_{N,\gb_N}^{\,A,C}(x)}{Z_{N,\gb_N}^A(x)} \,\cdot\,  \frac{ Z_{N,\gb_N}^{\,A,C}(y)}{Z_{N,\gb_N}^A(y)} \Bigg]
 \, \phi(\tfrac{x}{\sqrt{N}}) \phi(\tfrac{y}{\sqrt{N}}). \notag
\end{align*}
By Cauchy-Schwarz, we can further bound this as follows, for some constant $c$:
 \begin{align}\label{var_AB}
 & \,\,\frac{\log N}{N^2}   \bbE\Big[\Big( \frac{ Z_{N,\gb_N}^{\,A,C}(0)}{Z_{N,\gb_N}^A(0)}\Big)^2\Big] \sum_{x,y \in\Z^2} \, |\phi(\tfrac{x}{\sqrt{N}})| |\phi(\tfrac{y}{\sqrt{N}})|
 \leq  \, c \, \log N  \cdot |\phi|_{L^1(\R^2)}^2 \, \bbE\Big[\Big( \frac{ Z_{N,\gb_N}^{\,A,C}(0)}{Z_{N,\gb_N}^A(0)}\Big)^2\Big] \notag\\
& \qquad \qquad \leq \, c\, \log N  \cdot |\phi|_{L^1(\R^2)}^2\, \bbE\Big[\,Z_{N,\gb_N}^{\,A,C}(0) ^{2p}\,\Big]^{1/p} \bbE\Big[\frac{1}{ Z_{N,\gb_N}^A(0)^{2q} }\Big]^{1/q},
\end{align}
where in the last step we used H\"older inequality with parameters $(p,q)$ with $p^{-1}+q^{-1}=1$, and $p$ will be chosen sufficiently close to $1$,
to be determined below. The term $\bbE\big[ Z_{N,\gb_N}^A(0)^{-2q} \big]^{1/q}$ can be uniformly bounded by the negative moment estimate \eqref{eq:-pmom'}.

We can use hypercontractivity, see \eqref{eq:hyper}, to bound
\begin{align*}
\bbE\Big[\,Z_{N,\gb_N}^{\,A,C}(0) ^{2p}\,\Big]^{1/p} \leq
\sum_{\tau \subset\{1,\ldots, N\} \times \Z^2:
\ \tau \cap C_N^0\neq \emptyset} \,\,(c_{2p } \, \sigma_N)^{2|\tau|} \, \bq^{(0)}(\tau)^2 \,.
\end{align*}
The right hand side is the second moment of the
partition function, see \eqref{eq:2momZ}, except that
$\sigma_N$ is replaced by $c_{2p} \sigma_N$ (which
corresponds asymptotically to replacing $\hat\beta$ 
by $\hat\beta' := c_{2p} \, \hat\beta$,
see \eqref{eq:sigma} and \eqref{eq:betaN}) and the random walk $S$ must 
satisfy $\max\big\{|S_n|\colon n<N^{1-a_N} \big\}> 
N^{1/2-a_N/4} $. In particular, 
recalling \eqref{eq:2mom-overlap} and \eqref{eq:LN},
this can be bounded by
\begin{align}\label{Step1overlap}
\E\Big[e^{(1+o(1))\,(c_{2p}\gb_N)^2 \,
\cL_{N^{1-a_N}} \left( S^{(1)}, \, S^{(2)} \right)}\,;\,
\max_{n\leq N^{1-a_N}} |S^{(i)}_n| > N^{\tfrac{1}{2}-\tfrac{a_N}{4}},\,\,\text{for}\,\,i=1,2 \Big],
\end{align}
where $S^{(1)}, S^{(2)}$ are two independent random walk copies.
This is bounded via H\"older by
\begin{align*}
\E\Big[e^{(1+o(1))\,p\,(c_{2p}\gb_N)^2 
\, \cL_{N^{1-a_N}} \left( S^{(1)}, \, S^{(2)} \right)} \Big]^{1/p} \,
\P\Big( \max_{n\leq N^{1-a_N}} |S_n| > N^{\tfrac{1}{2}-\tfrac{a_N}{4}} \Big)^{2/q}.
\end{align*}
We can now choose $p>1$ sufficiently close to $1$ so that $\sqrt{p} \,c_{2p} \,\hat\gb<1$, 
i.e.\ still subcritical,
which is possible because $\lim_{p\to1}c_{2p}=1$, see \eqref{eq:cp1}. 
Hence the expectation above is uniformly bounded in $N$ as shown in 
Section~\ref{S:2ndmom}. On the other hand, standard moderate deviation estimates for the simple symmetric random walk show that
\begin{align*}
	\P\Big( \max_{n\leq N^{1-a_N}} |S_n| > N^{\tfrac{1}{2}-\tfrac{a_N}{4}} \Big) 
	\leq  \exp\big(-c \, N^{a_N/2}\big) =
	\exp\big(-c\, e^{(\log N)^\gamma/2}\big),
\end{align*}
where we recall that $a_N=(\log N)^{\gamma-1}$, see \eqref{eq:epsilon}. 
Inserting these estimates in \eqref{var_AB}, we get
\begin{align*}
\bbvar \Big( \frac{\sqrt{\log N}}{N} \sum_{x\in\Z^2}  \frac{ Z_{N,\gb_N}^{\,A,C}(x)}{Z_{N,\gb_N}^A(x)} \, \phi(\tfrac{x}{\sqrt{N}}) \Big)
\leq c\, \log N \cdot |\phi|_{L^1(\R^2)}^2 \, \exp\Big({-c e^{(\log N)^\gamma/2}}\Big) 
\xrightarrow[N\to\infty]{} 0 \,.
\end{align*}

\medskip

\noindent
{\bf Step \eqref{it:2}.} 
We show that in the chaos expansion \eqref{AB-exp} for $Z_{N,\gb_N}^{A,B}(x)$, 
the contribution from $(r,z)$ with $r<N^{1-9a_N/40}$ is negligible for \eqref{main_approx}.
The contribution we are after is
\begin{align}\label{defAB<}
	Z_{N,\gb_N}^{A,B^<}(x) :=\!\!\!\!\!\!\!\! \!\!
	\sumtwo{(t,w)\in A_N^x}{(r,z)\in B_N \colon r<N^{1-9a_N/40}}
	Z_{0,t,\gb_N}^A(x,w) \cdot\, q_{r-t}(z-w) \, \cdot \sigma_N \, \xi_{r,z} 
	\cdot \,\,Z_{r,N,\gb_N} (z) \,,
\end{align}
and we want to show that
\begin{equation}\label{eq:goalit2}
	\bbE \Bigg[ \bigg( \frac{\sqrt{\log N}}{N} \sum_{x\in \Z^2} 
	\,\phi(\tfrac{x}{\sqrt{N}}) \frac{Z_{N,\gb_N}^{A,B^<}(x)}{Z_{N,\gb_N}^A(x)}
	\bigg)^2 \Bigg] \,\xrightarrow[N\to\infty]\, 0 \,.
\end{equation}

The left hand side of \eqref{eq:goalit2} equals
\begin{align}\label{AB<}
	\frac{\log N}{N^2} \!\!\! \sum_{x,y\in \Z^2} \!\!\! 
	\phi(\tfrac{x}{\sqrt{N}}) \phi(\tfrac{y}{\sqrt{N}})
	\bbE\Bigg[\frac{Z_{N,\gb_N}^{A,B^<}(x)}{Z_{N,\gb_N}^A(x)} \cdot 
	\frac{Z_{N,\gb_N}^{A,B^<}(y)}{Z_{N,\gb_N}^A(y)} \Bigg].
\end{align}
We can restrict the summation over $x,y$ to $|x-y|>N^{\tfrac{1}{2}-\tfrac{a_N}{10}}$.
 Indeed, in the complementary regime, we first bound the expectation in \eqref{AB<} by Cauchy-Schwarz and obtain the bound
\begin{align*}
&\frac{\log N}{N^2} \bbE\Bigg[\Bigg( \frac{Z_{N,\gb_N}^{A,B^<}(0)}{Z_{N,\gb_N}^A(0)} \Bigg)^2\Bigg]
\sum_{|x-y|\leq N^{\tfrac{1}{2}-\tfrac{a_N}{10}}} \,|\phi(\tfrac{x}{\sqrt{N}})| \,|\phi(\tfrac{y}{\sqrt{N}})|  \\
&\qquad\qquad\qquad\qquad \qquad\qquad\qquad
\leq  N^{1-\tfrac{a_N}{5}} \,\,\frac{\log N}{N} \bbE\Bigg[\Bigg( \frac{Z_{N,\gb_N}^{A,B^<}(0)}{Z_{N,\gb_N}^A(0)} \Bigg)^2\Bigg] \,|\phi|_\infty \,|\phi|_{L^1(\R^2)} \\
&\qquad\qquad\qquad\qquad \qquad\qquad\qquad
=(\log N )\, e^{-\tfrac{1}{5}(\log N)^\gamma} \bbE\Bigg[\Bigg( \frac{Z_{N,\gb_N}^{A,B^<}(0)}{Z_{N,\gb_N}^A(0)} \Bigg)^2\Bigg] \,|\phi|_\infty \,|\phi|_{L^1(\R^2)},
\end{align*}
which goes to zero as $N\to\infty$, since expectation can be bounded via H\"older with an exponent $p$ for $Z_{N,\gb_N}^{A,B^<}(0)^2$ chosen sufficiently close to one, so that the hypecontractivity bound \eqref{eq:pmom} can be applied, while the negative moment $\bbE\big[ Z_{N,\gb_N}^A(0)^{-2q}\big]$ can be bounded by \eqref{eq:-pmom'}. The argument is the same as that for \eqref{var_AB} and we omit the details.

To deal with \eqref{AB<} when $(x,y)\in I^>:=\{x,y\in\Z^2\colon |x-y|>N^{\tfrac{1}{2}-\tfrac{a_N}{10}}\}$, we use the chaos expansion for $Z_{N,\gb_N}^{A,B^<}$, \eqref{defAB<}, and write \eqref{AB<} in this case as follows (recall that $\bbE[\xi^2]=1$):
\begin{align}\label{varAB}
& \frac{\sigma_N^2 \log N}{N^2}\,  \sum_{x,y\in I^>} \phi(\tfrac{x}{\sqrt{N}})  \phi(\tfrac{y}{\sqrt{N}})
\sumtwo{(t,w)\in A_N^x }{ (s,v)\in A_N^y }
  \bbE\Bigg[  \frac{Z_{0,t,\gb_N}^A(x,w)   }{Z_{N,\gb_N}^A(x)} \,\Bigg]\,\,\bbE\Bigg[\frac{ Z_{0,s,\gb_N}^A(y,v) }{Z_{N,\gb_N}^A(y)} \Bigg] \notag\\
&\qquad\qquad \qquad \times \sum_{ (r,z)\in B_N \colon  r<N^{1-9a_N/40}} q_{r-t}(z-w) q_{r-s}(z-v)  \,
\bbE\big[\,Z_{r,N,\gb_N}(z)^2\, \big]
\end{align}
where the first point $(r,z)\in B_N$ in the expansion for $Z_{N,\gb_N}^A(x)$ and $Z_{N,\gb_N}^A(y)$ must match because an unmatched $(r,z)$ gives $\bbE[\xi_{r,z}]=0$, and we used the independence between
\begin{align*}
\frac{Z_{0,t,\gb_N}^A(x,w)   }{Z_{N,\gb_N}^A(x)}  ,\qquad \frac{ Z_{0,s,\gb_N}^A(y,v) }{Z_{N,\gb_N}^A(y)}, \qquad \text{and}\qquad  Z_{r,N,\gb_N}(z),
\end{align*}
because they depend on disorder in the disjoint regions $A_N^x$, $A_N^y$ and $B_N$.

We can simplify \eqref{varAB} by noticing that $\bbE\big[Z_{r,N,\gb_N}(z)^2\big]$ is independent of $z$ and that $\sum_z q_{r-t}(z-w) q_{r-s}(z-v) = q_{2r-t-s}(w-v)$. Thus we can write it as
\begin{align}\label{varAB2}
&\frac{\sigma_N^2 \log N}{N^2}\, \sum_{x,y\in I^>} \phi(\tfrac{x}{\sqrt{N}})  \phi(\tfrac{y}{\sqrt{N}})
\sumtwo{(t,w)\in A_N^x }{ (s,v)\in A_N^y }
  \bbE\Bigg[  \frac{Z_{0,t,\gb_N}^A(x,w)   }{Z_{N,\gb_N}^A(x)}  \,\Bigg]\,\,\bbE\Bigg[ \frac{ Z_{0,s,\gb_N}^A(y,v) }{Z_{N,\gb_N}^A(y)} \Bigg] \notag\\
&\qquad\qquad \qquad \times \sum_{   N^{1-a_N} < r<N^{1-9a_N/40}} q_{2r-t-s}(w-v)   \,
\bbE\big[\,Z_{r,N,\gb_N}(0)^2\, \big]
\end{align}
Note that $\bbE\big[Z_{r,N,\gb_N}(0)^2 \big]\leq \bbE\big[Z_{N,\gb_N}(0)^2 \big]\leq C_{\hat\beta}$ uniformly in $N$ by \eqref{eq:2mom}. Moreover,
\begin{align*}
\Big|\,\bbE\Bigg[  \frac{Z_{0,t,\gb_N}^{A}(x,w)   }{Z_{N,\gb_N}^A(x)}  \Bigg]  \,\Big|
&\leq \bbE\Big[  Z_{0,t,\gb_N}^{A}(x,w)^2 \Big]^{1/2} \,\, \bbE\Bigg[  \frac{1  }{Z_{N,\gb_N}^A(x)^2}  \Bigg]^{1/2}\leq C_{2,\hat\gb}\, \bbE\Big[  Z_{0,t,\gb_N}^{A}(x,w)^2 \Big]^{1/2},
\end{align*}
where the constant $C_{2,\hat\gb}$ comes from the negative moment bound \eqref{eq:-pmom'}. The same bound holds with $(x,t,w)$ replaced by $(y,s,v)$. Therefore \eqref{varAB2} can be bounded by
\begin{align*}
C_{\hat\gb} C_{2,\hat\gb} \frac{\sigma_N^2 \log N}{N^2}\,\sum_{x,y\in I^>} |\phi(\tfrac{x}{\sqrt{N}})|  |\phi(\tfrac{y}{\sqrt{N}})|
& \sumtwo{(t,w)\in A_N^x }{ (s,v)\in A_N^y }
  \bbE\big[ Z_{0,t,\gb_N}^A(x,w)^2  \,\big]^{1/2}\,\,\bbE\big[  Z_{0,s,\gb_N}^A(y,v)^2  \big]^{1/2}  \notag\\
& \times \sum_{   N^{1-a_N} < r<N^{1-9a_N/40}} q_{2r-t-s}(w-v)   .
\end{align*}
By our definitions of
$\sigma_N$ and $\beta_N$ in \eqref{eq:sigma} and \eqref{eq:betaN},
we have $\sigma_N^2 \, \log N =O(1)$.
Applying Cauchy-Schwarz for the sum over $(t,w)$ and $(s,v)$, we obtain the bound
\begin{align*}
& \frac{C}{N^2}  \,\sum_{x,y\in I^>} |\phi(\tfrac{x}{\sqrt{N}})|  |\phi(\tfrac{y}{\sqrt{N}})|  \,\,\,
\Big(\sum_{(t,w)\in A_N^x } \bbE\big[  Z_{0,t,\gb_N}^A(x,w)^2  \,\big] \sum_{ (s,v)\in A_N^y } \bbE\big[  Z_{0,s,\gb_N}^A(y,v)^2  \big] \Big)^{1/2}  \notag\\
&\hskip 4cm\times
\sum_{   N^{1-a_N} < r<N^{1-9a_N/40}} \Big(\sum_{(t,w)\in A_N^x \,,\,(s,v)\in A_N^y} q_{2r-t-s}(w-v)^2\Big)^{1/2} \,.
\end{align*}
We next observe that $\sum_{(t,w)\in A_N^x } \bbE\big[  Z_{0,t,\gb_N}^A(x,w)^2  \,\big]
= \bbE\big[  Z_{N,\gb_N}^A(x)^2  \,\big] \le C_{\hat\beta}$,
see \eqref{AB-exp2}, \eqref{Z-chaos} and \eqref{eq:2mom'}, and similarly
for the sum over $(s,v)$. This leads to the bound
\begin{align}\label{varAB2.2}
&\frac{C\, C_{\hat\gb}}{N^2}  \sum_{x,y\in I^>} |\phi(\tfrac{x}{\sqrt{N}})|  |\phi(\tfrac{y}{\sqrt{N}})|
\!\!\! \sum_{   N^{1-a_N} < r<N^{1-9a_N/40}} \Big(\sum_{(t,w)\in A_N^x \,,\,(s,v)\in A_N^y} 
\!\!\! q_{2r-t-s}(w-v)^2\Big)^{1/2} \,.
\end{align}
Since $|x-y|>N^{\tfrac{1}{2}-\tfrac{a_N}{10}}$ and $|x-w|, |y-v|\leq N^{\frac{1}{2}-\frac{a_N}{4}}$, we have $|w-v|> \tfrac{1}{2}N^{\tfrac{1}{2}-\tfrac{a_N}{10}}$. Given $r<N^{1-9a_N/40}$, we then have
$$
q_{2r-t-s}(w-v)^2 \,\,\leq \,\,c \exp\Big(-\frac{|w-v|^2}{(2r-t-s)}\Big) \,\,\leq \,\, c \exp\Big(-c N^{a_N/40}\Big) \,\,
=\,\, \exp\big( - ce^{\,(\log N)^\gamma / 40} \, \big).
$$
The sums over $(t,w), (s,v)$ and $r$ give only a polynomial factor in $N$, 
and hence \eqref{varAB2.2} can be bounded by
\begin{align*}
 \frac{c}{N^2}\, \sum_{x,y\in I^>}  \phi(\tfrac{x}{\sqrt{N}}) \,\phi(\tfrac{y}{\sqrt{N}}) \, \cdot \, N^{3} \exp\big( - ce^{\,(\log N)^\gamma / 40} \,\big)
\xrightarrow[N\to\infty]{} 0 \,.
\end{align*}
This proves \eqref{eq:goalit2} and completes the step.

\medskip

\noindent
{\bf Step \eqref{it:3}.} 
Let $Z_{N,\gb_N}^{A,B^\geq}(x)$ be defined as in \eqref{defAB<}
but with the constraint $r\geq N^{1-9a_N/40}$ instead
of $r < N^{1-9a_N/40}$, i.e. with the sum over $(r,z) \in B_N^\geq$
instead of $B_N$ (recall \eqref{eq:B}):
\begin{align}\label{defABgeq}
	Z_{N,\gb_N}^{A,B^\geq}(x) := 
	\sum_{(t,w)\in A_N^x, \, (r,z)\in B_N^\geq}
	Z_{0,t,\gb_N}^A(x,w) \cdot\, q_{r-t}(z-w) \, \cdot \sigma_N \, \xi_{r,z} 
	\cdot \,\,Z_{r,N,\gb_N} (z) \,.
\end{align}
In view of \eqref{main_approx}, we focus on the averaged quantity
\begin{equation} \label{eq:decon}
	\frac{\sqrt{\log N}}{N} \sum_{x\in\Z^2}
	\frac{Z_{N, \gb_N}^{A,B\geq}(x)}{Z_{N,\gb_N}^A(x)} \, \phi(\tfrac{x}{\sqrt{N}}) \,.
\end{equation}
We will show that replacing in \eqref{defABgeq} the kernel $q_{r-t}(z-w)$
by $q_{r}(z-x)$ has a negligible effect on \eqref{eq:decon}, 
in the sense that the difference
tends to zero in $L^1(\bbP)$.

\smallskip

We introduce the notation (recall \eqref{eq:B})
\begin{align*}
	B_{N}^\geq (x) :=\big\{(r,z) \in B^\geq_N\colon \, |z-x|<r^{\tfrac{1}{2}+\tfrac{a_N}{80}}\big\}.
\end{align*}
Recall that $g_t(\cdot)$ denotes the heat kernel on $\R^2$, see \eqref{eq:g}.
By a refined local limit theorem for the simple random walk, see 
Theorem~2.3.11 in \cite{LL10}, we have
that for $(r,z) \in B_N^\geq (x)$,
\begin{align*}
q_r(z-x)=2g_{r/2}(z-x) \,\exp\Big( O\big(\tfrac{1}{r}+\tfrac{|z-x|^4}{r^3}\big) \Big)
=2g_{r/2}(z-x) \,\exp\Big( O\big(  r^{-1+\tfrac{a_N}{20}}  \big)\Big),
\end{align*}
and similarly for $(t,w)\in A_N^x$ (see \eqref{def:setA}),
\begin{align*}
q_{r-t}(z-w)=2\, g_{(r-t)/2}(z-w) \,\exp\Big( O\big(\tfrac{1}{r-t}+\tfrac{|z-w|^4}{(r-t)^3}\big) \Big)
=2\, g_{(r-t)/2}(z-w) \,\exp\Big( O\big(  r^{-1+\tfrac{a_N}{20}}  \big)\Big),
\end{align*}
because $|z-w| \le |z-x| + |w-x| \le r^{\frac{1}{2}+\frac{a_N}{80}}
+ N^{\frac{1}{2} - \frac{a_N}{4}}$ and hence, for large $N$, we can bound
\begin{align}\label{z-w_est}
	|z-w| \le 2 \,  r^{\tfrac{1}{2}+\tfrac{a_N}{80}}  
	\quad \text{and} \quad
	|r-t| \ge \tfrac{1}{2} r \,,
	\qquad \text{for}\quad \ \ t \le N^{1-a_N}\,, \ \ r\ge N^{1-9a_N/40} \,.
\end{align}
By a straightforward but tedious computation, there exists a positive constant $c$ such that
\begin{align}\label{kernel_est}
\sup\Bigg\{\Big|\frac{g_{r/2}(z-x)}{g_{(r-t)/2}(z-w)} -1 \Big|  \,\,&\colon \,\, r>N^{1-\tfrac{9a_N}{40}} \,,\, t<N^{1-a_N}\,,\,
 |w-x|<N^{\tfrac{1}{2}-\tfrac{a_N}{4}}, \\
 &\qquad\qquad \qquad |z-x|<r^{\tfrac{1}{2}+\tfrac{a_N}{80}}  \Bigg\} \,\,
 = O\Big(e^{-c (\log N)^\gamma }\Big) \notag
\end{align}
as $N$ tends to infinity, and by the local limit theorem, this bound can be 
transferred to the ratio $q_r(z-x)/q_{r-t}{(z-w)}$.

\smallskip

We are ready to estimate the error of replacing $q_{r-t}(z-w)$ by $q_{r}(z-x)$ 
in \eqref{defABgeq}.
We first restrict the sum on $(r,z)\in B_{N}^\geq(x)$.
Then the contribution to \eqref{eq:decon} is
\begin{align}\label{eq:errorM}
&\frac{\sqrt{\log N}}{N}\sum_{x\in \Z^2} \phi(\tfrac{x}{\sqrt{N}})
\sumtwo{(t,w)\in A_N^x}{(r,z)\in B_{N}^\geq(x)}
\frac{Z_{0,t,\gb_N}^A(x,w)}{Z_{N,\gb_N}^A(x)} \Big( q_{r-t}(z-w)-q_r(z-x)\Big) \cdot 
\sigma_N\,\xi_{r,z} \cdot Z_{r,N,\gb_N}(z) \,,
\end{align}
whose $L^1(\bbP)$ norm is bounded by
\begin{align}\label{L1error}
 &\frac{\sqrt{\log N}}{N}\sum_{x\in \Z^2} |\phi(\tfrac{x}{\sqrt{N}})| \,
\bbE \,\Bigg[ \frac{1}{Z_{N,\gb_N}^A(x)} \notag\\
&\qquad \times \Big|\sumtwo{(t,w)\in A_N^x}{(r,z)\in  B_{N}^\geq(x)}  \!\!\! Z_{0,t,\gb_N}^A(x,w) \,\,\Bigg\{ 1-\frac{q_r(z-x)}{q_{r-t}(z-w)}\Bigg\}\, \,q_{r-t}(z-w)
 \cdot \sigma_N\,\xi_{r,z} \cdot Z_{r,N,\gb_N}(z) \Big|\,\Bigg] \notag\\
 &\leq \frac{\sqrt{\log N}}{N}\sum_{x\in \Z^2} |\phi(\tfrac{x}{\sqrt{N}})|
\, \bbE  \Big[\frac{1}{Z_{N,\gb_N}^A(x)^2 }\Big]^{1/2} \\
&\qquad \times  \bbE \Bigg[\Bigg(\sumtwo{(t,w)\in A_N^x}{(r,z)\in B_{N}^\geq(x)}  \!\!\! Z_{0,t,\gb_N}^A(x,w) \,\,\Bigg\{ 1-\frac{q_r(z-x)}{q_{r-t}(z-w)}\Bigg\}\, \,q_{r-t}(z-w)
 \cdot \sigma_N\,\xi_{r,z} \cdot Z_{r,N,\gb_N}(z) \Bigg)^2\,\Bigg]^{1/2} \notag
\end{align}
We recall that $\bbE \big[\,Z_{N,\gb_N}^A(x)^{-2}\,\big]$ is uniformly 
bounded by the negative moment estimate \eqref{eq:-pmom'}, while by orthogonality of terms in the chaos expansion and applying \eqref{kernel_est}, the last expectation can be bounded as 
\begin{align*}
&\sumtwo{(t,w)\in A_N^x}{(r,z)\in B_{N}^\geq(x)}  \bbE \big[\,Z_{0,t,\gb_N}^A(x,w)^2\big]\,\,
\Big\{ 1-\frac{q_r(z-x)}{q_{r-t}(z-w)}\Big\}^2\, \,q_{r-t}(z-w)^2
 \, \sigma_N^2 \,\, \bbE\big[Z_{r,N,\gb_N}(z)^2 \,\big] \\
 & \qquad = O\Big(e^{-c (\log N)^\gamma}\Big) \!\!\!\!\!\!
 \sumtwo{(t,w)\in A_N^x}{ (r,z)\,\in B_N^\geq(x)} \!\!\!\!\!\!
 \bbE \big[\,Z_{0,t,\gb_N}^A(x,w)^2\big]\,\, \,q_{r-t}(z-w)^2\, \sigma_N^2\,\,\bbE\big[Z_{r,N,\gb_N}(z)^2 \,\big].
\end{align*}
By \eqref{defABgeq}, this last
sum is bounded by $\bbE \big[ \,Z_{N,\gb_N}^{A,B^\geq}(0)^2 \,\big]
\le \bbE\big[ \, Z_{N,\gb_N}(0)^2 \,\big] \le C_{\hat\beta}$
uniformly in $N$, see \eqref{eq:2mom}. These estimates show that \eqref{L1error} 
is $O\big( \sqrt{\log N} \, \exp\big( -c\,(\log N)^\gamma \big) \big)$
and hence converges to zero, thus the $L^1(\bbP)$ norm of \eqref{eq:errorM} converges to zero too.

\smallskip

To complete the step, it remains to check that in the chaos expansion
\eqref{defABgeq} for $Z_{N, \gb_N}^{A,B\geq}(x)$, 
the contribution of the complementary regime
$(r,z)\in B^\geq_N \,\setminus\, B^\geq_N(x)$,
i.e.\ $|z-x|\geq r^{1/2+ a_N/80}$,
vanishes in $L^1(\bbP)$ as $N\to\infty$, and the same is true 
if we replace the kernel $q_{r-t}(z-w)$ by $q_{r}(z-x)$. 
Note that in this regime, 
by moderate deviation estimates,
\begin{equation}\label{eq:decay}
	q_{r}(z-x)\leq  \exp\Big\{ -c\frac{|z-x|^2}{r}\Big\} 
	\leq \exp\Big\{ -c \, r^{a_N/40} \Big\} \leq 
	\exp\Big\{  -c\, e^{c\,(\log N)^\gamma}\Big\},
\end{equation}
and the same bound holds for $q_{r-t}(z-w)$,
because 
$|w-x| \le N^{\frac{1}{2} - \frac{a_N}{4}} = o(r^{1/2})$
as $N\to\infty$ (since $r \ge N^{1-9 a_N/40}$) and hence
$|w-x| = o(|z-x|)$ in this regime.
These bounds can then be used to show that
\begin{align*}
&\frac{\sqrt{\log N}}{N}\sum_{x\in \Z^2} \phi(\tfrac{x}{\sqrt{N}})
 \!\!\!
\sumtwo{(t,w)\in A_N^x\,,\,(r,z)\in B_{N}^\geq}{|z-x|> r^{1/2 + a_N/80}}  \!\!
\frac{Z_{0,t,\gb_N}^A(x,w)}{Z_{N,\gb_N}^A(x)}  q_{r-t}(z-w)\cdot \sigma_N\,\xi_{r,z} 
\cdot Z_{r,N,\gb_N}(z)
\, \xrightarrow[N\to\infty]{L^1(\bbP)} \, 0 \,,
\end{align*}
and the same holds
when $q_{r-t}(z-w)$ is replaced by $q_{r}(z-x)$.
Indeed, we can argue as in \eqref{L1error} 
and then use the fact that
the number of terms in the sums over $x, (t,w), (r,z)$ 
is only polynomial in $N$, while \eqref{eq:decay} decays faster.
\qed

\subsection{Proof of Proposition~\ref{prop:hatZA}}
\label{proof:prop:hatZA}

Recalling \eqref{eq:hatZAconv}, we want to prove that
\begin{align}\label{ZB_conv}
 	\frac{ \sqrt{\log N}}
	{\sqrt{\pi} \, \hat\beta} \frac{1}{N} \sum_{x\in\Z^2} (Z_{N,\gb_N}^{B^\geq}(x)-1)
	\, \phi(\tfrac{x}{\sqrt{N}}) \ \xrightarrow[N\to\infty]{d} \
	\langle v^{(\sqrt{2}c_{\hat\beta})}(1/2,\cdot), \phi \rangle \,,
\end{align}
where $v^{(c)}(s,x)$ is the solution of the two-dimensional additive SHE as in \eqref{eq:ASHE}.

\smallskip

The proof of \eqref{ZB_conv} follows the same line as the proof of 
Theorem 2.13 in \cite{CSZ17b}, 
which proved the convergence of
the fluctuations of the polymer partition function
$Z_{N,\beta_N}(x)$ as a space-time random field 
to the solution of the additive SHE. To see heuristically why the limit in \eqref{ZB_conv} 
should be Gaussian, 
we can write the LHS of \eqref{ZB_conv} as a polynomial chaos expansion, see \eqref{eq:polyB}, 
where the dominant contribution (in $L^2$) comes from terms of finite order in the expansion because 
$\hat\beta\in (0,1)$. Each such term is of the form 
$\sigma_N^k\prod_{i=1}^k q_{n_i-n_{i-1}}(x_i-x_{i-1}) \, \xi_{n_i, x_i}$, which due to the random 
walk transition kernels $q_\cdot(\cdot)$, depends only on disorder $\xi_{\cdot, \cdot}$ in a 
neighborhood of $(n_0, x_0):=(0,x)$ that is negligible on the diffusive scale. Given such local dependence 
on the disorder, it is then not surprising that when averaged over $(0,x)$ on the diffusive scale with 
weight $\phi(x/\sqrt{N})$, we should get a Gaussian limit. 
The proof in \cite{CSZ17b} also shows that terms of order two and higher 
in the chaos expansion leads to an independent white noise in the limit, which leads to a noise 
coefficient $c_{\hat\beta} > 1$ in \eqref{eq:ASHE}.

We now recall the key element in the proof of Theorem 2.13 in \cite{CSZ17b} and show how it can be adapted to our setting.
The key technical tool is the following variant of Proposition~8.1 in \cite{CSZ17b},  specialized to the simple random walk on $\Z^2$
(where we average in space, rather than in space-time).
It will show that,
in the polynomial chaos expansion of the left hand
side of \eqref{ZB_conv},
there are ``building blocks'' that converge
in distribution to independent Gaussian random variables.

 \begin{proposition}\label{Prop8.1}
 For integer $M$ and $i\in \{1,...,M\}  $, define intervals $I_i:=\big( N^{\tfrac{i-1}{M}} , N^{\tfrac{i}{M}}\big]$.
A $k$-tuple $(i_1,...,i_k)\in \{1,...,M\}^k$
is said to belong to  $\{1,...,M\}^k_\sharp$ if $|i_j-i_{j'}|\geq 2$
 for all $j\neq j'$. 
 
For $N\in\N$, let $\xi = (\xi^{(N)}_{n,x})_{(n,x) \in \N \times \Z^2}$
be i.i.d.\ with zero mean and unit variance.

Given $N,M\in\N$, a $k$-tuple  $(i_1,...,i_k)\in\{1,...,M\}^k_\sharp$
and a point $x\in\Z^2$, we define a random variable
$\Theta^{N;M}_{i_1,...,i_k} (x)$, a multilinear polynomial of degree $k$
 in the variables $\xi$'s, as follows:
 \begin{align*}
 \Theta^{N;M}_{i_1,...,i_k} (x)& :=
	\left(\frac{M}{R_N}\right)^{\tfrac{k-1}{2}} \!\!\!\!\!\!
	\sumtwo{n_1 \in I_{i_1}, \, n_2 - n_1 \in I_{i_2},  \ldots, \, n_k - n_{k-1} \in I_{i_k}}
	{n_0:=0,\,\, x_0:=x, \,\, x_1,  \ldots, x_k \in \Z^d}
	\, \prod_{j=1}^k q_{n_j-n_{j-1}}(x_j-x_{j-1})
	\, \prod_{i=1}^k \xi_{n_i,x_i} \,,
 \end{align*}
where $q_n(x)$ is the transition kernel of the simple symmetric random walk
on $\Z^2$, and $R_N$ is the expected overlap, defined in \eqref{eq:RN}. 
For $\phi \in C_c(\R^2)$, we define the space-averaged version
\begin{equation*}
	\Theta^{N;M;\phi}_{i_1,...,i_k}
	:= \frac{1}{N}\sum_{x\in \Z^2} 
	\Theta^{N;M}_{i_1,...,i_k} (x) \, \phi(\tfrac{x}{\sqrt{N}}) 
\end{equation*}

Let $D_M$ denote the subset of
$(i_1,\ldots, i_k) \in \{1,...,M\}^k_\sharp$
that satisfy $i_1>\max\{i_2,...,i_k\}$,
called {\em dominated sequences}. 
Then, for any fixed $M\in\N$ and $\phi \in C_c(\R^2)$, the family of
random variables $(\Theta^{N;M;\phi}_{(i_1,...,i_k)})_{(i_1, \ldots, i_k)\in D_M}$ 
converges in distribution as $N\to\infty$
to a family $(\zeta^{\phi}_{(i_1,...,i_k)})_{(i_1, \ldots, i_k)\in D_M}$ of independent
Gaussian random variables with 
\begin{equation} \label{eq:meanvar}
	\bbE\big[ \zeta^{\phi}_{(i_1,...,i_k)} \big] = 0 \,, \qquad
	\bbvar\big[\zeta^{\phi}_{(i_1,...,i_k)}\big] = 2 \,\sigma_{\phi}^2 \,
	\ind_{\{i_1 = M\}} \,,
\end{equation}
i.e.\ the variance is non-zero only if $i_1=M$, and is given by
 \begin{equation} \label{eq:sigmaphi}
	\sigma_{\phi}^2
	= \int_{(\R^2)^2} \phi(x) \, K_{\frac{1}{2}}(x,y) \, \phi(y) \, \dd x \, \dd y 
	\qquad \text{with}\qquad  
	K_{\frac{1}{2}}(x,y) 
	= \int_0^{\frac{1}{2}} \frac{1}{4\pi u} e^{-\frac{|x-y|^2}{4u}} \dd u \,.
\end{equation}
\end{proposition}

The proof of Proposition~\ref{Prop8.1} in \cite{CSZ17b}
is based on a variant of the {\em fourth-moment theorem} for polynomial chaos expansions,
as formulated in \cite[Theorem 4.2]{CSZ17b},
which was obtained in \cite{NPR}
building on \cite{NP05,dJ87, dJ90}. To check the variance, note that
\begin{equation} \label{eq:fili}
\begin{split}
	\bbvar\big[\Theta^{N;M;\phi}_{i_1,...,i_k}\big] 
	& = \frac{1}{N^2} \sum_{x,y\in \Z^2} \sum_{n_1\in I_{i_1}, \, x_1\in\Z^2} 
	\phi(\tfrac{x}{\sqrt{N}}) \phi(\tfrac{y}{\sqrt{N}}) q_{n_1}(x_1-x) q_{n_1}(x_1-y) \\
	&  \quad \qquad \times \Big(\frac{M}{R_N}\Big)^{k-1} 
	\!\!\!\! \sumtwo{n_2 - n_1 \in I_{i_2},  
	\ldots, \, n_k - n_{k-1} \in I_{i_k}}
	{x_2,  \ldots, x_k \in \Z^2}  \prod_{j=2}^k q_{n_j-n_{j-1}}(x_j-x_{j-1})^2 ,
\end{split}
\end{equation}
where the second line tends to $1$ as $N\to\infty$ by the definition of $R_N$ and $I_i$.
We can write
\begin{equation*}
	\sum_{x_1 \in \Z^2} q_{n_1}(x_1-x) q_{n_1}(x_1-y)
	= q_{2n_1}(x-y) 
	\underset{\, n_1\to\infty \, }{=}
	\big( g_{n_1}(x-y )
	+o\big(\tfrac{1}{n_1}\big) \big) \, 2 \, \ind_{\{x-y\in \Z^2_{\rm even}\}}
\end{equation*}
by the local limit theorem, 
where $\Z^2_{\rm even} := \{(a,b) \in \R^2:
a+b \text{ is even}\}$, $g_t(x)$ is as in \eqref{eq:g}, the factor $2$ 
is due to random walk periodicity and we have
$g_{n_1}(\cdot)$ instead of $g_{2n_1}(\cdot)$ because
the random walk $S_n$ has covariance matrix $\frac{n}{2} I$. Then, by a Riemann sum approximation,
as $N\to\infty$ the first line in \eqref{eq:fili} is close to the integral
\begin{align*}
	\int_{(\R^2)^2} \phi(x')\phi(y') 
	\Big(\int_{N^{\frac{i_1-1}{M}-1}}^{N^{\frac{i_1}{M}-1}} g_u(x'-y') \, \dd u\Big) 
	\dd x' \dd y' \, \xrightarrow[N\to\infty]{} \,
	\begin{cases}
	0 & \text{if } i_1 < M \\
	2 \, \sigma_\phi^2 & \text{if } i_1 = M
	\end{cases} \,,
\end{align*}
with $\sigma_\phi^2$ defined in \eqref{eq:sigmaphi}.
Also note that for $i_1=M$, the dominant contribution 
comes from $n_1\in [\eps N, N]$ for $\eps$ small, 
and hence restricting to $n_1\in [1,N]$, or $n_1 \in I_M = (N^{1-\frac{1}{M}}, N]$, 
or $n_1\geq  N^{1-9a_N/40}$ makes no difference as $N\to\infty$
(for any fixed $M\in\N$).

\medskip

Let us show how Proposition~\ref{Prop8.1} can be applied to prove \eqref{ZB_conv}.
Recall from \eqref{eq:ZAB} that
\begin{equation} \label{eq:polyB}
	Z_{N,\gb_N}^{B^\geq}(x)-1 = \sum_{k=1}^N \sigma_N^k
	\sumtwo{N^{1-9a_N/40}< n_1 < \ldots < n_k \le N}
	{n_0:=0,\,\, x_0:=x, \,\, x_1, \ldots, x_k \in \Z^2}
	\prod_{i=1}^k q_{n_i - n_{i-1}}(x_i - x_{i-1}) \, \xi_{n_i, x_i} \,.
\end{equation}
For fixed $M\in\N$,
grouping each $n_i-n_{i-1}$ according to which interval 
$I_j:=\big( N^{\tfrac{j-1}{M}} , N^{\tfrac{j}{M}}\big]$ it belongs to, and recalling
\eqref{eq:sigma} and \eqref{eq:betaN},  we have the following approximation:
\begin{align}\label{domin_conv}
	\frac{ \sqrt{\log N}}
	{\sqrt{\pi} \, \hat\beta} \frac{1}{N} \sum_{x\in\Z^2} 
	(Z_{N,\gb_N}^{B^\geq}(x) -1)\, \phi(\tfrac{x}{\sqrt{N}})  \approx
	\sum_{k=1}^M\frac{\hat \beta^{k-1}}{M^{(k-1)/2}} 
	\sumtwo{(i_1,...,i_k)\in\{1,...,M\}_\sharp^k}{i_1=M} \Theta_{i_1,...,i_k}^{N;M;\phi}
	\,,
\end{align}
where $\approx$ means that the difference of the two sides vanishes
in $L^2(\bbP)$ as $N\to\infty$ followed by $M\to\infty$.
The restriction $i_1=M$
in \eqref{domin_conv} is due to $n_1>N^{1-9a_N/40}$, 
which gives rise to a dominated sequence. The error 
from relaxing $n_1>N^{1-9a_N/40}$ to $n_1\in I_M=\big( N^{\tfrac{M-1}{M}} , N\big]$ 
is negligible in $L^2$, as noted above, while the error from restricting to
$(i_1, \ldots, i_k)\in \{1, \ldots, M\}_\sharp^k$
(rather than the whole $\{1, \ldots, M\}^k$)
is also negligible in $L^2(\bbP)$, when we first send $N\to\infty$ 
and then $M\to\infty$, as shown in \cite[Lemma 6.2]{CSZ17b}.

We can then apply Proposition~\ref{Prop8.1} to conclude that,
as
we let $N\to\infty$ for fixed $M\in\N$, the right hand side of 
\eqref{domin_conv} converges in distribution
to the same expression
with $\Theta_{i_1,...,i_k}^{N;M;\phi}$ replaced by $\zeta_{i_1,...,i_k}^{\phi}$,
i.e.\ to a Gaussian random variable with zero mean and with variance
\begin{equation} \label{eq:desu}
	\sum_{k=1}^M \frac{(\hat \beta^2)^{k-1}}{M^{k-1}} 
	\, 2 \, \sigma_\phi^2 \cdot \big| \big\{
	(i_1,...,i_k)\in\{1,...,M\}_\sharp^k: \ i_1 = M\big\} \big| \,.
\end{equation}
If we let $M\to\infty$,
since $\big| \big\{(i_1,...,i_k)\in\{1,...,M\}_\sharp^k: \ i_1 = M\big\} \big| 
= M^{k-1}(1+o(1))$, the sum in \eqref{eq:desu} converges to
the following explicit expression, with $c_{\hat\beta}$ as in \eqref{eq:ASHE}:
\begin{align*}
	2 \, \sigma_\phi^2\sum_{k\geq 1} \hat \beta^{2(k-1)}  
	= 2 \, \sigma_\phi^2 \, \frac{1}{1-\hat \beta^2} 
	= (\sqrt{2} \, c_{\hat\beta})^2 \sigma_\phi^2 \,,
\end{align*}
This agrees with the variance of 
$\langle v^{(\sqrt{2}c_{\hat\beta})}(1/2,\cdot), \phi \rangle$,
see \eqref{eq:sigma2phi},
which proves \eqref{ZB_conv}.\qed

\section{Edwards-Wilkinson Fluctuations for KPZ} \label{S:KPZ}

In this section we prove Theorem~\ref{T:KPZ}, which gives
Edwards-Wilkinson fluctuations for the Hopf-Cole solution
$h_\epsilon(t,z) = \log u_\epsilon(t,z)$ of the mollified KPZ
(where $u_\epsilon(t,z)$ solves the mollified SHE, see \eqref{eq:ueps0}).

\smallskip

\emph{The proof follows the same lines as in the directed polymer case}.
This is possible because $u^\epsilon(t,z)$ admits a Feynman-Kac representation,
which casts it in a form close to the directed polymer partition function
of size $N = \epsilon^{-2}t$.
Indeed, by \cite[Section~3]{BC95} (see also \cite[eq.~(2.27)]{CSZ17b}),
for fixed $(t,z)$ we have the following equality in law:
\begin{equation} \label{eq:ueps}
\begin{split}
	u^\epsilon(t,z) & \overset{d}{=}
	\E_{\epsilon^{-1}z} \bigg[ \exp\bigg\{
	\! \iint\limits_{(0,\epsilon^{-2}t) \times \R^2}
	\!\!\!\!\!\! \Big(\beta_\epsilon\, j(B_s - x)
	\xi(s, x) \dd s\dd x - \tfrac{1}{2} \beta_\epsilon^2
	\, j(B_s-x)^2  \,\dd s \, \dd x \Big) \bigg\} \bigg] \\
	& = \E_{\epsilon^{-1}z} \bigg[ \exp\bigg\{
	 \int_0^{\epsilon^{-2}t} \int_{\R^2}
	\beta_\epsilon\, j(B_s - x) \,
	\xi(s, x)\dd s \dd x - \tfrac{1}{2} \beta_\epsilon^2
	\, (\epsilon^{-2}t) \, \|j\|_{L^2(\R^2)}^2 \bigg\} \bigg]  \,,
\end{split}
\end{equation}
where $B = (B_s)_{s \in [0,\infty)}$ under $\P_x$ is a standard
Brownian motion on $\R^2$ started at $x$.

\smallskip

We first perform a decomposition of $u^\epsilon(t,z)$ similar to that described
described in Section~\ref{sec:methods}, which reduces
Theorem~\ref{T:KPZ} to the four
Propositions~\ref{prop:1KPZ}-\ref{prop:4KPZ} (see \S\ref{sec:decoKPZ}).
These are proved later (see \S\ref{sec:proofKPZ})
in analogy with the corresponding results
for directed polymer (see Section~\ref{S:polymer}),
exploiting \emph{moment bounds} analogous to those in
Section~\ref{S:moments} (see \S\ref{sec:momKPZ}).

\smallskip

\emph{Henceforth we set $t=1$ and we focus on $u^\epsilon(z) := u^\epsilon(1,z)$.}

\subsection{Decomposition, linearization and Wiener chaos}
\label{sec:decoKPZ}

By \eqref{eq:ueps} and \eqref{eq:cHN}-\eqref{eq:ZN}, $u^\epsilon(z)$ is comparable to $Z_N(x)$, provided \emph{we identify
$N = \epsilon^{-2} $, $x = \epsilon^{-1} z$}.

As in \eqref{eq:epsilon}-\eqref{def:setA}, we define (for a $\gamma^*$ 
small enough, depending only on $\hat\beta$ as in \eqref{eq:epsilon})
\begin{gather}
	\label{eq:aeps}
	a_\epsilon := \frac{1}{(\log \epsilon^{-2})^{1-\gamma}} \quad \text{for fixed }
	\gamma \in (0, \gamma^*) \,, \\
	\label{eq:Aeps}
	A_\epsilon^z := \big\{ (s,x): \ 0 < s \le (\epsilon^{-2})^{1-a_\epsilon}, \
	|x - \epsilon^{-1}z| < (\epsilon^{-2})^{\frac{1}{2} - \frac{a_\epsilon}{4}} \big\} \,,
\end{gather}
and we introduce a modified partition function $u^\epsilon_A(z)$,
obtained by restricting the double integral in the first line
of \eqref{eq:ueps} to the set $(s,x) \in A_\epsilon^z$.
This yields the decomposition
\begin{align}\label{decomposeSHE}
	u^\epsilon(z) = u^\epsilon_A(z)+ \hat u^\epsilon_A(z) \,,
\end{align}
where $\hat u^\epsilon_A(z)$, defined by this relation, is a ``remainder'' which,
for \emph{fixed} $z$, can be shown to be much smaller than $u^\epsilon_A(z)$.
More precisely, as in \eqref{first_approx_heur0}-\eqref{first_approx_heur}, we define
$O^\epsilon(z)$ by
\begin{equation} \label{eq:OKPZ}
	\log u^\epsilon(z) =
	\log u^\epsilon_A(z)
	+ \frac{\hat u^\epsilon_A(z)}{u^\epsilon_A(z)}+ O^\epsilon(z) \,,
\end{equation}
and we have the following analogues of Propositions~\ref{prop:R}-\ref{prop:ZA}.

\begin{proposition}\label{prop:1KPZ}
Let $O^\epsilon(\cdot)$ be defined as above, then for any $\phi \in C_c(\R^2)$
\begin{equation*}
	%\sqrt{\frac{\log \epsilon^{-2}}{2\pi}} 
	\sqrt{\log \eps^{-1}}
	\int_{ \R^2}
	\big( O^\epsilon(z) - \bbE[O^\epsilon(z)] \big) \, \phi(z) \,\dd z
	\ \xrightarrow[\epsilon\downarrow 0]{L^2(\bbP)} \ 0 \,.
\end{equation*}
\end{proposition}

\begin{proposition}\label{prop:2KPZ}
Let  $u^\epsilon_A(\cdot)$ be defined as above, then for any $\phi \in C_c(\R^2)$
\begin{equation*}
	\sqrt{\log \eps^{-1}}
	\int_{ \R^2}
	\big( \log u^\epsilon_A(z) - \bbE[\log u^\epsilon_A(z)] \big) \,\phi(z) \,\dd z
	\ \xrightarrow[\epsilon \downarrow 0]{L^2(\bbP)} \ 0 \,.
\end{equation*}
\end{proposition}

Next, in analogy with \eqref{eq:B}-\eqref{eq:ZB0}, we introduce the subset
\begin{equation}\label{eq:Beps}
	B_\epsilon^\geq := \big( (\epsilon^{-2})^{1-9a_\epsilon/40}, \epsilon^{-2} \big)
	\times \R^2 \,,
\end{equation}
and we introduce $u^\epsilon_{B^\geq}(z)$,
obtained by restricting the double integral in the first line
of \eqref{eq:ueps} to the set $(s,x) \in B^\geq_\epsilon$.
We have the following analogues of Propositions~\ref{prop:ZB}-\ref{prop:hatZA}.

\begin{proposition}\label{prop:3KPZ}
Let $u^\epsilon_A(\cdot)$,
$\hat u^\epsilon_A(\cdot)$,
$u^\epsilon_{B^\geq}(\cdot)$
be defined as above, then for any $\phi \in C_c(\R^2)$
\begin{equation} \label{eq:main-approx-KPZ}
	\sqrt{\log \eps^{-1}} \int_{\R^2}
	\bigg( \frac{\hat u^\epsilon_A(z)}{u^\epsilon_A(z)} -
	\big( u^\epsilon_{B^\geq}(z) - 1 \big) \bigg) \, \phi(z)\,\dd z
	\ \xrightarrow[\epsilon\downarrow 0]{L^1(\bbP)} \ 0 \,.
\end{equation}
\end{proposition}

\begin{proposition}\label{prop:4KPZ}
Let $u^\epsilon_{B^\geq}(\cdot)$
be defined as above, then for any $\phi \in C_c(\R^2)$
\begin{equation} \label{hatu_eps}
	\frac{\sqrt{\log \eps^{-1}}}{\sqrt{2\pi} \, \hat\beta}
	\int_{\R^2}
	\big( u^\epsilon_{B^\geq}(z) - 1 \big) \, \phi(z)
	\ \xrightarrow[\epsilon\downarrow 0]{d} \
	\langle v^{(c_{\hat\beta})}(1,\cdot), \phi \rangle \,.
\end{equation}
\end{proposition}

Theorem~\ref{T:KPZ} is a direct consequence of
Propositions~\ref{prop:1KPZ}-\ref{prop:4KPZ}.
Regarding the centering, by \eqref{eq:OKPZ} we have
$\bbE[\log u^\epsilon(z)] =
\bbE[\log u^\epsilon_A(z)] +
\bbE[O^\epsilon(z)]$, because
$\hat u^\epsilon_A(z)/u^\epsilon_A(z)$ has zero mean,
as we show in a moment.

\smallskip

By \eqref{eq:ueps} and the definition of Wick exponential \cite[\S 3.2]{J97},
we have the following Wiener chaos representation for $u^\epsilon(z)$,
where we set $t_0 := 0$ and $y_0 := \epsilon^{-1} z$:
\begin{equation} \label{eq:WienerKPZ}
	u^\epsilon(z) \overset{d}{=} 1+ \!\! \sum_{k\geq1}
	\gb_\epsilon^k \!\!\!\!\!
	\idotsint\limits_{\substack{0<t_1<\cdots<t_k<\epsilon^{-2} \\
	\vec{x} \in (\R^2)^k}} \!\!
	\Bigg( \, \int\limits_{(\R^2)^k} \!\prod_{i=1}^k g_{t_i-t_{i-1}}(y_i-y_{i-1})
	\,j(y_i-x_i) \,\dd \vec y\Bigg) \prod_{i=1}^k \xi(t_i, x_i) \dd t_i \dd x_i \,,
\end{equation}
where $g_t(\cdot)$ is the transition kernel of the Brownian motion.

The modified partition function $u^\epsilon_A(z)$
admits a similar Wiener chaos expansions, with the outer integrals
restricted to the set $\{(t_1, x_1), \ldots, (t_k,x_k)\} \subseteq A_\epsilon^z$.
It follows that the Wiener chaos expansion of
$\hat u^\epsilon_A(z) := u^\epsilon(z) - u^\epsilon_A(z)$
contains at least one factor $\xi(t_i, x_i)$ with $(t_i, x_i)$ outside
$A_z^\epsilon$, which is not present in $u^\epsilon_A(z)$,
hence $\bbE[\hat u^\epsilon_A(z)/u^\epsilon_A(z)]=0$.

Similarly, the Wiener chaos expansions of $u^\epsilon_{B^\geq}(z)$
is obtained by restricting the outer integrals in \eqref{eq:WienerKPZ}
to the set $\{(t_1, x_1), \ldots, (t_k,x_k)\} \subseteq B^\geq_\epsilon$,
i.e.\ imposing $t_1 > (\epsilon^{-2})^{1-9a_\epsilon/40}$.

\subsection{Moment bounds}
\label{sec:momKPZ}

We estimate positive and negative moments of $u^\epsilon(z)$.

\smallskip

We start with the second moment.
We prove below the following bounds for $u^\epsilon(z)$, $u^\epsilon_A(z)$
and $\hat u^\epsilon_A(z)$, which are close
analogues of \eqref{eq:2mom}, \eqref{eq:2mom'}, \eqref{eq:2mom''}:
\begin{align}
	\nonumber
	& \forall \hat\beta \in (0,1) \ \ \exists C_{\hat\beta} < \infty
	\ \text{ such that } \ \forall \epsilon > 0: \\
	\label{eq:2momKPZ}
	& \bbE[u^\epsilon(z)^2] \le C_{\hat\beta} \,, \qquad
	\bbE[u^\epsilon_A(z)^2] \le C_{\hat\beta} \,, \qquad
	\bbE[\hat u^\epsilon_A(z)^2] \le C_{\hat\beta} \, a_\epsilon \,.
\end{align}

We can now easily deduce bounds for
higher positive moments. 
By hypercontractivity \cite[Theorem~5.1]{J97},
the $L^p$ norm of a Wiener chaos expansion like \eqref{eq:WienerKPZ}
is bounded by the $L^2$ norm of a modified expansion, with the $k$-th order term
multiplied by $(c_p)^k$ (i.e., $\hat\beta$
replaced by $c_p \hat\beta$), with $c_p := \sqrt{p-1}$.
For $\hat\beta \in (0,1)$ we can choose $p > 2$
such that $\hat\beta c_{p} < 1$, so as to apply the bounds in \eqref{eq:2momKPZ}.
This yields an analogue of \eqref{eq:pmom}:
\begin{align}
	\nonumber
	& \forall \hat\beta \in (0,1) \ \ \exists
	p = p_{\hat\beta} \in (2,\infty)
	\ \ \exists C'_{\hat\beta} < \infty \ \text{ such that } \ \forall \epsilon > 0: \\
	\label{eq:pmomKPZ}
	& \bbE[u^\epsilon(z)^p] \le C'_{\hat\beta} \,, \qquad
	\bbE[u^\epsilon_A(z)^p] \le C'_{\hat\beta} \,, \qquad
	\bbE[|\hat u^\epsilon_A(z)|^p] \le C'_{\hat\beta} \, (a_\epsilon)^{p/2} \,.
\end{align}

\begin{proof}[Proof of \eqref{eq:2momKPZ}]
We compute $\bbE[u^\epsilon(z)^2]$ by using \eqref{eq:WienerKPZ}, applying
the identity $g_t(y) g_t(y') = 4 g_{2t}(y-y') g_{2t} (y+y')$, and
switching to new variables $z_i := y_i - y'_i$, $w_i := y_i + y'_i$.
This leads to
the following expression (see \cite[\S8.2]{CSZ18} for details):
\begin{equation} \label{eq:2momKPZ2}
\begin{split}
	\bbE[u^\epsilon(z)^2] = 1 + \sum_{k\ge 1} (\beta_\epsilon^2)^k
	\idotsint\limits_{\substack{0<t_1<\cdots<t_k<\epsilon^{-2} \\
	\vec{z} \in (\R^2)^k , \, \vec{w} \in (\R^2)^k}}
	\Bigg( & \prod_{i=1}^k
	 g_{2(t_i-t_{i-1})}(z_i-z_{i-1})
	\,J(z_i) \cdot \\
	& \ \cdot g_{2(t_i-t_{i-1})}(w_i - w_{i-1}) \Bigg)
	\prod_{i=1}^k \dd t_i \, \dd z_i \, \dd w_i \,,
\end{split}
\end{equation}
where $J := j * j$ and we set $z_0 := 0$, $w_0 := 2\epsilon^{-1}z$.
Integrating out $w_k, w_{k-1}, \ldots, w_1$, we get
\begin{equation*}
	\bbE[u^\epsilon(z)^2] = 1 + \sum_{k\ge 1} (\beta_\epsilon^2)^k
	\idotsint\limits_{\substack{0<t_1<\cdots<t_k<\epsilon^{-2} \\
	\vec{z} \in (\R^2)^k}}
	\Bigg( \prod_{i=1}^k g_{2(t_i-t_{i-1})}(z_i-z_{i-1})
	\,J(z_i) \Bigg) \prod_{i=1}^k \dd t_i \, \dd z_i \,.
\end{equation*}
We recall that $j$, hence $J$, has compact support.
If we define
\begin{equation*}
	\bar{r}(t) := \sup_{z' \in \mathrm{supp}(J)}
	\int_{\R^2} g_{2t}(z-z') \, J(z) \, \dd z \,,
\end{equation*}
we can bound
\begin{equation*}
	\bbE[u^\epsilon(z)^2] \le 1 + \sum_{k\ge 1} (\beta_\epsilon^2)^k
	\!\!\idotsint\limits_{0<t_1<\cdots<t_k<\epsilon^{-2}}
	\prod_{i=1}^k \bar{r}(t_i - t_{i-1}) \, \dd t_i
	\le 1 + \sum_{k\ge 1}
	\bigg\{ \beta_\epsilon^2 \, \int_0^{\epsilon^{-2}} \bar r(t) \, \dd t \bigg\}^k\, .
\end{equation*}
Note that $\bar{r}(\cdot)$ is bounded for $t \ge 0$
and it satisfies $\bar{r}(t) = \frac{1}{4\pi t} + O(1)$ as $t \to \infty$,
by \eqref{eq:g}.
Recalling \eqref{eq:beps}, we see that
the bracket converges to $\hat\beta^2$ as $\epsilon \to 0$,
hence the series is uniformly bounded for $\hat\beta < 1$. This proves
the first bound in \eqref{eq:2momKPZ}.

\smallskip

The second bound in \eqref{eq:2momKPZ} follows
because $\bbE[u^\epsilon_A(z)^2] \le \bbE[u^\epsilon(z)^2]$,
since the Wiener chaos expansion for $u^\epsilon_A(z)$ is a subset
of the expansion for $u^\epsilon(z)$.

\smallskip

Finally, the third bound in \eqref{eq:2momKPZ} can be proved similarly
to \eqref{eq:2mom''} (see Subsection~\ref{sec:2mom''}),
because $\bbE[\hat u^\epsilon_A(z)^2]$ can be bounded
by an expression analogous to \eqref{eq:2momKPZ2}.\footnote{Note
that $\hat u^\epsilon_A(z)$
contains at least one point $(t_i, x_i)$ outside $A_\epsilon^z$
in the Wiener chaos
representation \eqref{eq:WienerKPZ}. Since $j(\cdot)$ has compact
support, say included in the ball $B_r := \{x \in\R^2: \, |x| \le r\}$,
the
corresponding point $(t_i, y_i)$ in \eqref{eq:WienerKPZ} must be close to
(i.e.\ at distance at most $r$ from) the point
$(t_i, x_i)$. Then
$\bbE[\hat u^\epsilon_A(z)^2]$ can be bounded
by an expression analogous to \eqref{eq:2momKPZ2}, with the integrals restricted
to the set where
at least one point $(t_i, \frac{1}{2} w_i) = (t_i, \frac{1}{2}(y_i+y'_i))$
is close to $(A_\epsilon^z)^c$.
This allows to follow the proof in Subsection~\ref{sec:2mom''}.}
\end{proof}

\medskip

We next estimate
\emph{negative moments},
establishing the following analogues of \eqref{eq:-pmom}-\eqref{eq:momlog}:
\begin{align}
	\nonumber
	& \forall \hat\beta \in (0,1) \ \ \forall p \in (0,\infty) \ \
	\exists C_{p,\hat\beta} < \infty \ \ \text{such that}
	\ \ \forall \epsilon > 0: \\
	\label{eq:-pmomKPZ}
	& \qquad \bbE \big[ u^\epsilon(z)^{-p} \big] \le C_{p,\hat\beta} < \infty \,, \\
	\label{eq:-pmom'KPZ}
	& \qquad\bbE \big[ u^\epsilon_A(z)^{-p} \big] \le C_{p,\hat\beta} < \infty \,, \\
	\label{eq:momlogKPZ}
	& \qquad\bbE[|\log u^\epsilon_A(z)|^p] \le C_{p,\hat\beta} < \infty \,.
\end{align}
Since \eqref{eq:momlogKPZ} follows easily from \eqref{eq:-pmom'KPZ},
it suffices to prove \eqref{eq:-pmomKPZ}-\eqref{eq:-pmom'KPZ}.
These are
direct corollaries of the following result, analogous to Proposition~\ref{neg-mom-poly}.

\begin{proposition}[Left tail for KPZ]\label{neg-mom-polyKPZ}
For $\Lambda \subseteq (0, \epsilon^{-2}) \times \R^2$,
denote by $u^\epsilon_{\Lambda}(z)$ what we obtain
by restricting the double integral in
the first line of \eqref{eq:ueps} to $(s,x) \in \Lambda$, i.e.\
\begin{equation} \label{eq:uepsLambda}
\begin{split}
	u^\epsilon_\Lambda(z) & :=
	\E_{\epsilon^{-1}z} \bigg[ \exp\bigg\{
	\! \iint\limits_{(s,x) \in \Lambda}
	\!\!\!\! \Big( \beta_\epsilon\, j(B_s - x) \,
	\xi(s,x) \dd s \dd x - \tfrac{1}{2} \beta_\epsilon^2
	\, j(B_s-x)^2  \,\dd s \, \dd x \Big) \bigg\} \bigg] \,.
\end{split}
\end{equation}
For any $\hat\beta \in (0,1)$ there is
$c_{\hat\beta} \in (0,\infty)$ with the following property:
for any $\epsilon > 0$ and for any choice of subset
$\Lambda \subseteq (0, \epsilon^{-2}) \times \R^2$,
one has
\begin{align} \label{eq:boundf2}
	\forall t \ge 0: \qquad
	\ \bbP(\log u^\epsilon_\Lambda(z) \leq -t) \leq
	c_{\hat\gb} \, e^{- t^2 / c_{\hat\gb}} \,.
\end{align}
\end{proposition}

It remains to prove Proposition~\ref{neg-mom-polyKPZ}.
We first need to recall concentration inequalities for
white noise (see Appendix~\ref{sec:concGauss} for more details).

The white noise $\xi = (\xi(s, y))_{(s,y) \in [0,\infty) \times \R^2}$
can be viewed as a random element of a separable Banach space $E$
of distributions on $[0,\infty) \times \R^2$
(e.g.\ a negative H\"older space, see \cite{CD18}).
Its law $\mu$ is the Gaussian measure on $E$ with Cameron-Martin space
$H = L^2([0,\infty) \times \R^2)$, and the triple $(H,E,\mu)$ is a so-called
\emph{abstract Wiener space}. In this setting, sharp concentration
inequalities are known to hold
for (not necessarily convex)
functions $f: E \to \R$ that are \emph{Lipschitz in the directions of $H$},
see \cite[eq.~(4.7) and~(4.8)]{Led96}.

We need to work with \emph{convex} functions $f: E \to \R \cup \{-\infty,+\infty\}$
that are not globally Lipschitz.
Remarkably, such functions still enjoy concentration inequalities for the \emph{left tail}
(but not, in general, for the right tail).
For $x\in E$ with $|f(x)| < \infty$,
denote by $|\nabla f(x)| \in [0,\infty]$
the maximal gradient of $f$ in the directions of $H$, defined by
\begin{equation} \label{eq:gradf}
	|\nabla f(x)| :=
	\sup_{h\in H: \, \|h\|_{H} \le 1}  \,
	\lim_{\delta \downarrow 0} \, \frac{|f(x+\delta h)-f(x)|}{\delta} \,,
\end{equation}
where the limit exists by convexity.
Then the following inequality holds (see Theorem~\ref{th:conclow}):
\begin{equation} \label{eq:conclow0}
	\mu(f \le a - t) \, \mu^*(f \ge a, |\nabla f| \le c)
	\le e^{-\frac{1}{4} (t/c)^2}  \qquad
	\forall a \in \R \,, \ \forall t, c \in (0,\infty) \,,
\end{equation}
where $\mu^*$ is the outer measure (to avoid the issue of
measurability of $|\nabla f|$).

Note that, if we \emph{fix} $a,c$ such that $\mu^*(f \ge a, |\nabla f| \le c) > 0$,
relation \eqref{eq:conclow0}
gives a bound on the left tail $\mu(f \le a - t)$ for all $t > 0$.

\smallskip

\begin{proof}[Proof of Proposition~\ref{neg-mom-polyKPZ}]
We can set $z=0$, since the law of $u_\Lambda^\epsilon(z)$ in \eqref{eq:uepsLambda}
does not depend on $z$, and we write $u_\Lambda^\epsilon := u_\Lambda^\epsilon(0)$.
We denote by $\cH_\epsilon^\xi(B)$ the argument of the exponential in
\eqref{eq:uepsLambda}, so that $u_\Lambda^\epsilon =
\E[\exp(\cH_\epsilon^\xi(B))]$.
We also introduce the shorthand
\begin{equation}\label{eq:angle}
	\langle j(B), \xi \rangle := \iint_{(s,x) \in \Lambda}
	j(B_s - x) \, \xi(s, x) \, \dd s \, \dd x \,.
\end{equation}

We start with a second moment computation:
\begin{equation} \label{eq:2momaltKPZ}
	\bbE[(u_\Lambda^\epsilon)^2] =
	\E[ \bbE[e^{\cH_\epsilon^\xi(B) + \cH_\epsilon^\xi(\tilde B)}]]
	= \E \Big[e^{\beta_\epsilon^2 \, \cL_{\Lambda}(B,\tilde B)} \Big] \,,
\end{equation}
where $B$, $\tilde B$ are independent Brownian motions, and
$\cL_{\Lambda}(B,\tilde B)$ is their overlap on $\Lambda$:
\begin{equation} \label{eq:ovl}
	\cL_{\Lambda}(B,\tilde B) := \iint_{(s,x) \in \Lambda}
	j(B_s - x) \, j(\tilde B_s - x) \, \dd s \, \dd x \,.
\end{equation}

Note that $u^\epsilon_\Lambda$
is a function of the white noise $\xi$, so we can define
\begin{equation} \label{eq:hu}
	 h_\epsilon(\xi) := \log u^\epsilon_\Lambda \,.
\end{equation}
The function $h_\epsilon(\cdot)$ is convex by H\"older's
inequality, because $\xi \mapsto \langle j(B), \xi \rangle$ is linear
(more precisely, we can ensure that $h_\epsilon(\cdot)$ is convex by
choosing a suitable version of the stochastic integral $\langle j(B), \xi \rangle$; see
Appendix~\ref{sec:stocversion}).
Then \eqref{eq:boundf2} follows by \eqref{eq:conclow0}
if we show that $\mu^*(h_\epsilon \ge a, |\nabla h_\epsilon| \le c)$ is uniformly bounded
from below, for $a = -\log 2$
and for suitable $c = c_{\hat\beta}$.

\smallskip

We need to evaluate the maximal gradient
$|\nabla h_\epsilon(\xi)|$, see \eqref{eq:gradf}.
We define a Gibbs change of measure $\bP^\xi$ on the Brownian path $B = (B_s)_{s \ge 0}$ by
\begin{equation*}
	\frac{\dd\bP^\xi}{\dd\P}(B) :=
	\frac{e^{\cH_\epsilon^\xi(B)}}{u_\Lambda^\epsilon} \,.
\end{equation*}
Let us fix $f\in H = L^2([0,\infty) \times \R^2)$.
Recalling \eqref{eq:angle} and \eqref{eq:hu}, we have
\begin{equation*}
\begin{split}
	\lim_{\delta\downarrow 0} \frac{h_\epsilon(\xi+\delta f) - h_\epsilon(\xi)}{\delta}
	& = \lim_{\delta\downarrow 0}
	\frac{1}{\delta} \log \bE^\xi[e^{\beta_\epsilon \langle j(B),
	\delta f\rangle}] =
	\beta_\epsilon \,  \bE^\xi[\langle j(B),	f \rangle ] \\
	& = \beta_\epsilon \, \iint_{(s,x) \in \Lambda}
	 \bE^\xi[ j(B_s - x) ]
	\, f(s,x) \,\dd s \, \dd x \,.
\end{split}
\end{equation*}
Taking $f$ with $\|f\|_{L^2} \le 1$ and recalling \eqref{eq:ovl}, it follows
by Cauchy-Schwarz that
\begin{equation*}
	|\nabla h_\epsilon(\xi)|^2 \le
	\beta_\epsilon^2 \, \iint_{(s,x) \in \Lambda}
	\bE^\xi[ j(B_s - y) ]^2
	\, \dd s \, \dd y =
	\frac{\E \Big[ \beta_\epsilon^2 \, \cL_{\Lambda}(B,\tilde B) \,
	e^{\cH_\epsilon^\xi(B) + \cH_\epsilon^\xi(\tilde B)}
	\Big]}{(u^\epsilon_\Lambda)^2} \,.
\end{equation*}
Then, on the event $h_\epsilon(\xi) > a = -\log 2$, i.e.\ $u^\epsilon_\Lambda > \frac{1}{2}$,
recalling \eqref{eq:2momaltKPZ}, we can bound
\begin{equation} \label{eq:valuein}
\begin{split}
	\bbE[ |\nabla h_\epsilon(\xi)|^2 \, \ind_{\{h_\epsilon(\xi) > a\}} ]
	& \le  4 \, \bbE \, \E \Big[
	\beta_\epsilon^2 \, \cL_{\Lambda}(B,\tilde B) \,
	e^{\cH_\epsilon(B) + \cH_\epsilon(\tilde B)}
	\Big] \\
	& \le 4 \, \E \Big[
	\beta_\epsilon^2 \, \cL_{\Lambda}(B,\tilde B) \,
	e^{\beta_\epsilon^2 \, \cL_{\Lambda}(B,\tilde B)}
	\Big]  \le \frac{4}{\delta} \, \E \Big[
	e^{(1+\delta)\beta_\epsilon^2 \, \cL_{\Lambda}(B,\tilde B)}
	\Big] \,,
\end{split}
\end{equation}
for any $\delta > 0$ (by $x \le \frac{1}{\delta} e^{\delta x}$). For any subcritical $\hat\beta \in (0,1)$,
we can fix $\delta = \delta_{\hat\beta} > 0$
small, so that $\hat\beta' := \hat\beta \sqrt{1+\delta} < 1$ is still subcritical.
By \eqref{eq:2momaltKPZ}, the last expected value in \eqref{eq:valuein}
is the second moment of $u_\Lambda^\epsilon$
with $\hat\beta'$ instead of $\hat\beta$, hence it is uniformly bounded
by some constant $C_{\hat\beta} < \infty$, by \eqref{eq:2momKPZ}, uniformly over all subsets
$\Lambda \subseteq (0, \epsilon^{-2}) \times \R^2$. Summarizing:
\begin{equation} \label{eq:labau}
	\sup_{\epsilon > 0} \, \bbE[ |\nabla h_\epsilon(\xi)|^2 \, \ind_{\{h_\epsilon(\xi) > a\}} ]
	\le C'_{\hat\beta} < \infty \,.
\end{equation}

We can continue as in the directed polymer case
(see Proposition~\ref{neg-mom-poly}), noting that
\begin{equation} \label{eq:intto}
\begin{split}
	\mu(h_\epsilon \ge a, |\nabla h_\epsilon| \le c)
	& = \mu(h_\epsilon \ge a)
	- \mu(h_\epsilon \ge a, |\nabla h_\epsilon| > c) \\
	& \ge \mu(h_\epsilon \ge a) -
	\frac{1}{c^2}
	\, \bbE[ |\nabla h_\epsilon(\xi)|^2 \, \ind_{\{h_\epsilon(\xi) > a\}} ] \,.
\end{split}
\end{equation}
Since $a := -\log 2$, we have
$\mu(h_\epsilon \ge a) = \mu(u_\Lambda^\epsilon \ge \frac{1}{2})
\ge (4 C_{\hat\beta})^{-1}$ as in \eqref{PL}. Plugging this bound
together with \eqref{eq:labau} into \eqref{eq:intto}, we are done by
choosing $c = c_{\hat\beta}$ large enough.
\end{proof}

\subsection{Proof of Propositions~\ref{prop:1KPZ}-\ref{prop:4KPZ}}
\label{sec:proofKPZ}

Propositions~\ref{prop:1KPZ} and~\ref{prop:2KPZ} are proved
repeating almost verbatim the proofs of Propositions~\ref{prop:R} and~\ref{prop:ZA},
which are the corresponding results for directed polymers.
We omit the details and  refer to Subsections~\ref{sec:prop:R} and~\ref{sec:prop:ZA}.

\smallskip

\begin{proof}[Proof of Proposition~\ref{prop:3KPZ}.]
We follow closely
the proof of Proposition~\ref{prop:ZB} in Subsection~\ref{sec:proof:ZB}.
Recall the decomposition $u^\epsilon(z) = u^\epsilon_A(z)+ \hat u^\epsilon_A(z)$,
see \eqref{decomposeSHE}.
Then we further decompose
\begin{align} \label{eq:deco2}
	\hat u^\epsilon_A(z) = u^\epsilon_{A,C}(z) + u^\epsilon_{A,B}(z) \,,
\end{align}
where $u^\epsilon_{A,C}(z)$, $u^\epsilon_{A,B}(z)$
are defined in analogy with
$Z_{N,\gb_N}^{\,A,B}(x)$, $ Z_{N,\gb_N}^{\,A,C}(x)$ 
from \eqref{defAC}, \eqref{defAB}:
\begin{align*}
u^\epsilon_{A,C}(z)&:=\sum_{k\geq1}  \gb_\epsilon^k 
\!\!\!\!\! \idotsint\limits_{\substack{0<t_1<\cdots<t_k<\epsilon^{-2} , \, \vec x  \in\, 
(\R^2)^k \\ 
\{(t_1,x_1),\ldots,(t_k,x_k)\} \cap C^z_\epsilon\neq \emptyset}  }
\!\!\!\!\!\!
\Big(\int_{(\R^2)^k}  \prod_{i=1}^k g_{t_i-t_{i-1}}(y_i-y_{i-1}) \,j(y_i-x_i) \,\dd \vec y\Big) 
\prod_{i=1}^k \xi(t_i, x_i) \dd t_i \dd x_i \,, \\
u^\epsilon_{A,B}(z)&:=\sum_{k\geq1}  \gb_\epsilon^k \!\!\!\!\!\!
\idotsint\limits_{\substack{0<t_1<\cdots<t_k<\epsilon^{-2}, \,
\vec{x} \in (\R^2)^k \\ 
 \{(t_1,x_1),\ldots,(t_k,x_k)\}  
\subset A^z_\epsilon \cup B^z_\epsilon \\
\{(t_1,x_1),\ldots,(t_k,x_k)\} \cap B^z_\epsilon\neq \emptyset }}
\!\!\!\!\!\!\!
\Big( \int_{(\R^2)^k} \prod_{i=1}^k g_{t_i-t_{i-1}}(y_i-y_{i-1}) \,j(y_i-x_i) \,\dd \vec y\Big) 
\prod_{i=1}^k \xi(t_i, x_i) \dd t_i \dd x_i \,,
\end{align*}
where we set $t_0 := 0$, $y_0 := \epsilon^{-1} z$,
we recall that $A_\epsilon^z$
was defined in \eqref{eq:Aeps}, while
 $B_\epsilon, C_\epsilon^z$
are defined similarly to $B_N, C_N^x$ from
\eqref{eq:B2}, \eqref{eq:C}
with $N=\epsilon^{-2}$ and $x = \epsilon^{-1}z$:
more precisely, recalling $a_\epsilon$ from \eqref{eq:aeps}, we set
\begin{equation*}
	B_\epsilon := \big( (\epsilon^{-2})^{1-a_\epsilon}, \epsilon^{-2} \big] \,, \quad \ \
	C_\epsilon^z := \big\{ (t,x) \in \R^2: \ 0 < t \le (\epsilon^{-2})^{1-a_\epsilon} \,,
	\ |x-\epsilon^{-1}z| \ge (\epsilon^{-2})^{1-\frac{a_\epsilon}{4}} \big\} \,.
\end{equation*}
The proof of Proposition~\ref{prop:3KPZ}, similarly
to Proposition~\ref{prop:ZB}, proceeds in three steps.

The first step is to show that $u^\epsilon_{A,C}(x)$ in \eqref{eq:deco2}
gives a negligible contribution, that is
\begin{align} \label{eq:inviewof}
 \sqrt{\log \epsilon^{-1}} \int_{\R^2}
	\frac{u^\epsilon_{A,C}(z)}{u^\epsilon_{A}(z)} \, \phi( z) \,\dd z
	\ \xrightarrow[\epsilon\to0]{L^2(\bbP)} \ 0 \,.
\end{align}
The proof is identical to the case for directed polymer,
see \eqref{eq:stuno} and the following lines.
The only difference is that \eqref{Step1overlap} will be replaced by its continuum analogue, which is
\begin{align*}
	\E\Bigg[  \exp\Big\{ (c_{2p}\gb_\epsilon)^2 \int_0^{\,\epsilon^{-2(1-a_\eps)}} 
	J(B^{(1)}_s- B^{(2)}_s) \,\dd s  \Big\} \,\,;\, \,\sup_{s\leq \epsilon^{-2(1-a_\epsilon)} }
	|B^{(i)}(s)| > \epsilon^{-\big(1-\tfrac{a_\epsilon}{2}\big)} ,\,\,\,\text{for}\,\, i=1,2\Bigg],
\end{align*}
where $c_{2p}:=\sqrt{2p-1}$ is the hypercontractivity constant for white noise,
$B^{(1)}, B^{(2)}$ are two independent Brownian motions and we recall that
$J(\cdot) = (j * j)(\cdot)$.
The rest of the estimates follow the same lines as in the polymer case.

In view of \eqref{eq:deco2} and
\eqref{eq:inviewof}, to complete the proof it remains to show that
\begin{equation}\label{L1SHE}
	\sqrt{\frac{\log \epsilon^{-1}}{2\pi}}  \Bigg\{\int_{\R^2}
	\frac{ u^\epsilon_{A,B}(z)}{u^\epsilon_A(z)} \, \phi(z)\,\dd z
	- \int_{\R^2}
	(u^\epsilon_{B^\geq}(z)-1) \, \phi(z)\,\dd z\,\Bigg\}
	\ \xrightarrow[\epsilon\to0]{L^1(\bbP)} 0.
\end{equation}
For $u^\epsilon_{A,B}(z)$ we can give an expression analogous to \eqref{AB-exp},
integrating over the last point $(t,w) \in A_\epsilon^z$ and the first point 
$(r,v) \in B_\epsilon$:
\begin{equation} \label{eq:align}
\begin{aligned}
	u^\epsilon_{A,B}(z) = \!\!\!\! \idotsint
	\limits_{\substack{(t,w)\in A_\epsilon^z\,,\, w' \in \R^2 \\  
	(r,v)\in B_\epsilon \,,
	\, v' \in \R^2} }
	\!\!\!\! 
	u^\epsilon_A(0,z; \dd t, \dd w, \dd w') 
	& \cdot  g_{r-t}(v'-w') \, j(v'-v) \, \beta_\epsilon \, \xi(r, v) 
	\, \dd r \,\dd v \, \dd v'
	 \\
	& 
	\cdot u^\epsilon(r,v';\epsilon^{-2},\cdot)
	\,,
 \end{aligned}
\end{equation}
where $u^\epsilon_A(0,z; \dd t, \dd w, \dd w')$ is the ``point-to-point''
partition function from $(0,z)$ to $(t,w,w')$,
similar to \eqref{AB-exp2} (the extra space variable $w'$ is
due to the convolution with $j(\cdot)$),
which is defined as follows,
where we set $t_0 := 0$ and $y_0 := \epsilon^{-1} z$:
\begin{align*}
u^\epsilon_A(0,z; \dd t, \dd w, \dd w') :=
\sum_{k\geq1}  \gb_\epsilon^k \, \Bigg\{  \!\!\!\!\!
\idotsint\limits_{\substack{0<t_1<\cdots<t_{k-1}<\epsilon^{-2} \\
(x_1, \ldots, x_{k-1}) \in (\R^2)^{k-1} \\ 
 \{(t_1,x_1),\ldots,(t_{k-1},x_{k-1})\}  
\subset A^z_\epsilon}}
\!\!\!\!\!\!\!\!\!
& \bigg( \int_{(\R^2)^{k-1}} \prod_{i=1}^{k-1}
g_{t_i-t_{i-1}}(y_i-y_{i-1}) \,j(y_i-x_i) 
\\
& \cdot g_{t-t_{k-1}}(w'-y_{k-1}) \,\dd \vec y   \bigg)
 \prod_{i=1}^{k-1} \xi(t_i, x_i) \, \dd t_i \, \dd x_i \Bigg\}
\\
& \rule{0pt}{1.4em} \cdot j(w' - w) \, \xi(t,w) \, \dd t \, \dd w \, \dd w' \, ,
\end{align*}
while $u^\epsilon(r,v';\epsilon^{-2},\cdot)$ is the ``point-to-plane'' partition
from $(r,v')$ until time $\epsilon^{-2}$,
defined by \eqref{eq:WienerKPZ}
where we set $t_0 := r$, $y_0 := v'$ and we replace $0<t_1<\cdots \,$ 
by $r<t_1<\cdots$.

In order to prove \eqref{L1SHE}, as in the polymer case,
we need two more steps:
the second step is to prove that
the contribution from $r < (\epsilon^{-2})^{1-9a_\epsilon/40}$
to the decomposition \eqref{eq:align} is negligible;
the third step is to show that we can replace 
$g_{r-t}(z'-w')$ by $g_{r}(z'-\epsilon^{-1}z)$
in \eqref{eq:align}, because their difference is negligible for \eqref{L1SHE}.
These steps are proved using exactly the same analysis as in the
polymer case, see Subsection~\ref{sec:proof:ZB}.

Finally, when we restrict the integral in
\eqref{eq:align} to $r \ge (\epsilon^{-2})^{1-9a_\epsilon/40}$,
i.e.\ to $(r,z) \in B^\geq_\epsilon$ (recall \eqref{eq:Beps}),
and we replace $g_{r-t}(z'-w')$ by $g_{r}(z'-\epsilon^{-1}z)$,
the right hand side of \eqref{eq:align} becomes exactly
$u^\epsilon_A(z) \, (u^\epsilon_{B^\geq}(z)-1)$,
which proves \eqref{L1SHE}.
\end{proof}

\smallskip

\begin{proof}[Proof of Proposition~\ref{prop:4KPZ}]
In principle, also this last result could be proved
as in the polymer case, see Subsection~\ref{proof:prop:hatZA},
using a continuum analogue of Proposition~\ref{Prop8.1}. 
However, it is simpler to \emph{deduce it from
Proposition~\ref{prop:hatZA}}, approximating $u^\eps_{B^\geq}(z)$ in $L^2(\bbP)$
by a directed polymer partition function $Z_{N,\beta_N}^{B^\geq}(x)$
with $N= \epsilon^{-2}$, $x = \epsilon^{-1}z$ built
on the same probability space. The details are
described in Section~9 in \cite{CSZ17b}
(where the space-time fluctuations of $u^\eps(\cdot, \cdot)$ are shown to converge
to the solution of the additive SHE).
\end{proof}

\appendix

\section{Scaling relations for KPZ}\label{S:scaling}

We prove a scaling relation between the solutions of the mollified KPZ equations with different parameters. 
See also \cite[Section 2]{CD18}. 
In particular, we will verify the identity \eqref{eq:KPZeq} which relates the solution 
of the mollified KPZ equation with the small parameter $\beta_\eps$ either in front of the 
noise or in front of the non-linearity.

Given $\nu, \lambda, D>0$, let $\psi^\eps:= \psi^{\eps; \nu, \lambda, D}$ denote the solution of the mollified KPZ equation
\begin{equation}\label{eq:KPZgen}
\partial_t \psi^{\eps} = \frac{\nu}{2}\Delta \psi^{\eps} + \frac{\lambda}{2}  |\nabla  \psi^{\eps}|^2
+ \sqrt{D} \xi^\eps, \qquad x\in \R^2, t\geq 0, \mbox{ and } \psi^\eps(0, \cdot)\equiv 0,
\end{equation}
where $\xi^\eps(t,x)$ is the mollification of the space-time white noise $\xi$ in space with $j_\eps(x) = \eps^{-2} j(x/\eps)$, and
$j\in C_c(\R^2)$ is a probability density on $\R^2$ with $j(x)=j(-x)$.

\begin{proposition}\label{prop:scaling} Let $\psi^{\eps; \nu, \lambda, D}$ be defined as above. Then for any $a>0$, we have
\be \label{eq:KPZgenscale}
\big(\psi^{\eps; \nu, \lambda, D}(t, x)\big)_{t\geq 0, x\in\R^2} \stackrel{\rm dist}{=} \Big(\frac{\nu}{\lambda} \psi^{a\eps; 1, 1, \beta^2}(a^2 \nu t, a x)\Big)_{t\geq 0, x\in\R^2},
\ee
where $\beta^2:=\frac{\lambda^2 D}{\nu^3}$, known as the {\em effective coupling constant}, see \cite{CCDW10}.
\end{proposition}
\begin{remark}
In \eqref{eq:KPZgenscale}, setting $a=1$, $\nu=1$, $\lambda := \beta_\eps = \hbeta \sqrt{\frac{2\pi}{\log \eps^{-1}}}$ and $D=1$ gives \eqref{eq:KPZeq}, since the constant term $C_\eps$ in \eqref{eq:mollifiedKPZ} only shifts the solution deterministically.
\end{remark}

We need the following scaling relation for the mollified white noise $\xi^\eps$.
\begin{lemma}
Let $\xi$ be the space-time white noise on $\R\times \R^2$ and let $\xi^\eps := \xi* j_\eps$, where $j_\eps(x) = \eps^{-2} j(x/\eps)$. Then for any $a>0$ and $\nu>0$, we have
\be\label{dotW}
\xi^\eps (\nu a^2 \cdot, a \cdot) \stackrel{\rm dist}{=} \frac{1}{a^2\sqrt{\nu}} \xi^{\frac{\eps}{a}}(\cdot,\cdot)
\ee
in the sense that for all $\phi\in C_c^\infty(\R\times\R^2)$,
\be\label{eq:gaussian}
\int \phi(t,x) \xi^\eps (\nu a^2 t, a x) \dd t \dd x \stackrel{\rm dist}{=}  \frac{1}{a^2\sqrt{\nu}} \int \phi(t, x) \xi^{\frac{\eps}{a}}(t,x) \dd t \dd x.
\ee
\end{lemma}
\begin{proof} Since both sides of \eqref{eq:gaussian} are centered normal random variables, it suffices to check that their variances equal. Note that
$$
\begin{aligned}
X:=\int\limits_{\R\times \R^2} \phi(t,x) \xi^\eps (a^2 \nu t, a x) \dd t \dd x  & =  \int\limits_{\R\times \R^2 \times \R^2} \phi(t,x) \eps^{-2}j\Big(\frac{ax-y}{\eps}\Big) \xi(a^2\nu t, y)  \dd t \dd x \dd y \\
& = \frac{1}{a^4 \nu \eps^2} \int\limits_{\R\times \R^2} \Big(\int\limits_{\R^2}\phi\Big(\frac{\tilde t}{a^2\nu}, \frac{\tilde x}{a}\Big) j\Big(\frac{\tilde x -\tilde y}{\eps}\Big)  \dd \tilde x\Big) \, \xi(\tilde t, \tilde y) \dd \tilde t \dd \tilde y.
\end{aligned}
$$
Therefore
$$
\begin{aligned}
{\bbV}ar(X)  & = \frac{1}{a^8 \nu^2 \eps^4} \int\limits_{\R\times \R^2} \Big(\int\limits_{\R^2}\phi\Big(\frac{\tilde t}{a^2\nu}, \frac{\tilde x}{a}\Big) j\Big(\frac{\tilde x -\tilde y}{\eps}\Big)  \dd \tilde x\Big)^2  \dd \tilde t \dd \tilde y \\
& = \frac{1}{\nu \eps^4} \int\limits_{\R\times \R^2} \Big(\int\limits_{\R^2}\phi(t, x) j\Big(\frac{a(x -y)}{\eps}\Big)  \dd x\Big)^2  \dd t \dd y.
\end{aligned}
$$
On the other hand,
$$
Y:= \frac{1}{a^2 \sqrt{\nu}} \int\limits_{\R\times\R^2} \phi(t, x) \xi^{\frac{\eps}{a}}(t,x) \dd t \dd x = \frac{1}{a^2\sqrt{\nu}} \int\limits_{\R\times \R^2 \times \R^2} \phi(t,x) \frac{a^2}{\eps^2} j\Big(\frac{a(x-y)}{\eps}\Big) \xi(t, y) \dd t \dd x \dd y.
$$
Therefore
$$
{\bbV}ar(Y) =\frac{1}{\nu\eps^4} \int\limits_{\R\times \R^2} \Big( \int\limits_{\R^2} \phi(t, x) j\Big(\frac{a(x-y)}{\eps}\Big) \dd x \Big)^2 \dd t \dd y.
$$
Note that the two variances agree, so we are done.
\end{proof}
\medskip

\noindent
{\bf Proof of Proposition~\ref{prop:scaling}.} For $a,b, \tilde \eps, \beta>0$ to be chosen later, define
$$
g(t,x) := b \psi^{\tilde\eps; 1,1, \beta^2}(a^2\nu t, ax).
$$
By \eqref{eq:KPZgen}, we have
$$
\partial_t \psi^{\tilde \eps; 1, 1, \beta^2} = \frac{1}{2}\Delta \psi^{\tilde \eps; 1, 1, \beta^2} + \frac{1}{2}  |\nabla  \psi^{\tilde \eps; 1,1, \beta^2}|^2 + \beta \xi^{\tilde \eps}.
$$
Therefore
\be\label{partg}
\begin{aligned}
\frac{\partial g}{\partial t}(t,x) & = a^2 \nu b \, \frac{\partial \psi^{\tilde \eps, 1, 1, \beta^2}}{\partial t}(a^2\nu t, ax) \\
& = \frac{a^2\nu b}{2}\Delta \psi^{\tilde \eps, 1, 1, \beta^2}(a^2\nu t, ax) + \frac{a^2\nu b}{2} |\nabla \psi^{\tilde \eps, 1, 1, \beta^2}(a^2\nu t, ax)|^2 + a^2 \nu b \beta \xi^{\tilde \eps}(a^2\nu t, ax) \\
& \stackrel{\rm dist}{=} \frac{\nu}{2}\Delta g(t,x) + \frac{\nu}{2b} |\nabla g(t, x)|^2 + b \beta{\sqrt \nu} \xi^{\frac{\tilde \eps}{a}}(t,x),
\end{aligned}
\ee
where we used \eqref{dotW}.

To find $a,b, \tilde \eps$ and $\beta$ such that $g$ solves \eqref{eq:KPZgen} with parameters $\nu, \lambda, D$,  they should satisfy
\be
\eps= \frac{\tilde \eps}{a}, \quad \lambda = \frac{\nu}{b},  \quad D= b^2\beta^2\nu.
\ee
There fore we must have
\be
b= \frac{\nu}{\lambda}, \quad \tilde \eps=a\eps, \quad \beta^2 = \frac{D\lambda^2}{\nu^3},
\ee
while we are free to choose $a>0$. This proves \eqref{eq:KPZgenscale}.
\qed

\section{Hypercontractivity of polynomial chaos}
\label{sec:hyper}

We recall and refine the hypercontractivity property of polynomial chaos
established in \cite{MOO}.
Let $(\xi_i)_{i\in\bbT}$ be i.i.d.\ random variables,
labeled by a countable set $\bbT$, with
\begin{equation*}
	\bbE[\xi_i] = 0 \,, \qquad \bbE[\xi_i^2] = 1  \,.
\end{equation*}
For every $k\in\N$, let $X_k$ be a multi-linear homogeneous polynomial of degree $k$ in
the $\xi_i$'s, i.e.
\begin{equation} \label{eq:Xpoly}
	X_k = \sum_{I \subseteq \bbT: \ |I| = k}
	f_k(I)
	\prod_{i \in I} \xi_i \,,
\end{equation}
where $f_k(I)$ are real coefficients.
For $k=0$, let $X_0 = f_0 \in \R$ be a constant. Then for $k \ge 1$
\begin{equation} \label{eq:12mom}
	\bbE[X_k] = 0 \,, \qquad\quad
	\bbE[X_k^2]
	= \sum_{I \subseteq \bbT: \ |I| = k} f_k(I)^2 \,.
\end{equation}
If we assume that
\begin{equation}\label{eq:assL2}
	\sum_{k\in\N} \ \sum_{I \subseteq \bbT: \ |I| = k} f_k(I)^2 < \infty \,,
\end{equation}
then the series $X := \sum_{k=0}^\infty X_k$
is easily seen to define an $L^2$ random variable.
The next key result allows to control higher moments of $X$
in terms of second moments.

It is useful to allow the law of the $\xi_i = \xi_i^{(N)}$ to depend on a
parameter $N\in\N$.

\begin{theorem}[Hypercontractivity]\label{th:hyper+}
For $N\in\N$, let
$(\xi_i = \xi_i^{(N)})_{i\in\bbT}$ be i.i.d.\ such that
\begin{equation}\label{eq:asseta+}
	\bbE[\xi_i^{(N)}] = 0 \,, \qquad\ \bbE[(\xi_i^{(N)})^2] = 1 \,, \qquad\
	\exists p_0 \in (2,\infty): \quad
	\sup_{N\in\N} \bbE[|\xi^{(N)}_{i}|^{p_0}] < \infty \,.
\end{equation}
Then, for every $p\in (2, p_0)$, there exists a constant $c_p \in (1,\infty)$
with the following property:
for any choice of coefficients $\{f_k(I)\}_{k \in \N, \, I \subseteq \bbT,
\, |I| = k}$ satisfying \eqref{eq:assL2},
if we define $X_k$ by \eqref{eq:Xpoly}, then the $p$-th moment of the
random variable $X = \sum_{k=0}^\infty X_k$ can be bounded as
\begin{equation} \label{eq:hyper+}
	\bbE \bigg[ \bigg| \sum_{k=0}^\infty X_k \bigg|^p \bigg]
	\le
	\bigg( \sum_{k=0}^\infty (c_p^k)^2 \, \bbE[ X_k^2 ] \bigg)^{p/2} \,,
\end{equation}
with $\bbE[X_k^2]$ given in \eqref{eq:12mom}.
The constant $c_p$ only depends on the laws of $(\xi_i^{(N)})$ and satisfies
\begin{equation}\label{eq:cp1+}
	\lim_{p\downarrow 2} c_p = 1 \,.
\end{equation}
\end{theorem}

Except for relation~\eqref{eq:cp1+},
which we prove below, this theorem was proved in \cite{MOO} as an
extension of the corresponding result in the Gaussian framework, see~\cite{J97}. In fact,
\cite[Proposition~3.16]{MOO} gave the following explicit bound on $c_p$:
\begin{equation*}
	c_p \le \tilde c_p = 2 \sqrt{p-1} \, \sup_{N\in\N} \,
	\frac{\bbE[|\xi_i^{(N)}|^p]^{1/p}}{\bbE[|\xi_i^{(N)}|^2]^{1/2}}
	= 2 \sqrt{p-1} \, \sup_{N\in\N} \,
	\bbE[|\xi_i^{(N)}|^p]^{1/p} \,,
\end{equation*}
and note that $\lim_{p\downarrow 2} \tilde c_p = 2$. This extra factor
$2$ is the byproduct of a non-optimal symmetrization argument
in the proof in \cite{MOO}. We now prove \eqref{eq:cp1+}.

\begin{proof}[Proof of equation \eqref{eq:cp1+}]
By \cite[Section 3.2]{MOO},
relation \eqref{eq:hyper+} holds with constant $c_p$
if the law of the random vairable $\xi = \xi_i$ in \eqref{eq:Xpoly} is
\emph{$(2,p,1/c_p)$-hypercontractive}, that is
\begin{equation*}
	\forall a \in \R: \qquad
	\| a + \tfrac{1}{c_p} \xi \|_p \le
	\| a + \xi \|_2 \,,
\end{equation*}
where $\|\cdot\|_p := \bbE[|\cdot|^p]^{1/p}$ denotes the $L^p$ norm.
Since we allow the law of $\xi = \xi^{(N)}$ to depend on $N\in\N$, it follows that
we can characterize $c_p$ as follows:
\begin{equation}\label{eq:optimalcp}
	c_p = \inf\big\{c > 1: \
	\| a + \tfrac{1}{c} \xi^{(N)} \|_p \le
	\| a + \xi^{(N)} \|_2 \ \ \forall
	a\in\R\,, \ \forall N \in \N \big\} \,.
\end{equation}
For simplicity, we split the proof in two steps.

\medskip
\noindent
\textit{Step 1.}
We first consider the case of a fixed law
for the random variable $\xi$ (independent of $N\in\N$)
satisfying \eqref{eq:asseta+}.
In view of \eqref{eq:optimalcp}, we can rephrase our goal
$\lim_{p\downarrow 2} c_p = 1$ as follows:
\begin{equation}\label{eq:goa1}
	\forall c > 1 \ \exists p > 2 : \quad
	\| a + \tfrac{1}{c} \, \xi \|_p \le
	\| a + \xi \|_2 \quad \forall a \in \R \,.
\end{equation}
We will first show that given $c>1$, we can find $\bar p = \bar p_{c, p_0}>2$ and 
$K = K_{c,p_0}>0$, such that for all $p\in (2, \bar p]$ and $|a|> K$, the inequality in \eqref{eq:goa1} 
holds. We will 
then find $p\in (2, \bar p]$ such that the inequality in \eqref{eq:goa1} also holds for all
$|a|\leq K$. 

\smallskip

We first need an elementary estimate:
for any $p_0 \in (2,\infty)$
there exists $C = C_{p_0} < \infty$ such that,
for all $p \in [2,p_0]$ and $x\in\R$,
\begin{equation} \label{eq:rem}
	|1+ x|^p = 1 + p x + \tfrac{p(p-1)}{2} x^2
	+ R(x) \,, \qquad \text{with} \qquad
	|R(x)| \le C \big( |x|^3 \wedge |x|^{p_0} \big)\,.
\end{equation}
This follows by Taylor's formula for $|x| \le \frac{1}{2}$ (say)
and by the triangle inequality for $|x| > \frac{1}{2}$.

We may assume that $p_0 \in (2,3]$ in \eqref{eq:asseta+}
(just replace $p_0$ by $p_0 \wedge 3$). Then for every $\delta \in \R$
with $|\delta| \le 1$ we can bound
\begin{equation*}
	|R(\delta \xi)| \le C \big(|\xi|^3 \wedge |\xi|^{p_0}\big)
	|\delta|^{p_0} \le C \big(1 + |\xi|^{p_0}\big)
	|\delta|^{p_0} \,.
\end{equation*}
Since $\bbE[\xi] = 0$, it follows by
\eqref{eq:asseta+} and \eqref{eq:rem} that for every $\delta \in\R$ with $|\delta| \le 1$
\begin{equation} \label{eq:boundr}
\begin{split}
	& \bbE[ |1 + \delta \xi|^p ] = 1 + \tfrac{p(p-1)}{2} \, \delta^2 \, \bbE[\xi^2] +
	r(\delta) \qquad \text{with} \qquad
	| r(\delta) | \le C' \delta^{p_0} \,, \\
	& \text{where} \qquad
	C' = C'_{p_0} := C \big(1 + \bbE[|\xi|^{p_0}]\big) \,.
\end{split}
\end{equation}
Then, as $|\delta| \to 0$,
\begin{equation*}
	\| 1 + \delta \xi \|_p = 1 + \tfrac{p-1}{2} \, \delta^2 \, \bbE[\xi^2] +
	O(|\delta|^{p_0}) \,,
\end{equation*}
\emph{uniformly for $p \in [2,p_0]$.} This implies that as $|a| \to \infty$
\begin{equation} \label{eq:theratio}
\begin{split}
	\frac{\| a + \tfrac{1}{c} \, \xi \|_p}{\| a + \xi \|_2}
	& = \frac{\| 1 + \frac{1}{ca} \, \xi \|_p}{\| 1 +
	\frac{1}{a} \xi \|_2}
	= \frac{1 + \frac{(p-1) \bbE[\xi^2]}{2c^2|a|^2}
	+ O(\frac{1}{|a|^{p_0}})}{1 + \frac{\bbE[\xi^2]}{2|a|^2}
	+ O(\frac{1}{|a|^{p_0}})} \\
	& = 1 + \big\{ \tfrac{p-1}{c^2}  - 1 \big\} \, \frac{\bbE[\xi^2]}{2|a|^2}
	+ O(\tfrac{1}{|a|^{p_0}}) \,.
\end{split}
\end{equation}
Given $c>1$, we can take $\bar p = \bar p_{c,p_0} :=
\min\{1 + c, p_0\} > 2$ so that \emph{uniformly in $p\in (2, \bar p]$}, 
the term in bracket is bounded by $c^{-1}-1<0$.
Then the RHS of \eqref{eq:theratio} is $< 1$
for large $|a|$, say for $|a| > K$, where $K = K_{p_0,c} < \infty$ only depends
on $c$ and $p_0$. This proves the inequality in \eqref{eq:goa1} for all $p\in (2, \bar p]$ and $|a| > K$.

\smallskip

To complete the proof, we now show that there exists $p \in (2, \bar p]$ such that the inequality in \eqref{eq:goa1} holds for $|a| \le K$.
If this is false, then for any sequence $p_n \in (2, p_0]$ with $p_n \downarrow 2$, we can find $a_n \in [-K,K]$ such that
\begin{equation} \label{eq:limitof}
	\| a_n + \tfrac{1}{c} \, \xi \|_{p_n} >
	\| a_n + \xi \|_2 \quad \forall n \in \N \,.
\end{equation}
Extracting subsequences, we may assume that $a_n \to a \in [-K, K]$.
Since $f(p,a) := \| a + \tfrac{1}{c}\xi \|_p$
is a continuous function of $(p,a) \in [2,p_0] \times [-K,K]$
(by dominated convergence), we may take the limit of \eqref{eq:limitof}
as $n\to\infty$ and get
\begin{equation} \label{eq:andget}
	\| a + \tfrac{1}{c} \, \xi \|_{2} \ge
	\| a + \xi \|_2  \,,
\end{equation}
which is a contradiction, since $\| a + \tfrac{1}{c}\xi \|_2
= \sqrt{a^2 + \tfrac{1}{c^2} \bbE[\xi^2]} < \| a + \xi \|_2$
(recall that $c>1$).

\medskip
\noindent
\textit{Step 2.}
Next we allow the law of $\xi = \xi^{(N)}$ to depend on $N\in\N$.
In view of \eqref{eq:optimalcp}, our goal $\lim_{p\downarrow 2} c_p = 1$ can
be rephrased as follows:
\begin{equation}\label{eq:goa2}
	\forall c>1 \ \exists p > 2 : \quad
	\| a + \tfrac{1}{c} \, \xi^{(N)} \|_p \le
	\| a + \xi^{(N)} \|_2 \quad \forall a \in \R \,, \ \forall N \in \N \,.
\end{equation}
We follow the same proof as in Step 1. We just need to check
the uniformity in $N\in\N$.

Relation \eqref{eq:boundr} still holds with $\xi$ replaced by $\xi^{(N)}$,
where we stress that $C' = C'_{p_0} < \infty$ because we assume that
$\sup_{N\in\N} \bbE[|\xi^{(N)}|^{p_0}] < \infty$, see \eqref{eq:asseta+}.
Then \eqref{eq:theratio} holds as $|a| \to \infty$,
uniformly for $p \in [2,p_0]$ \emph{and also for $N\in\N$}.
This proves that \eqref{eq:goa2} holds if we restrict $|a| \le K$,
for a suitable $K = K_{p_0,c}$ depending only on $c > 1$ and $p_0$.

It remains to fix $c > 1$, $K<\infty$
and prove that \eqref{eq:goa2} holds, for some $p > 2$ and for every $|a| \le K$.
Arguing again by contradiction, assume now that there
are sequences $p_n \in (2, p_0]$, $a_n \in [-K,K]$, $N_n \in \N$, with $p_n \downarrow 2$,
such that
\begin{equation} \label{eq:limitof2}
	\| a_n + \tfrac{1}{c} \, \xi^{(N_n)} \|_{p_n} >
	\| a_n + \xi^{(N_n)} \|_2 \quad \forall n \in \N \,.
\end{equation}
Extracting subsequences, we may assume that $a_n \to a \in [-K, K]$,
and also that $\xi^{(N_n)}$ converges in law to a random variable $\xi$
(the sequence is tight, by \eqref{eq:asseta+}).
Since $|a_N + \tfrac{1}{c} \xi^{(N_n)}|^{p_n}$ are uniformly integrable,
again by \eqref{eq:asseta+}, we can take the limit of relation \eqref{eq:limitof2}
and we get precisely \eqref{eq:andget}, which we already showed to be a contradiction.
\end{proof}

\section{Gaussian concentration in the continuum}
\label{sec:concGauss}

We prove a Gaussian concentration result,
based on \cite{Led96,Led}, which can be viewed as a ``one-sided
version'' of \cite[Theorem~2]{FO} (cf.\ \eqref{eq:conclude} below
with eq.~(4) in~\cite{FO}).

Given a probability measure $\mu$ on a measurable space $(E, \cE)$, we denote by
$\mu_*$ and $\mu^*$ the inner and outer measures:
$\mu_*(A) := \sup\{\mu(A'): \, A' \subseteq B, \,
A' \in \cE\}$ and $\mu^*(A) = 1 - \mu_*(A^c)$.

\begin{theorem}\label{th:conclow}
Let $\mu$ be a Gaussian measure on a separable Banach space $E$,
with Cameron-Martin space $H$.\footnote{This means that
$H$ is a separable Hilbert space,
continuously
embedded as a dense subset of
the separable Banach space $E$, and $\mu$ is a probability
on $E$ that can be described as follows: given any complete orthonormal
set $(h_n)_{n\in\N}$ in $H$ and
given i.i.d.\ $N(0,1)$ random variables $(Z_n)_{n\in\N}$, the
sequence of random elements $X_N := \sum_{n=1}^N Z_n \, h_n$ converges a.s.\
\emph{in the space $E$}, and $\mu$ is the law on $E$ of the
limit $X := \sum_{n\in\N} Z_n \, h_n$.
The triple $(H, E, \mu)$ is called an
\emph{abstract Wiener space}. We refer to \cite{Led96} for more details.}
Let $f: E \to \R \cup \{-\infty,+\infty\}$ be convex. For $x \in E$
with $|f(x)| < \infty$,
define the maximal gradient $|\nabla f(x)| \in [0,\infty]$
in the directions of $H$ by \eqref{eq:gradf}.
Then
\begin{equation} \label{eq:conclow}
	\mu(f \le a - t) \, \mu^*(f \ge a, |\nabla f| \le c)
	\le e^{-\frac{1}{4} (t/c)^2}  \qquad
	\forall a \in \R \,, \ \forall t, c \in (0,\infty) \,.
\end{equation}
\end{theorem}

\noindent
(The outer measure $\mu^*$ appears in \eqref{eq:conclow}
to avoid the issue of measurability of $|\nabla f|$.)

\smallskip

Let us denote by $\cK := \{h \in H: \ \|h\|_{H} \le 1\}$ the unit ball in
the Cameron-Martin space $H$.
Given a subset $A \subseteq E$, we define
its enlargement $A + r\cK := \{ x + r h: \ x \in A, \ h \in \cK\}$.
We recall the classical concentration property
established by Borell \cite[Theorem~4.3]{Led96}:
\begin{equation}\label{eq:concset}
\begin{split}
	&\forall A \subseteq E \text{ with } 0 < \mu(A) < 1, \ \text{ setting } a := \Phi^{-1}(\mu(A)) \,, \\
	& \qquad
	\mu_* (A + r\cK) \ge \Phi(a+r) \qquad \forall r \ge 0 \,,
\end{split}
\end{equation}
where $\Phi(x) = \int_{-\infty}^x \frac{1}{\sqrt{2\pi}} e^{-t^2/2} \, \dd t$ is
the standard Gaussian distribution function.

\smallskip

The proof of Theorem~\ref{th:conclow} is based on the following Lemma
of independent interest, which follows from \eqref{eq:concset}.
It is close to \cite[Corollary~1.4]{Led} (see also \cite[Appendix~B.1]{CTT}).

\begin{lemma}\label{th:concset2}
For any measurable subset $A \subseteq E$, the following inequality holds:
\begin{equation}\label{eq:concset2}
\begin{split}
	\mu(A) \, \big(1 - \mu_*( A + r\cK ) \big) \le
	e^{-\frac{1}{4} r^2} \qquad \forall r \ge 0 \,.
\end{split}
\end{equation}
\end{lemma}

\begin{proof}
We may assume $0 < \mu(A) < 1$ (otherwise \eqref{eq:concset2} holds trivially)
and we apply \eqref{eq:concset}:
\begin{equation} \label{eq:1ststep}
	1 - \mu_*( A + r\cK ) \le
	1 - \Phi(r+a) \le e^{-\frac{1}{2} ((r+a)^+)^2} \qquad \forall r \ge 0 \,,
\end{equation}
where $x^+ := \max\{x,0\}$ and we used the basic bound
$1-\Phi(x) \le e^{-x^2/2}$ for $x \ge 0$.

Consider first the case $\mu(A) \ge \frac{1}{2}$: then $a = \Phi^{-1}(\mu(A)) \ge 0$ and
$(r+a)^+ \ge r$, so \eqref{eq:concset2} follows by \eqref{eq:1ststep}
(just bound $\mu(A) \le 1$). Henceforth we take $\mu(A) < \frac{1}{2}$, so $a < 0$.
Note that
\begin{equation} \label{eq:2ndstep}
	\mu(A) = \Phi(a) = 1-\Phi(|a|) \le e^{-\frac{1}{2} |a|^2} \,.
\end{equation}
Fix $r \ge 0$. If $|a| \ge r$, then \eqref{eq:concset2} follows  by \eqref{eq:2ndstep}
(just bound $1 - \mu_*( A + r\cK ) \le 1$).
If $|a| < r$, then $(r+a)^+ = (r-|a|)^+ = r - |a|$ and
relations \eqref{eq:1ststep}-\eqref{eq:2ndstep} yield
\begin{equation*}
	\mu(A) \, \big( 1 - \mu_*( A + r\cK ) \big) \le
	e^{-\frac{1}{2}\{|a|^2 + (r-|a|)^2\}}
	\le \sup_{x\in\R} \, e^{-\frac{1}{2}\{x^2 + (r-x)^2\}} =
	e^{-\frac{1}{4} r^2} \,.\qedhere
\end{equation*}
\end{proof}

\begin{proof}[Proof of Theorem~\ref{th:conclow}]

Fix $x, x' \in E$ such that $h := x'-x \in H$.
The function $g: [0,1] \to \R$ defined by
$g(s) := f( (1-s)x + s x') = f(x + sh)$ is convex
(since $f$ is convex), hence
\begin{equation*}
	f(x') - f(x) = g(1) - g(0) \le g'(1-) := \lim_{\epsilon \downarrow 0} \frac{g(1) - g(1-\epsilon)}{\epsilon}
	= \lim_{\epsilon \downarrow 0}
	\frac{f(x') - f(x' - \epsilon h)}{\epsilon} \,.
\end{equation*}
Recalling \eqref{eq:gradf}, we have shown that
\begin{equation} \label{eq:conclude}
	f(x') - f(x) \le |\nabla f(x')| \, \|x'-x\|_{H} \,.
\end{equation}

Let us now set
\begin{equation*}
	A := \{f \le a-t\} \,, \qquad B := \{f \ge a, \, |\nabla f| \le c\} \,.
\end{equation*}
In view of Lemma~\ref{th:concset2}, to prove \eqref{eq:conclow}
it suffices to show that for any $r < \frac{t}{c}$ we have
$B \subseteq (A + r\cK)^c$,
i.e.\ $A + r\cK \subseteq B^c$.
So we fix $x\in A$, $h\in H$ with $\|h\|_{H} < \frac{t}{c}$
and we show that $x' := x + h \in B^c$.
Either $f(x') < a$, and then $x' \not\in B$,
or $f(x') \ge a$, and then (by $x\in A$)
\begin{equation*}
	|\nabla f(x')| \ge \frac{f(x')-f(x)}{\|x' - x\|_{H}}
	\ge \frac{a - (a-t)}{\|h\|_{H}} > \frac{t}{t/c} = c \,,
\end{equation*}
hence again $x' \not\in B$. This completes the proof.
\end{proof}

\section{Stochastic integral as a linear function}
\label{sec:stocversion}

We formulate a linearity result for the stochastic integral
with respect to the white noise $\xi = (\xi(z))_{z\in\R^d}$ on $\R^d$, which is needed in the proof of Proposition~\ref{neg-mom-polyKPZ}. Recall that the
white noise can be realized as a random element of a separable Banach space $E$ of distributions
on $\R^d$ (e.g.\ a negative H\"older space),
equipped with its Borel
$\sigma$-algebra.
Denoting by $\mu$ the law of the white noise on $E$, we
will use the probability space $(E,\mu)$ as a canonical construction of $\xi$.
We also set $H = L^2(\R^d)$.

For any $h \in H$, the stochastic integral $\langle h, \xi \rangle
:= \int_{\R^d} h(z) \, \xi(z) \, \dd z \sim N(0, \|h\|_H^2)$ is a random variable
in $L^2(E,\mu)$, so it is not canonically defined for any given $\xi \in E$.
The following results guarantees the existence of a convenient version of $\langle h, \xi \rangle$.

\begin{theorem}\label{th:linearity}
It is possible to define $\langle h, \xi \rangle$ as a jointly measurable of
$(h,\xi) \in H \times E$, with the following properties.
\begin{itemize}
\item $\langle h, \xi \rangle$ is a version of the stochastic integral
$\int_{\R^d} h(z) \, \xi(z) \, \dd z$, for every $h\in H$.

\item For any probability measure $\nu$ on $H$, there is a measurable vector
space $V_\nu \subseteq E$ with
\begin{equation*}
	\mu(V_\nu) = 1\,, \qquad V_\nu + H = V_\nu \,,
\end{equation*}
such that the following property holds:
\begin{equation} \label{eq:property2}
	\forall \xi, \xi' \in V_\nu: \quad
	\langle h, \alpha\xi + \alpha'\xi' \rangle 
	= \alpha \langle h, \xi \rangle + \alpha' \langle h, \xi' \rangle < \infty
	\quad \text{for $\nu$-a.e.\ $h\in H$}, \,
	\forall \alpha, \alpha' \in \R \,.
\end{equation}
\end{itemize}
\end{theorem}

\begin{remark}
Given any
probability $\nu$ on $H$, we can define $f: E \to \R \cup \{+\infty\}$ by
\begin{equation} \label{eq:acca}
	f(\xi) := \log \int_H e^{\langle h, \xi \rangle} \, \nu(\dd h) \,.
\end{equation}
This function is convex when
restricted to the vector space $V_\nu$ of Theorem~\ref{th:linearity},
by \eqref{eq:property2} and H\"older's inequality.
If we redefine $f(\xi) := +\infty$ for $\xi \not\in V_\nu$,
we obtain a version of $f$ (recall that $\mu(V_\nu) = 1$)
which is convex on the whole space $E$.

This applies, in particular, to the function
$h_\epsilon(\xi) := \log u_{\Lambda}^\epsilon$
in the proof of Proposition~\ref{neg-mom-polyKPZ},
see \eqref{eq:hu}, with $u_\Lambda^\epsilon = u_{\Lambda,\xi}^\epsilon(0)$ defined
in \eqref{eq:uepsLambda}. In this case
$\R^d = \R^{1+2}$ and $\nu$ is the law of the process $(\beta_\epsilon
\, j(B_s-x))_{(s, x) \in [0,\epsilon^{-2}] \times \R^2} \in L^2(\R^{1+2})$.
\end{remark}

\begin{proof}[Proof of Theorem~\ref{th:linearity}]
Fix a probability density
$\rho \in C^\infty_c(\R^d)$ and set $\rho_\epsilon(z) :=
\epsilon^{-d} \rho(\epsilon^{-1} z)$.
Also fix a smooth cutoff function $\chi :\R^d \to [0,1]$
with $\chi(x) \equiv 1$ for $|x| \le 1$
and $\chi(x) \equiv 0$ for $|x| \ge 2$,
and set $\chi_\epsilon(z) := \chi(\epsilon z)$. For any $h\in H= L^2(\R^d)$, we define
$h_\epsilon \in C^\infty_c(\R^d)$ by
\begin{equation*}
	h_\epsilon(z) := \chi_\epsilon(z) \, (\rho_\epsilon * h)(z) \,.
\end{equation*}
Since $\lim_{\epsilon \downarrow 0}
\|h_\epsilon - h\|_H = 0$, we can find
$(\epsilon_n = \epsilon_n^h)_{n\in\N}$ such that
$\|h_{\epsilon_n} - h\|_H \le 2^{-n}$.
(We can ensure that $\epsilon_n^h$ is measurable in $h$,
e.g.\ $\epsilon_n^h := \max\{\epsilon \in \{\frac{1}{k}: \, k\in\N\}: \ 
\|h_{\epsilon} - h\|_H \le 2^{-n}\}$.)

For every $n\in\N$ we have $h_{\epsilon_n} \in C^\infty_c(\R^d)$,
hence the map
\begin{equation}\label{eq:cano}
	(h,\xi) \mapsto \langle h, \xi \rangle_n := \langle h_{\epsilon_n}, \xi \rangle
\end{equation}
is canonically defined for any distribution $\xi\in E$,
and is jointly measurable in $(h,\xi) \in H \times E$.
By the Ito isometry of the stochastic integral and Borel-Cantelli,
for any fixed $h\in H$ we have $\lim_{n\to\infty} \langle h, \xi \rangle_n =
\langle h, \xi \rangle$ for $\mu$-a.e.\ $\xi\in E$.
We can finally define the measurable map
\begin{equation*}
	\langle h, \xi \rangle := 
	\begin{cases}
	\lim_{n\to\infty} \langle h, \xi \rangle_n
	& \text{if the limit exists in $\R$} \\
	+\infty & \text{otherwise}
	\end{cases} \,.
\end{equation*}
For every $n\in\N$ the maps $\xi \mapsto \langle h, \xi \rangle_n$ are linear, 
hence the limit map $\xi \mapsto \langle h, \xi \rangle$ is linear too
whenever it is finite. More precisely, for every $h\in H$ and $\xi,\zeta \in E$:
\begin{equation}\label{eq:property}
	\langle h,\xi \rangle < \infty \,, \
	\langle h,\zeta \rangle < \infty \quad \Longrightarrow \quad 
	\langle h, \alpha\xi + \beta\zeta \rangle 
	= \alpha \langle h, \xi \rangle + \beta \langle h, \zeta \rangle < \infty
	\ \ \ \forall \alpha,\beta\in\R \,.
\end{equation}

By construction, for every $h\in H$ we have 
$\langle h,\xi \rangle \in L^2(E,\mu)$, so
$\langle h,\xi \rangle < \infty$ for $\mu$-a.e.\ $\xi\in E$.
If we now fix a probability $\nu$ on $H$, 
and we define the measurable subset $V_\nu \subseteq E$ by
\begin{equation*}
	V_\nu := \{\xi \in E: \ \langle h,\xi \rangle < \infty
	\text{ for $\nu$-a.e.\ $h\in H$}\} \,,
\end{equation*}
it follows by Fubini's theorem that $\mu(V_\nu) = 1$.
Note that $V_\nu + H = V_\nu$, because
$\langle h, g \rangle < \infty$ for all $h,g\in H$.
Finally, relation \eqref{eq:property} implies \eqref{eq:property2},
which shows that $V_\nu$ is a vector space.
\end{proof}

\section*{Acknowledgements}
F.C. is supported by the PRIN Grant  20155PAWZB ``Large Scale Random Structures''.
R.S. is supported by NUS grant R-146-000-253-114. N.Z.
is supported by EPRSC through grant EP/R024456/1. We thank MFO, the Oberwolfach 
Research Institute for Mathematics,
for hosting the workshop {\em Scaling Limits in Models of Statistical Mechanics}
in September 2018,
during which part of this work was carried out.

\end{document}